\documentclass[a4paper]{amsart}

\usepackage[T1]{fontenc}
\usepackage{lmodern}
\usepackage[utf8]{inputenc}

\usepackage[hidelinks]{hyperref}
\hypersetup{
colorlinks=true,
allcolors=blue
}

\usepackage[noadjust,nocompress]{cite}

\usepackage{amsmath}
\usepackage{amssymb}
\usepackage{tikz-cd}
\usepackage{enumitem}
\usepackage{todonotes}
\usepackage{stmaryrd}
\DeclareSymbolFont{extraup}{U}{zavm}{m}{n}
\DeclareMathSymbol{\varheart}{\mathalpha}{extraup}{86}
\DeclareMathSymbol{\vardiamond}{\mathalpha}{extraup}{87}

\usepackage{ dsfont }

\newcommand{\axiom}[1]{\mathsf{#1}}
\newcommand{\ZFC}{\axiom{ZFC}}

\newcommand{\ZF}{\axiom{ZF}}
\newcommand{\AC}{\axiom{AC}}
\newcommand{\HS}{\axiom{HS}}
\newcommand{\HR}{\axiom{HR}}
\newcommand{\SN}{\axiom{N}}
\newcommand{\DC}{\axiom{DC}}
\newcommand{\C}{\axiom{C}}
\newcommand{\R}{\axiom{R}}
\newcommand{\HC}{\axiom{HC}}

\newcommand{\NBG}{\axiom{NBG}}
\newcommand{\KW}{\axiom{KWP}}
\newcommand{\KWP}{\axiom{KWP}}

\newcommand{\sG}{\mathcal{G}}
\newcommand{\sF}{\mathcal{F}}
\newcommand{\sS}{\mathcal{S}}
\newcommand{\sR}{\mathcal{R}}
\newcommand{\sT}{\mathcal{T}}
\newcommand{\sH}{\mathcal{H}}
\newcommand{\sE}{\mathcal{E}}
\newcommand{\sC}{\mathcal{C}}
\newcommand{\sO}{\mathcal{O}}

\newcommand{\B}{\mathcal{B}}

\DeclareMathOperator{\sym}{sym}
\DeclareMathOperator{\fix}{fix}
\DeclareMathOperator{\Add}{Add}

\DeclareMathOperator{\Aut}{Aut}
\DeclareMathOperator{\id}{id}
\DeclareMathOperator{\dom}{dom}
\DeclareMathOperator{\Ord}{Ord}
\DeclareMathOperator{\Coll}{Coll}
\DeclareMathOperator{\ON}{Ord}
\DeclareMathOperator{\rk}{rk}
\DeclareMathOperator{\trcl}{trcl}

\DeclareMathOperator{\res}{res}
\DeclareMathOperator{\cl}{cl}

\DeclareMathOperator{\ran}{ran}
\DeclareMathOperator{\OD}{OD}
\DeclareMathOperator{\HOD}{HOD}

\DeclareMathOperator{\supp}{supp}

\theoremstyle{plain}
\newtheorem{thm}{Theorem}[section]
\newtheorem{lemma}[thm]{Lemma}
\newtheorem{prop}[thm]{Proposition}
\newtheorem{cor}[thm]{Corollary}
\newtheorem{claim}[thm]{Claim}

\newtheorem*{remark}{Remark}
\theoremstyle{definition}
\newtheorem{definition}[thm]{Definition}

\newtheorem{exmp}[thm]{Example}
\newtheorem{quest}[thm]{Question}

\author{Asaf Karagila}
\author{Jonathan Schilhan}
\email{karagila@math.huji.ac.il}
\urladdr{https://karagila.org}
\address{School of Mathematics,
   University of Leeds.
   Leeds, LS2~9JT, UK}
\email{jonathan.schilhan@univie.ac.at}
\urladdr{https://mat.univie.ac.at/~jschilhan/}
\address{University of Vienna,
Institute of Mathematics,
Kurt Gödel Research Center,
Kolingasse 14-16,
1090 Vienna,
Austria}

\date{\today}
\title{Towards a theory of symmetric extensions}

\begin{document}

\begin{abstract}
The technique of \emph{symmetric extensions} is derived from forcing and it is one of the most important tools for studying models without the Axiom of Choice. Despite being incredibly successful since the 1960s, our understanding of the technique remained fairly limited compared to the theory of forcing. Whereas forcing developed products and iterations, no serious attempts at developing any general framework for iterating symmetric extensions were presented before \cite{Karagila2019}, where only finite support iterations are treated. In this paper we develop the theory of symmetric extensions including different types of iterations, quotients, equivalents, and the structural results that can be described in this language. In particular, we give a modern exposition to some of the important theorems of Grigorieff \cite{Grigorieff1975}, study Kinna--Wagner Principles in symmetric extensions, and show that it is provable from $\ZF$ that every set lies in a symmetric extension of $\HOD$.
\end{abstract}

\maketitle
\setcounter{tocdepth}{1}
\tableofcontents
\vfil
\subsection*{Acknowledgements}
The authors were partially supported by a UKRI Future Leaders Fellowship [MR/T021705/2]. The first author was supported by UKRI Future Leaders Fellowship [MR/Y034058/1]. This research was funded in whole or in part by the Austrian Science Fund (FWF) [10.55776/ESP5711024]. For open access purposes, the authors have applied a CC BY public copyright license to any author accepted manuscript version arising from this submission. The authors would like to express special gratitude to Peter Holy for various discussions and comments regarding earlier manuscripts of this paper.
\newpage
\section{Introduction}

Iteration is a common methodology in mathematics. It allows us to solve a large problem by breaking it into many small problems, and then solve one problem at a time. Set theory is not different, and shortly after Cohen published his two papers on forcing \cite{Cohen:1963a,Cohen:1963b} the first instances of iterated forcing were discovered as well. Most famously, Solovay and Tennenbaum \cite{SolovayTennenbaum:1971} proved the consistency of the Suslin Hypothesis using iterated forcing, and the technique had been extended significantly in the decades since its conception.

Forcing is very useful for the study of models of $\ZFC$, as it preserves the Axiom of Choice,\footnote{That is, starting with a model of $\ZFC$, any forcing extension will be a model of $\ZFC$.} which is certainly a desirable feature, but it makes studying choiceless universes more difficult. Cohen, however, already noticed in his initial work that one can pass to an inner model of the extension, an \emph{intermediate model}, where the Axiom of Choice fails. This eventually gave rise to the technique of \emph{symmetric extensions} where in addition to the forcing notion we are given an automorphism group and a filter of subgroups, which then allow us to identify in advance an intermediate model of $\ZF$ and use forcing-theoretic tools to study it. This concept was studied more abstractly by Grigorieff \cite{Grigorieff1975}, where several structural results were obtained about symmetric extensions.

Iterating symmetric extensions was something that was present in the mathematical literature for decades, but until \cite{Karagila2019} this was always in ad-hoc form and never in a generalized framework. Incredible examples include the works in \cite{Sageev,Monro:KW,Morris:PhD}. 

A first attempt to design a general framework for iterating symmetric extension was published in \cite{Karagila2019} by the first author of this paper. The presentation of the framework is certainly difficult, as its use is still limited. Moreover, the framework only allows finite support iterations which makes it harder to preserve certain choice principles, such as Dependent Choice which naturally lends itself to be preserved by countably closed extensions.

In this paper we extend and rework the basic theory of symmetric extensions and the iteration framework. Indeed, we aim to streamline the definitions and notation from \cite{Karagila2019} to make the application of the method much simpler. The goal of this manuscript is to make symmetric extensions and their extensions more accessible.

We develop the theory of symmetric extensions in parallel to the theory of forcing. We study the basics of symmetric extensions, iterations and products, as well as an intermediate notion of reduced products.\footnote{Reduced products correspond to ``productive iterations'' in \cite{Karagila2019}.} These lead naturally to completions, embeddings, and quotients of symmetric systems. 

All these ideals provide us with a new and more accessible language to use and present refreshed and generalized theorems from Grigorieff's original paper. In turn, this allows us to prove that starting from a model of $\ZFC$, every symmetric extension has an associated cardinal $\kappa$, such that $\KWP^*_{\kappa^+}$ must hold in the symmetric extension (see Theorem~\ref{thm:kw}). We also this new language to prove that every set lies in a symmetric extension of $\HOD$ (see Theorem~\ref{thm:symoverHOD}).

\section{Symmetric systems and extensions}\label{sec:prelim}
In this section we will reintroduce symmetric systems and symmetric extensions using new notation and terminology which we think is more streamlined and we expect will be more conducive for future work. We follow the standard treatment of forcing, wherein a notion of forcing is a preordered set $\mathbb{P}$ with a maximum element, often denoted by $\mathds{1}_{\mathbb{P}}$. The elements of $\mathbb{P}$ are called conditions and we write $q\leq p$ to mean that $q$ is a stronger condition than $p$, or that $q$ extends $p$. Two conditions are compatible if they have a common extension, and otherwise they are incompatible.

The class of $\mathbb{P}$-names, $V^{\mathbb{P}}$ is defined recursively by $V^{\mathbb{P}}_\alpha=\bigcup\{\mathcal P(\mathbb{P}\times V_\beta^{\mathbb{P}}):\beta<\alpha\}$, with $V^\mathbb{P}=\bigcup_{\alpha\in\Ord}V^\mathbb{P}_\alpha$. Given a set of $\mathbb{P}$-names, $X$, we write $X^\bullet$ to denote the ``obvious'' name it generates. Namely, \[X^\bullet = \{(\mathds 1,\dot x) : \dot x\in X\}.\]
We extend this notation to ordered pairs, as well as sequences and functions whose domains are $\bullet$-names. Using this notation we can define the canonical ground model names, $\check x=\{\check y: y\in x\}^\bullet$.
\subsection{Symmetric systems}

\begin{definition}
    Let $\sG$ be a group. Then a set $\sF$ of subgroups of $\sG$ is a \emph{filter} if 
    
    \begin{enumerate}
        \item $\forall H_0 \leq H_1 \leq \sG ( H_0 \in \sF \rightarrow H_1 \in \sF)$,
        \item $\forall H_0, H_1 \in \sF (H_0 \cap H_1 \in \sF)$. 
    \end{enumerate}$\sF$ is \emph{normal} if for every $\pi \in \sG$, $H \in \sF$, also $\pi H \pi^{-1} \in \sF$.
\end{definition}

\begin{definition}\label{def:symsys}
    A \emph{symmetric system} is a triple $\sS = (\mathbb{P}, \sG, \sF)$ where $\mathbb{P}$ is a forcing notion, $\sG$ is a group of automorphisms on $\mathbb{P}$ and $\sF$ is a normal filter of subgroups of $\sG$.
\end{definition}

In many situations, we may in fact allow $\sG$ not to literally consist of automorphisms on $\mathbb{P}$. Rather sometimes there will merely be an implied action of $\sG$ on $\mathbb{P}$. This is a purely technical distinction that does not affect any of the arguments or the practice of symmetric extensions. In particular, all the notions the we define below can be derived in an analogous way for this more general situation. The official definition remains Definition~\ref{def:symsys}.

Whenever $\pi$ is an automorphism of $\mathbb{P}$, $\pi$ naturally extends to $\mathbb{P}$-names by a recursive definition: \[\pi(\dot x) = \{ (\pi(p), \pi(\dot y)) : (p, \dot y) \in \dot x \}.\] This action is well-behaved with respect to the forcing relation. In other words, if $p\Vdash\varphi(\dot x)$, then $\pi(p)\Vdash\varphi(\pi(\dot x))$.

For the rest of this section, fix a symmetric system $\sS = (\mathbb{P}, \sG, \sF)$.

\begin{definition}
    A $\mathbb{P}$-name $\dot x$ is $\sS$-symmetric if $\sym_\sS(\dot x) \in \sF$, where \[\sym_\sS(\dot x) := \{ \pi \in \sG : \pi(\dot x) = \dot x \}.\]
\end{definition}

Since $\sym_\sS(\dot x)$ only depends on $\dot x$ and the group $\sG$, we may occasionally also write $\sym_\sG(\dot x)$. Also, whenever $\sS$ is clear from context we we may simply write $\sym(\dot x)$. As usual, this will apply to all definitions we make.

\begin{definition}
    The set of \emph{hereditarily $\sS$-symmetric names of rank $\leq \alpha$}, is defined by recursion on $\alpha$ as \[\HS_{\alpha,\sS} = \left\{ \dot x \subseteq \mathbb{P} \times \bigcup_{\beta < \alpha} \HS_{\beta,\sS} : \dot x \text{ is $\sS$-symmetric} \right\}.\]
    The class of \emph{hereditarily $\sS$-symmetric names} is \[\HS_\sS = \bigcup_{\alpha \in \Ord} \HS_{\alpha, \sS}.\] The elements of $\HS_{\sS}$ are also called \emph{$\sS$-names}.
\end{definition}

One of the fundamental properties of $\HS_\sS$ is closure under the action of $\sG$: for every $\dot x\in\HS_\sS$, then for any $\pi\in\sG$ $\pi(\dot x)\in\HS_\sS$. This follows from the normality of $\sF$ and the fact that $\sym_\sS(\pi(\dot x))=\pi\sym_\sS(\dot x)\pi^{-1}$.

By moving from the forcing notion to its Boolean completion, and from a group of automorphisms to the entire group of automorphisms, we can slightly weaken the definition of normality to obtain a more uniform definition of a symmetric extension which now depends entirely on the filter of subgroups. This is done in Grigorieff's work \cite{Grigorieff1975}, but in practice this is a bit more cumbersome to apply to standard situations, and so limiting the group is often useful. In other manuscripts, $\HS_{\sS}$ is usually denoted by $\HS_\sF$ instead, as the same group may have different filters. But as our new notation unfolds throughout the rest of the section, the choice of $\HS_\sS$ will make more sense.

\begin{definition}
    $\sS$ is called \emph{tenacious} if there is a dense set of conditions $p \in \mathbb{P}$, such that $\fix(p) \in \sF$, where $\fix(p) := \{ \pi \in \sG : \pi(p) = p \}$.
\end{definition}

\subsection{Symmetric extensions}

Recall that a filter $G$ on $\mathbb{P}$ is a non-empty upwards-closed subset of $\mathbb{P}$ so that for any $p, q\in G$, there is $r \in G$ so that $r \leq p,q$.

\begin{definition}
    Let $G$ be a filter on $\mathbb{P}$. Then we define the model \[V[G]_\sS := \{ \dot x^G : \dot x \in \HS_\sS \}.\]
\end{definition}

\begin{thm}[see, e.g., {\cite{Jech1973}}]\label{thm:basicsym}
    Let $G$ be a $\mathbb{P}$-generic filter over $V$. Then $V[G]_\sS \models \ZF$.
\end{thm}

Moreover we can define a forcing relation $\Vdash_\sS$ for symmetric extensions. More specifically, whenever $\varphi$ is a formula in the language of set theory, $p \in \mathbb{P}$ and $\dot x_0, \dots, \dot x_n$ are $\mathbb{P}$-names so that $\Vdash \dot x_0, \dots, \dot x_n \in V[\dot G]_\sS$, then \[p \Vdash_\sS \varphi(\dot x_0, \dots, \dot x_n)\] iff $V[G]_\sS \models \varphi(\dot x_0^G, \dots, \dot x_n^G)$ for every $\mathbb{P}$-generic $G$ over $V$ containing $p$. Note that when $\varphi$ is $\Delta_0$, there is no distinction between $p \Vdash \varphi$ and $p \Vdash_\sS \varphi$.

It turns out that $\mathbb{P}$-genericity is far stronger than necessary for Theorem~\ref{thm:basicsym}.\footnote{This is implicit in \S8 of \cite{Karagila2019}.}

\begin{definition}
    We say that $G$ is an \emph{$\mathcal{S}$-generic} filter over $V$ if $G$ is a filter on $\mathbb{P}$ and $G \cap D \neq \emptyset$ for every $D \subseteq \mathbb{P}$, $D \in V$ such that 
    \begin{enumerate}
        \item $D$ is dense in $\mathbb{P}$ and
        \item $D$ is $\sS$-symmetric, i.e., there is $H \in \mathcal{F}$ such that $\pi``D = D$ for every $\pi \in H$.
    \end{enumerate}
    Under these assumptions $D$ is called \emph{symmetrically dense}. We may also say that $G$ is \emph{symmetrically generic} to make clear that it is not necessarily $\mathbb{P}$-generic.
\end{definition}

\begin{thm}[Symmetric Forcing Theorem]\label{thm:forcingtheorem}
Let $p \in \mathbb{P}$, $\dot x_0, \dots, \dot x_n$ be $\sS$-names and $\varphi$ a formula in the language of set theory. Then $p \Vdash_{\mathcal{S}} \varphi(\dot x_0, \dots, \dot x_n)$ iff for any $\mathcal{S}$-generic filter $G$ over $V$ containing $p$, $V[G]_\sS \models \varphi(\dot x_0^G, \dots, \dot x_n^G)$. In particular, $V[G]_\sS$ is a model of $\ZF$.
\end{thm}

\begin{proof}
This is easy to verify by noting that in the usual inductive proof of the forcing theorem the only dense sets we need to hit are symmetric. We use standard facts on the forcing relation $\Vdash_\sS$ that we already know, which makes this quite short.

Consider $p \Vdash_\sS \dot x \subseteq \dot y$ and $p \Vdash_\sS \dot x \in \dot y$ by induction on the ranks of $\dot x$ and $\dot y$. If $\rk(\dot x), \rk(\dot y) \leq \alpha$, then $p \Vdash_\sS \dot x \subseteq \dot y$ iff for every $(r, \tau)\in \dot x, r \parallel p$, the set $D_{(r,\tau)} = \{ q : \exists (r', \sigma) \in \dot y : q \leq r' \wedge q \Vdash_\sS \tau = \sigma \}$ is dense in $\{ q : q \leq p,r \}$. The set $D_{(r,\tau)}$ above is $\sS$-symmetric and $\tilde D_{(r,\tau)} = D_{(r,\tau)} \cup \{ q : q \perp D_{(r,\tau)} \}$ is symmetrically dense. Thus $G \cap \tilde  D_{(r,\tau)} \neq \emptyset$ and if $p, r\in G$, then for any $q \in G \cap \tilde  D_{(r,\tau)}$, there is $q' \leq q,p,r$ and so necessarily $q \in D_{(r,\tau)}$ and $G \cap D_{(r,\tau)} \neq \emptyset$. By the inductive assumption, since $q \Vdash_\sS \tau = \sigma$ and $q \in G$, $\tau^G = \sigma^G \in \dot x^G \cap \dot y^G.$ Bringing everything together, we find that $p \Vdash_\sS \dot x \subseteq \dot y$ and $p \in G$ imply $\dot x^G \subseteq \dot y^G$. On the other hand, if for every $\sS$-generic $G \ni p$, $\dot x^G \subseteq \dot y^G$, then all sets $D_{(r,\tau)}$ must be dense in $\{ t : t \leq p,r \}$, otherwise there is a symmetric generic $G$ with $p,r \in G$, but $G \cap D_{(r,\tau)} = \emptyset$. But then $\tau^G \in \dot x^G \setminus \dot y^G$ using the inductive assumption. Thus indeed, $p \Vdash_\sS \dot x \subseteq \dot y$. Then we have also proven the theorem for $p \Vdash_\sS \dot x = \dot y$. The proof for $p \Vdash_\sS \dot x \in \dot y$ builds on this and is similar. 

Once we have proven the theorem for atomic formulas, the rest follows easily by induction on the complexity of $\varphi$. 

If $p \Vdash_\sS \neg \varphi$, $p\in G$, then $V[G]_{\sS} \not\models \varphi$, since otherwise for some $q \in G$ and thus some $r\leq p, q$, $r \Vdash_\sS \varphi$, by the inductive assumption. Thus $V[G]_{\sS} \models \neg\varphi$. On the other hand, if $p \not\Vdash_\sS \neg\varphi$, there is $q \leq p$, such that $q \Vdash_\sS \varphi$ and thus there is an $\sS$-generic with $q \in G$ and $V[G]_\sS \models \varphi$.

Finally, $p \Vdash_\sS \forall x \varphi(x)$ iff for every $\sS$-name $\dot x$, $p \Vdash_\sS \varphi(\dot x)$, iff for every $\sS$-name $\dot x$ and any $\sS$-generic $G \ni p$, $V[G]_\sS \models \varphi(\dot x^G)$ iff for every $\sS$-generic $G \ni p$, $V[G]_\sS \models \forall x \varphi(x)$.
\end{proof}

Implicit in the above proof is the fact that if $\dot x$ is $\sS$-symmetric, then $\{p\in\mathbb{P}:p\Vdash_\sS\varphi(\dot x)\lor p\Vdash_\sS\lnot\varphi(\dot x)\}$ is a symmetrically dense open set. Let us remark that while the forcing relation $\Vdash_\sS$ may take arbitrary names $\dot x_0, \dots, \dot x_n$ that are forced by $\mathbb{P}$ to be elements of the symmetric extension, in the above forcing theorem we strictly need that these are hereditarily symmetric. For instance consider the case that $\sG = \Aut(\mathbb{P})$, $\sF = \{ \sG \}$ and say $\mathbb{P} = 2^{<\omega}$. Then it is easy to note that any $x \in 2^\omega$ induces an $\sS$-generic $G_x = \{ x \restriction n : n \in \omega \}$. Let $A \subseteq 2^{<\omega}$ be a maximal antichain and $x \in 2^\omega$ so that $G_x \cap A = \emptyset$. The name $\dot y = \{ (s, \check{0}) : s \in A\}$ is clearly forced to be in the symmetric extension as it is forced to equal $\{ 0\}$. On the other hand, $\dot y^{G_x} = \emptyset$. 

\begin{definition}
    A \emph{symmetric extension of $V$} is any model of the form $V[G]_\sS$, where $\sS$ is a symmetric system and $G$ is $\sS$-generic over $V$.
\end{definition}

Traditionally, and throughout most of the literature, symmetric extensions are only formed using $\mathbb{P}$-generic filters. But in fact, as we shall see later, these two notions completely agree. More precisely, we will show that whenever $G$ is $\sS$-generic there is also a $\mathbb{P}$-generic $H$ so that $V[G]_\sS = V[H]_\sS$ (see Theorem~\ref{thm:quotsummary}). In consideration of the general theory, it makes sense though to consider this more general form of genericity and this will become clear through Theorem~\ref{lem:twostepfactor}.

\subsection{Closed names}

\begin{definition}\label{def:closedname}
    For an $\sS$-name $\dot x$ and an ordinal $\alpha$, define \[\cl_\alpha(\dot x) := \{ (p, \dot y) : \dot y \in \HS_\alpha \wedge p \Vdash \dot y \in \dot x \}.\]
    Moreover, we let $\cl(\dot x) := \cl_\alpha(\dot x)$, where $\alpha$ is minimal such that $\Vdash \dot x = \cl_\alpha(\dot x)$. We say that $\dot x$ is \emph{closed} if $\dot x = \cl(\dot x)$.
\end{definition}

Note of course that $\cl(\cl(\dot x)) = \cl(\dot x)$, justifying this terminology. The following lemma is trivial but nevertheless quite substantial. It implies that we can uniformly pick representatives from each equivalence class of names under forcing-equivalence ($\dot x \sim \dot y$ if and only if $\Vdash \dot x = \dot y$). This is a particularly useful observation in the context of $\neg \AC$.

\begin{lemma}
    Every $\sS$-name is forcing equivalent to a unique closed $\sS$-name.
\end{lemma}

\begin{lemma}\label{lem:definablename}
Let $\dot x_0, \dots, \dot x_n$ be $\sS$-names and $\varphi$ a formula such that \[\Vdash_\sS \exists ! y \varphi(y, \dot x_0, \dots, \dot x_n).\] Then there is a unique closed $\sS$-name $\dot y$, such that $\Vdash_\sS\varphi(\dot y, \dot x_0, \dots, \dot x_n)$ and we have that $\sym(\dot x_0, \dots, \dot x_n) \leq \sym(\dot y)$.
\end{lemma}

\begin{proof}
    Simply let \[\dot y = \{ (p, \dot z) : \dot z \in \HS_\alpha \wedge p \Vdash_\sS \forall y ( \varphi(y, \dot x_0, \dots, \dot x_n) \rightarrow \dot z \in y) \},\] where $\alpha$ is minimal such that this works. 
\end{proof}

Thus, whenever we extend the language of $\ZF$ by a definable function (such as $\mathcal{P}(\cdot), (\cdot,\cdot), \bigcup, \trcl,  \dots$), as we implicitly do all the time, we can canonically extend it to $\sS$-names (equally to $\mathbb{P}$-names of course). A particular instance of this that will appear in the next section is when $\dot f$ is an $\sS$-name for a function and $\dot f(\dot x)$ denotes the unique closed name $\dot y$ such that $\Vdash (\dot x, \dot y) \in \dot f$. Another important use of Lemma~\ref{lem:definablename} is the following which serves as a very useful trick and can be thought of as a weak form of mixing.

\begin{lemma}\label{lem:thetrick}
    Let $\dot x, \dot x_0, \dots, \dot x_n$ be $\sS$-names and $p \in \mathbb{P}$ be so that for a formula $\chi$ \[p \Vdash_\sS \chi(\dot x, \dot x_0, \dots, \dot x_n).\] Suppose that there is some $\sS$-name $\dot z$ so that $\Vdash_\sS \chi(\dot z, \dot x_0, \dots, \dot x_n)$. 
    
    Then there is an $\sS$-name $\dot y$, uniformly definable from $\dot x$, $\dot x_0, \dots, \dot x_n$ and $\dot z$, so that $\sym(\dot z, \dot x, \dot x_0, \dots, \dot x_n) \leq \sym(\dot y)$, $p \Vdash_\sS \dot y = \dot x$ and $\Vdash_\sS \chi(\dot y, \dot x_0, \dots, \dot x_n)$.
\end{lemma}

\begin{proof}
Simply consider the formula \[\varphi(\dot x, \dot x_0, \dots, \dot x_n, \dot z, y) = (\chi(\dot x, \dot x_0, \dots) \wedge y = \dot x ) \vee (\neg\chi(\dot x, \dot x_0, \dots) \wedge y = \dot z) .\]

The name $\dot y$ obtained through Lemma~\ref{lem:definablename} clearly works.
\end{proof}

What the lemma is saying effectively, is that in many situations where we consider an object $x\in V[G]_\sS$ with a certain property $\chi$, there is a name $\dot y$ for $x$ so that already $\Vdash_\sS \chi(\dot y)$. In many common such situations we have that $\sym(\dot z) = \sym(\dot x_0) = \sG$ and thus $\sym(\dot y) = \sym(\dot x)$. For example, if $\dot{\mathbb{Q}}$ is an $\sS$-name for a forcing poset, $\sym(\dot{\mathbb{Q}}) = \sG$ and $p \Vdash \dot q \in \dot{\mathbb{Q}}$, we may simply assume that $\Vdash_\sS \dot q \in \dot{\mathbb{Q}}$ as the maximal condition of ${\mathbb{Q}}$ is a canonical witness for an element of $\mathbb{Q}$. For $p \Vdash \dot \pi \in \dot H$, where $\Vdash \dot H \leq \Aut(\dot{\mathbb{Q}})$, there is the identity $\id$, etc.

While Lemma~\ref{lem:thetrick} is a quite innocent looking observation, it is a key point in our renewed exposition of iterations that eliminates many of the technical difficulties that arise in \cite{Karagila2019}.

\begin{exmp}[Basic Cohen model]\label{ex:Cohen}
    Let $\mathbb{P} = \Add(\omega, \omega)$, that is, $\mathbb{P}$ consists of finite partial functions $p \colon \omega \times \omega \to 2$. Let $\sG$ be the group of finitary permutations $\pi$ of $\omega$ acting on a condition $p$ by $\pi(p)(\pi(n),m) = p(n, m)$. The filter $\sF$ is generated by subgroups of $\sG$ of the form $\fix(n) := \{ \pi \in \sG : \pi(n) = n \}$. The system $\sC = (\mathbb{P}, \sG, \sF)$ generates what is known as the \emph{basic Cohen model}. 
    
    If $\dot c_n$ is the canonical name for the $n$th Cohen real, then $\pi(\dot c_n) = \dot c_{\pi(n)}$ and it follows that $\sym(\dot c_n) = \fix(n)$. The only names appearing in $\dot c_n$ are check names, so $\dot c_n \in \HS$. Moreover $\dot A := \{ \dot c_n : n \in \omega \}^\bullet \in \HS$, as $\sym(\dot A) = \sG$, but in $V[G]_\sC$, $\dot A^G$ can't be well-ordered and in fact is Dedekind-finite. For more information on the basic Cohen model, see \cite{Jech1973}.
\end{exmp}

\section{Two step iterations}\label{sec:twostep}

\begin{definition}\label{def:twostep}
    Let $\sS_0 = ({\mathbb{P}}_0,{\sG}_0, {\sF}_0)$ be a symmetric system and let $\dot \sS_1 = (\dot{\mathbb{P}}_1, \dot \sG_1, \dot\sF_1)^\bullet$ be an $\sS_0$-name such that $\Vdash_{{\sS_0}} \dot \sS_1 \text{ is a symmetric system}$ and $\sym (\dot \sS_1) = \mathcal{G}_0$. 
    We define the two step iteration $\sS_0 * \dot \sS_1 = (\mathbb{P}, \sG, \sF)$ as follows. 

    \begin{enumerate}
        \item $\mathbb{P} = \mathbb{P}_0 *_{\sS_0} \dot{\mathbb{P}}_1$ consists of pairs $(p, \dot q)$ such that $p \in \mathbb{P}_0$ and $\dot q$ is a closed $\sS_0$-name such that $ \Vdash_{\sS_0} \dot q \in \dot{\mathbb{P}}_1$. The extension relation is given by $(p', \dot q') \leq (p, \dot q)$ iff $p' \leq p$ and $p' \Vdash_{\sS_0} \dot q' \leq \dot q$. Note that this makes $\mathbb{P}$ a dense subposet of the usual forcing iteration $\mathbb{P}_0 * \dot{\mathbb{P}}_1$.\footnote{Also note $\mathbb{P}$ is indeed a set and not a proper class, as there are only set-many closed names for elements of $\dot{\mathbb{P}}_1$.} 
        \item Any $\bar \pi = (\pi_0, \dot \pi_1)$, where $\pi_0 \in \sG_0$, $\dot \pi_1$ is an $\sS_0$-name and $\Vdash_{\sS_0} \dot \pi_1 \in \dot\sG_1$, is identified with the map given by \[\bar \pi (p, \dot q) = (\pi_0(p), \dot \pi_1(\pi_0(\dot q))).\] $\sG$ is the group of all such $\bar \pi$.\footnote{Note that by the convention set after Lemma~\ref{lem:definablename}, $\dot \pi_1(\pi_0(\dot q))$ is another closed name. It is not hard to check that the map induced by $\bar \pi$ is indeed an automorphism and we will check that $\sG$ forms a group in Lemma~\ref{lem:twostep}. Let us remark as well that different pairs $(\pi_0, \dot \pi_1)$ may give rise to the same map.} 
        
        \item Let $(H_0, \dot H_1)$ be such that $H_0 \in {\sF}_0$, $\dot H_1$ is an $\sS_0$-name, $H_0 \leq \sym(\dot H_1)$ and $\Vdash_{\sS_0} \dot H_1 \in \dot{\sF}_1$. Then we identify $(H_0, \dot H_1)$ with the subgroup of $\sG$ of all automorphisms $\bar \pi = (\pi_0, \dot \pi_1)$, where $\pi_0 \in H_0$ and $\Vdash \dot \pi_1 \in \dot H_1$. $\sF$ is generated by these subgroups. 
    \end{enumerate}
\end{definition}

There are many alternate and equivalent ways to define and denote two-step iterations. For us, the purpose is to at least give a technically precise and workable definition in order to study the basic theory. But our hope is that eventually, just as with forcing, when more fundamental questions have been cleared and enough intuition has been built, the precise definition won't matter as much in practice. In this regard, let us make a few remarks. 

The purpose of using closed names in (1) is simply to ensure that the maps induced by pairs $\bar \pi$ are well-defined and form automorphisms. But of course there are many other ways to get around this. For one, there are certainly other choices of representatives for names than closed names as defined in Definition~\ref{def:closedname} that we could use. More interestingly though, we may completely eliminate the mention of closed names in practice by identifying $(p, \dot q)$ with $(p, \cl(\dot q))$, when $\dot q$ is an arbitrary $\sS_0$-name. This is exactly the type of abuse of notation that we already are using by writing $(\pi_0, \dot \pi_1)$ and the like.\footnote{We may also note that whenever $(\pi_0, \dot \pi_1)$ and $(\sigma_0, \dot \sigma_1)$ induce the same map, then actually $\pi_0 = \sigma_0$ and $\Vdash \dot \pi_1 = \dot \sigma_1$.}

\begin{lemma}\label{lem:twostep}
    $\sS_0 * \dot \sS_1$ is a symmetric system.
\end{lemma}

\begin{proof}
    To check that each $(H_0, \dot H_1)$, and in particular $\sG = (\sG_0, \dot \sG_1)$ is a group, we first compute a formula for composition: \begin{align*}
    (\pi_0, \dot \pi_1)\circ (\sigma_0, \dot \sigma_1) (p, \dot q) &= (\pi_0, \dot \pi_1)(\sigma_0(p), \dot \sigma_1 (\sigma_0(\dot q)))\\
    &=((\pi_0 \circ \sigma_0)(p), \dot \pi_1 (\pi_0(\dot \sigma_1 (\sigma_0(\dot q)))) )\\
    &= ((\pi_0 \circ \sigma_0)(p), \dot \pi_1 \circ \pi_0(\dot \sigma_1) (\pi_0 \circ \sigma_0(\dot q)))\\
    &= (\pi_0 \circ \sigma_0, \dot \pi_1 \circ \pi_0(\dot \sigma_1))(p, \dot q).
\end{align*}
Since $\pi_0, \sigma_0 \in H_0$, also $\pi_0 \circ \sigma_0 \in H_0$ and since $\Vdash \dot \sigma_1 \in \dot H_1$ and $\pi_0$ fixes $\dot H_1$, also $\Vdash \pi_0(\dot \sigma_1) \in \dot H_1$ and so $\Vdash \dot \pi_1 \circ \pi_0(\dot \sigma_1) \in \dot H_1$.
It is also easy to see that \[(\sigma_0, \dot \sigma_1)^{-1} = (\sigma_0^{-1}, \sigma_0^{-1}(\dot \sigma_1^{-1})) \in (H_0, \dot H_1).\]

Let us check that $\sF$ is normal. Suppose $(H_0, \dot H_1) \in \sF$ and $ \bar \sigma = (\sigma_0, \dot \sigma_1) \in \sG$ is arbitrary. Consider $H_0' = H_0 \cap \sym(\sigma_0^{-1}(\dot \sigma_1^{-1}))$, so $(H_0', \dot H_1) \in \sF$ is a subgroup of $(H_0, \dot H_1)$. 
Let $(\pi_0, \dot \pi_1) \in (H_0', \dot H_1)$ be arbitrary. Then:
\begin{align*}
    (\sigma_0, \dot \sigma_1)\circ (\pi_0, \dot \pi_1) \circ  (\sigma_0, \dot \sigma_1)^{-1}&= (\sigma_0, \dot \sigma_1)\circ (\pi_0, \dot \pi_1) \circ (\sigma_0^{-1}, \sigma_0^{-1}(\dot \sigma_1^{-1}))\\
    &=(\sigma_0, \dot \sigma_1)\circ (\pi_0 \circ \sigma_0^{-1}, \dot \pi_1 \circ \pi_0(\sigma_0^{-1}(\dot \sigma_1^{-1})) )\\
     &=(\sigma_0, \dot \sigma_1)\circ (\pi_0 \circ \sigma_0^{-1}, \dot \pi_1 \circ \sigma_0^{-1}(\dot \sigma_1^{-1}) )\\
    &= (\sigma_0 \circ \pi_0 \circ \sigma_0^{-1}, \dot \sigma_1 \circ \sigma_0(  \dot \pi_1 \circ \sigma_0^{-1}(\dot \sigma_1^{-1})))\\
    &= (\sigma_0 \circ \pi_0 \circ \sigma_0^{-1},  \dot \sigma_1 \circ \sigma_0(  \dot \pi_1) \circ \sigma_0\sigma_0^{-1}(\dot \sigma_1^{-1}))\\
    &= (\sigma_0 \circ \pi_0 \circ \sigma_0^{-1},  \dot \sigma_1 \circ \sigma_0(  \dot \pi_1) \circ \dot \sigma_1^{-1})\\
    &= (\sigma_0 \circ \pi_0 \circ \sigma_0^{-1},  \dot \sigma_1 \circ \dot \tau \circ \dot \sigma_1^{-1}),
\end{align*}
where $\Vdash \dot \tau \in \sigma_0(\dot H_1)$. Note that $\Vdash \sigma_0(\dot H_1) \in \dot \sF_1$, since $\sym(\dot \sS_1) = \sG_0$. So if we let $\dot K_1$ be a name for $\dot \sigma_1 \sigma_0(\dot H_1) \dot \sigma_1^{-1}$ and $K_0 = H_0' \cap \sigma_0 H'_0 \sigma_0^{-1} \cap \sym(\dot K_1)$, then $(K_0, \dot K_1) \in \sF$ and $(K_0, \dot K_1) \subseteq \bar \sigma (H_0', \dot H_1) \bar \sigma^{-1} \subseteq \bar \sigma (H_0, \dot H_1) \bar \sigma^{-1}$.
\end{proof}

The following factorization theorem says that this is a correct definition of the two-step iteration.

\begin{thm}[Factorization Theorem]\label{lem:twostepfactor}
Let $\sS = \sS_0 * \dot \sS_1$ be a two-step iteration as above. Whenever $G$ is $\sS$-generic over $V$, then $G_0 = \dom G$ is $\sS_0$-generic over $V$ and $G_1 = \{ \dot q^{G_0} : \exists p \in G_0 ((p, \dot q) \in G)\}$ is $\sS_1$-generic over $V[G_0]_{\sS_0}$, where $\sS_1 = \dot \sS_1^{G_0}$, and \[V[G]_{\sS} = V[G_0]_{\sS_0}[G_1]_{\sS_1}.\]

On the other hand, whenever $G_0$ is $\sS_0$-generic over $V$ and $G_1$ is $\sS_1$-generic over $V[G_0]_{\sS_0}$, then $G = G_0 * G_1 = \{ (p, \dot q) : p \in G_0 \wedge \dot q^{G_0} \in G_1 \}$ is $\sS$-generic over $V$ and $V[G]_{\sS} = V[G_0]_{\sS_0}[G_1]_{\sS_1}.$
\end{thm}

\begin{proof}
    The proof is essentially the same as for forcing and relies on the observation that a symmetrically dense subset of $\mathbb{P}$ is essentially the same as an $\sS_0$-name for an $\dot \sS_1$-symmetrically dense subset of $\dot{\mathbb{P}}_1$.
    
    First let $G$ be $\sS$-generic over $V$. Let $D_0$ be a symmetrically dense subset of $\mathbb{P}_0$ as witnessed by $H \in \sF_0$. Then $D = \{ (p, \dot q) \in \mathbb{P} : p \in D_0 \}$ is a symmetrically dense subset of $\mathbb{P}$ as witnessed by $(H, \dot \sG_1) \in \sF$. As $G \cap D \neq \emptyset$, we also have that $G_0 \cap D_0 \neq \emptyset$. Thus, $G_0$ is $\sS_0$-generic over $V$. Next let $D_1 \in V[G_0]_{\sS_0}$ be a symmetrically dense subset of $\mathbb{P}_1$ as witnessed by $H_1 \in \sF_1$. Then there are $\sS_0$-names $\dot D_1$ and $\dot H_1$ such that $\dot D_1^{G_0} = D_1$, $\dot H_1^{G_0} = H_1$, $\Vdash_{\sS_0} \dot D_1 \text{ is dense}$, $\Vdash_{\sS_0} \dot H_1 \in \dot \sF_1$ and $\Vdash_{\sS_0} \forall \pi \in \dot H_1 ( \pi``\dot D_1 = \dot D_1)$.\footnote{ If $\dot D_1'^{G_0} = D_1$, $\dot H_1'^{G_0} = H_1$ are arbitrary, we find canonical $\sS$-names $\dot D_1$, $\dot H_1$ so that it is forced that ``if $\dot D_1'$ is symmetrically dense as witnessed by $\dot H_1'$, then $\dot D_1 = \dot D_1'$, $\dot H_1 = \dot H_1'$ and otherwise $\dot D_1 = \dot{\mathbb{P}}_1$, $\dot H_1 = \dot{\sG}_1$". This is an instance of Lemma~\ref{lem:thetrick}.} Consider \[D = \{ (p, \dot q) : p \Vdash_{\sS_0} \dot q \in \dot D_1 \} \in V.\]
    $D$ is clearly dense and symmetric as witnessed by $(\sym(\dot D_1, \dot H_1), \dot H_1)$: 
    
    For $(\pi_0,\dot \pi_1) \in (\sym(\dot D_1, \dot H_1), \dot H_1)$, 
        \begin{align*}
            (\pi_0,\dot \pi_1)(p_0, \dot q) \in D &\text{ iff } (\pi_0(p_0),  \dot \pi_1(\pi_0(\dot q))) \in D\\
            &\text{ iff } \pi_0(p_0) \Vdash_{\sS_0} \dot \pi_1(\pi_0(\dot q)) \in \dot D_1\\
            &\text{ iff } p_0 \Vdash_{\sS_0} \pi_0^{-1}(\dot \pi_1(\pi_0(\dot q))) \in \pi_0^{-1}(\dot D_1)\\
            &\text{ iff } p_0 \Vdash_{\sS_0} \pi_0^{-1}(\dot \pi_1)(\dot q) \in \dot D_1 \\
            &\text{ iff } p_0 \Vdash_{\sS_0}  \dot q \in \dot D_1,
        \text{ since } p_0 \Vdash_{\sS_0} \pi_0^{-1}(\dot \pi_1) \in \pi_0(\dot H_1) = \dot H_1, \\
        &\text{ iff } (p_0, \dot q) \in D.
        \end{align*} 
    Thus $G \cap D \neq \emptyset$ and $G_1 \cap D_1 \neq \emptyset$.

    In the other direction, suppose $G_0$ is $\sS_0$-generic over $V$, $G_1$ is $\sS_1$-generic over $V[G_0]_{\sS_0}$ and $G$ is defined as in the statement of the Theorem. Let $D$ be a symmetrically dense subset of $\mathbb{P}$ as witnessed by $(H_0, \dot H_1)$. Then $\dot D_1 := D$ is literally an $\sS_0$-name with $H_0 \subseteq \sym(\dot D_1)$, since \[\pi_0(\dot D_1) = \{ (\pi_0(p), \pi_0(\dot q)) : (p, \dot q) \in D \} = (\pi_0, \dot \id)`` D = D = \dot D_1,\] for any $\pi_0 \in H_0$. Moreover, $\dot D_1$ is a name for a dense set and whenever $\dot \pi_1$ is such that $\Vdash \dot \pi_1 \in \dot H_1$, $\dot D_1 = D = (\id, \dot \pi_1) `` D = \{(p, \dot \pi_1(\dot q)) : (p, \dot q) \in D \}$ is an $\sS_0$-name for $\dot \pi_1 `` \dot D_1$, so \[\Vdash \dot \pi_1 `` \dot D_1 = \dot D_1.\]
    Thus $G_1 \cap \dot D_1^{G_0} \neq \emptyset$, i.e., there is $(p, \dot q) \in \dot D_1 = D$ such that $p \in G_0$ and $\dot q^{G_0} \in G_1$. According to the definition, $(p, \dot q) \in G$ so $G \cap D \neq \emptyset$.

    Finally to see that $V[G]_{\sS} = V[G_0]_{\sS_0}[G_1]_{\sS_1}$, note that we can identify an $\sS$-name $\dot x$ with an $\sS_0$-name for an $\sS_1$-name, by recursively defining \[[\dot x] = \{ (p, (\dot q, [\dot y])^\bullet) : ((p, \dot q), \dot y) \in \dot x \}.\]
    It is easy to check that $\dot x^G = ([\dot x]^{G_0})^{G_1}$. Similarly, when $\dot x$ is an $\sS_0$-name for an $\sS_1$-name, we define \[]\dot x [\ = \{ ((p, \dot q), ]\dot y [) : p \Vdash_{\sS_0} (\dot q, \dot y)^\bullet \in \dot x\}.\qedhere\]\end{proof}

It follows also quite easily that any $\sS_0 * \dot \sS_1$-generic is of the form $G_0*G_1$ as above. Let us note that the use of symmetric generics, in contrast to fully forcing generic filters, is strictly needed for the Factorization Theorem. For instance, consider the system $\sS_0 = (\mathbb{C}, \Aut(\mathbb{C}), \{\Aut(\mathbb{C})\})$, where $\mathbb{C} = \Add(\omega, 1)$ is Cohen forcing. By induction on the rank of an $\sS_0$-name $\dot x$ one can easily see that $\Vdash_{\sS_0} \dot x = \check x$ for some set $x$. Thus this system doesn't add any new sets. But a $\mathbb{C}$-generic is a Cohen real $c$ and $V[c]_{\sS_0} = V$.  Now let $\sS_1 = (\mathbb{C}, \{ \id \}, \{\{ \id \}\}) \in V[c]_{\sS_0}$. This second system adds exactly a Cohen real over $V[c]_{\sS_0}$, which $c$ itself is. But $(c,c)$ is certainly not $\mathbb{C} \times \mathbb{C}$-generic over $V$. It does, however, correspond to a symmetrically generic filter for $\sS_0 * \dot{\sS}_1$.

This is not a pathological example. In fact, by the results of Grigorieff \cite{Grigorieff1975}, and as we shall reprove later, for any symmetric system $(\mathbb{P}, \sG, \sF)$ and a $\mathbb{P}$-generic $G$ over $V$, $G$ itself is generic over $V[G]_\sS$ for some forcing notion, which must be a non-trivial forcing if $V[G]\neq V[G]_\sS$. 

\section{Ideal support iterations}

Building on the last section we now introduce transfinite iterations of symmetric systems, specifically iterations with supports coming from some arbitrary ideal on the ordinals. Some of what is written here can also be found identically in the second author's joint work with P.~Holy \cite{HolySchilhan}, which presents the particular case of finite support iterations in the framework developed here.

\subsection{General iterations}

We first give some raw definition of how iterations are supposed to look and be built. We call these \emph{pre-iterations} and verify some basic properties. The reason why we do not refer to these as ``true'' iterations yet is that initial stages of pre-iterations do not even need to correspond to intermediate models. The are a few simple conditions (Definition~\ref{def:iteration}) that we can add to ensure this, but the presentation is more transparent by giving them afterwards.

\begin{definition}\label{def:preiteration}
Consider a sequence of the form \[\langle \sS_{\alpha}, \dot \sT_\alpha : \alpha < \delta \rangle = \langle (\mathbb{P}_{\alpha}, \sG_\alpha, \sF_\alpha), (\dot{\mathbb{Q}}_\alpha, \dot \sH_\alpha, \dot \sE_\alpha)^\bullet : \alpha < \delta \rangle,\] where each $\sS_\alpha$ is a symmetric system, $\dot \sT_\alpha$ is an $\sS_\alpha$-name for a symmetric system and $\sym(\dot \sT_\alpha) = \sG_\alpha$. Then we call this sequence a \emph{pre-iteration} of length $\delta$ if for each $\alpha < \delta$, 
    \begin{enumerate}
       \item  \begin{enumerate}
        \item $\mathbb{P}_\alpha$ consists of sequences $\bar p = \langle \dot p(\beta) : \beta < \alpha \rangle$, such that $\bar p \restriction \beta \in \mathbb{P}_\beta$ and $\dot p(\beta)$ is an $\sS_\beta$ name for an element of $\dot{\mathbb{Q}}_\beta$, for all $\beta < \alpha$,\footnote{These may or may not be all sequences.}
        \item $\sG_\alpha$ consists of automorphisms of $\mathbb{P}_\alpha$ represented, as detailed below, by sequences $\bar \pi = \langle \dot \pi(\beta) : \beta < \alpha \rangle$,  where $\bar \pi \restriction \beta \in \sG_\beta$ and $\dot \pi(\beta)$ is an $\sS_\beta$ name for an element of $\dot \sH_\beta$, for all $\beta < \alpha$,
        \item $\sF_\alpha$ is generated by subgroups of $\sG_\alpha$ represented, as detailed below, by sequences $\bar H = \langle \dot H(\beta) : \beta < \alpha \rangle$,  where $\bar H \restriction \beta \in \sF_\beta$, $\dot H(\beta)$ is an $\sS_\beta$ name for an element of $\dot \sE_\beta$ and $\bar H \restriction \beta \leq \sym(\dot H(\beta))$, for all $\beta < \alpha$,
        \end{enumerate}
        \item $\sS_{\alpha +1} = \sS_\alpha * \dot{\sT_\alpha}$, where pairs $(\bar p, \dot q)$, $(\bar \pi, \dot \sigma)$, $(\bar H, \dot K)$ as in Definition~\ref{def:twostep} are replaced with the sequences $\bar p^\frown \dot q$, $\bar \pi^\frown \dot \sigma$ and $\bar H^\frown \dot K$ respectively.
    \end{enumerate}
    For $\alpha$ limit,
    \begin{enumerate}\setcounter{enumi}{2}
       \item \begin{enumerate}
        \item the order on $\mathbb{P}_\alpha$ is given by $\bar q \leq \bar p$ iff $\bar q \restriction \beta \leq \bar p \restriction \beta$ for each $\beta < \alpha$,
        \item for any $\bar p \in \mathbb{P}_\alpha$ and  $\bar \pi \in \sG_\alpha$, \[\bar \pi(\bar p) = \bigcup_{\beta < \alpha} \bar \pi \restriction \beta (\bar p \restriction \beta),\]
        \item for any $\bar \pi \in \sG_\alpha$, $\bar H \in \sF_\alpha $, $\bar \pi \in \bar H$ iff $\bar \pi \restriction \beta \in \bar H \restriction \beta$ for all $\beta < \alpha$.
        \end{enumerate}
    \end{enumerate}
\end{definition}

Note that, in each case, ``as detailed below" refers to both the information gained from (2), which tells us for instance which automorphism is represented by $\bar \pi^ \frown \dot \sigma$ assuming we know what $\bar \pi$ represents, as well as the information given in (3) for the limit steps. The action of $\bar \pi$ and the meaning of $\bar H$ can thus be seen as being given inductively. On the other hand, the above is simply a definition which may or may not hold true for $\langle \sS_{\alpha}, \dot \sT_\alpha : \alpha < \delta \rangle$. Of course, in practice, Definition~\ref{def:preiteration} is supposed to be read as an instruction as how to construct an iteration.\footnote{Following our official definition of two-step iterations, it may also be helpful to remark that each $\dot p(\beta)$ in a condition will be a closed name. Again, this is a rather immaterial matter.}

Note one aspect of the definition that could be seen as an oddity, but corresponds to an issue (or rather a non-issue) that also appears in most presentations of iterated forcing. Namely, if read carefully, $\mathbb{P}_0$ can only be the trivial forcing as it can only contain the empty sequence. All this is saying is that according to the official definition, iterations start with a trivial symmetric system $\sS_0$. In practice, this is of course irrelevant and we can safely ignore this.

Also, in practice one only considers iterations that have a last final step, the system that we want to extend with. But we need to be more general with the definition here for the presentation to make sense. Iterations of the form $\langle \sS_{\alpha}, \dot \sT_\alpha : \alpha \leq \delta \rangle$ are just special cases, where $\leq \delta$ is thought of as $< \delta +1$.

Let us now make some simple observations about pre-iterations. As in Lemma~\ref{lem:twostep} we compute some formulas for the composition and inverses of automorphisms. Towards general support iterations, we also define $\supp(\bar p) := \{ \alpha : \not\Vdash_{\sS_\alpha} \dot p(\alpha) = \mathds 1\}$, $\supp(\bar \pi) := \{ \alpha : \not\Vdash_{\sS_\alpha} \dot \pi(\alpha) = \id\}$ and $\supp(\bar H) := \{ \alpha : \not\Vdash_{\sS_\alpha} \dot H(\alpha) = \dot\sH_\alpha\}$. We refer to the above as \emph{support} in each case.

\begin{lemma}\label{lem:preiteration}
Let $\langle \sS_{\alpha}, \dot \sT_\alpha : \alpha < \delta \rangle$ be a pre-iteration as above, $\alpha < \delta$ arbitrary. 
\begin{enumerate}
    \item For any $\bar p \in \mathbb{P}_\alpha$, $\bar \pi \in \sG_\alpha$, $\bar \pi(\bar p) = \langle \dot \pi(\beta)(\bar \pi \restriction \beta(\dot p(\beta))) : \beta < \alpha \rangle$ and $\supp(\bar p) = \supp(\bar \pi(\bar p))$.
    \item For any $\bar p \in \mathbb{P}_\alpha$, $\bar \pi \in \sG_\alpha$, $\beta \leq \alpha$, $\bar \pi(\bar p) \restriction \beta = (\bar\pi \restriction \beta)(\bar p \restriction \beta)$.
    \item For any $\bar \pi, \bar \sigma \in \sG_\alpha$, $\bar \pi \circ \bar \sigma = \langle \dot \pi(\beta) \circ \bar \pi \restriction \beta(\dot \sigma(\beta)) : \beta < \alpha\rangle$ and $\supp(\bar \pi \circ \bar \sigma) \subseteq \supp(\bar \pi) \cup \supp(\bar \sigma)$.
    \item For any $\bar \pi \in \sG_\alpha$, $\bar \pi ^{-1} = \langle \bar \pi \restriction \beta^{-1}(\dot \pi(\beta)^{-1}) : \beta < \alpha\rangle$ and $\supp(\bar \pi^{-1}) = \supp(\bar \pi)$.
    \item For any $\bar \sigma \in \sG_\alpha$ and $\bar H \in \sF_\alpha$, where $\bar H \restriction \beta \leq \sym(\bar \sigma \restriction \beta^{-1}(\dot \sigma(\beta)^{-1}) ),$ for each $\beta<\alpha$, we have that \[\bar \sigma \bar H \bar \sigma^{-1} = \langle \dot \sigma(\beta) \bar \sigma \restriction \beta(\dot H(\beta)) \dot \sigma(\beta)^{-1} : \beta < \alpha \rangle\] and $\supp(\bar \sigma \bar H \bar \sigma^{-1}) = \supp(\bar H).$
    \item For any $\bar \sigma \in \sG_\alpha$ and $\bar H \in \sF_\alpha$, where $\bar H \restriction \beta \leq \sym(\dot \sigma(\beta)),$ for each $\beta<\alpha$, we have that \[\bar \sigma^{-1} \bar H \bar \sigma = \langle( \bar \sigma \restriction \beta)^{-1}\left(\dot \sigma(\beta)^{-1} \dot H(\beta) \dot \sigma(\beta)\right) : \beta < \alpha \rangle\] and $\supp(\bar \sigma \bar H \bar \sigma^{-1}) = \supp(\bar H).$
\end{enumerate}
\end{lemma}

\begin{proof}
    All of these are fairly easy inductions on $\alpha$. 
    \begin{enumerate}
        \item For the successor step, suppose we are given $\bar p  = \bar r^\frown \dot s$ and $\bar \pi = \bar \rho ^\frown \dot \sigma$. Then $\bar \pi(\bar p) = \bar \rho(\bar r)^\frown \dot \sigma(\bar \rho(\dot s))$ (see Definition~\ref{def:twostep}). We know already that $\supp(\bar \rho(\bar r)) = \supp(\bar r)$, and we simply note that $\Vdash \dot \sigma(\bar \rho(\dot s)) = \mathds 1$ iff $\Vdash \dot s = \mathds 1$. For the limit step, directly apply the inductive hypothesis and (3)(b) of Definition~\ref{def:preiteration}.
        \item Immediate from (1).
        \item Let us check the successor step. We are given $\bar \pi = \bar \rho^\frown \dot \mu$, $\bar \sigma = \bar \tau^\frown \dot \upsilon$. Then according to the composition formula computed in Lemma~\ref{lem:twostep}, \[(\bar \rho^\frown \dot \mu) \circ (\bar \tau^\frown \dot \upsilon) = (\bar \rho \circ \bar \tau)^\frown \dot \mu \circ \bar \rho(\dot \upsilon).\] Using the inductive hypothesis this corresponds exactly to what is being claimed. The claim about supports is equally simple. Again, the rest follows trivially from how automorphisms act in the limit case.
        \item Exactly as before.
        \item Similar to before, following the last paragraph in the proof of Lemma~\ref{lem:twostep}.
        \item Directly apply (4) and (5).\qedhere
    \end{enumerate}
\end{proof}

\begin{definition}\label{def:iteration}
    $\langle \sS_{\alpha}, \dot \sT_\alpha : \alpha < \delta \rangle$ is called an \emph{iteration} if for each $\alpha < \gamma < \delta$, 
    \begin{enumerate}
        \item for any $\bar p \in \mathbb{P}_\alpha$, $\bar p ^\frown \langle \dot{\mathds 1}_\beta : \beta \in [\alpha, \gamma) \rangle \in \mathbb{P}_\gamma$, where $\dot{\mathds 1}_\beta$ is the closed $\sS_\beta$-name for the trivial condition of $\dot \sT_\beta$,
       \item for any $\bar \pi \in \sG_\alpha$, $\bar \pi ^\frown \langle \dot{\id}_\beta : \beta \in [\alpha, \gamma) \rangle \in \sG_\gamma$, where $\dot{\id}_\beta$ is the closed $\sS_\beta$-name for the trivial automorphism of (the forcing of) $\dot \sT_\beta$,
       \item for any $\bar H \in \sF_\alpha$, $\bar H ^\frown \langle \dot{\sH}_\beta : \beta \in [\alpha, \gamma) \rangle \in \sF_\gamma$.
    \end{enumerate}
\end{definition}

As we will see later, these conditions suffice to show that $\sS_\alpha$ is a \emph{complete subsystem} (Definition~\ref{def:completesubsys}) of $\sS_\gamma$ for each $\alpha < \gamma < \delta$. This implies in particular that an $\sS_\alpha$ extension is always intermediate to an $\sS_\gamma$ extension. This is a crucial implicit part of most iteration arguments. For instance, when at stage $\alpha$ we chose $\dot \sT_\alpha$ specifically to perform some task over the $\sS_\alpha$-extension, we want that in the final iterated extension this really occurred, meaning that indeed there is a corresponding intermediate $\sS_{\alpha +1} = \sS_{\alpha} * \dot \sT_\alpha$ extension.

\subsection{Ideal support iterations}

\begin{definition}\label{def:Isupportiteration}
    Let $\delta$ be an ordinal and $\mathcal{I}$ be an ideal on $\delta$ containing all finite sets. Let $\langle \sS_{\alpha}, \dot \sT_\alpha : \alpha < \delta \rangle$ be a pre-iteration. Then it is called an $\mathcal{I}$-support iteration, if for each limit $\alpha < \delta$: 
    \begin{enumerate}
        \item $\sF_\alpha$ is generated by $\bar H$ as in (1)(c) of Definition~\ref{def:preiteration} such that $\supp(\bar H) \in \mathcal{I}$,
        \item $\mathbb{P}_\alpha$ consists exactly of those $\bar p$ as in (1)(a) of Definition~\ref{def:preiteration} such that 
            \begin{enumerate}
            \item $\supp(\bar p) \in \mathcal{I}$ and
            \item there is a generator $\bar H$ as above such that $\bar H \restriction \beta \leq \sym_{\sS_\beta}(\dot p(\beta))$, for each $\beta < \alpha$,
        \end{enumerate}
        \item $\sG_\alpha$ consists exactly of those $\bar \pi$ as in (1)(b) of Definition~\ref{def:preiteration} such that 
            \begin{enumerate}
            \item $\supp(\bar \pi) \in \mathcal{I}$ and
            \item there is a generator $\bar H$ as above such that $\bar H \restriction \beta \leq \sym_{\sS_\beta}(\dot \pi(\beta))$, for each $\beta < \alpha$.
        \end{enumerate}
    \end{enumerate}
\end{definition}

Together with the definition of a pre-iteration this tells us exactly how to construct $\mathcal{I}$-support iterations by telling us what to do at limit steps. Note immediately by a simple induction that an $\mathcal{I}$-support iteration is an iteration as per Definition~\ref{def:iteration}. We speak of a \emph{finite}-, \emph{countable}-, or more generally a $<\kappa$-support iteration when we mean an $\mathcal{I}$-support iteration for $\mathcal{I}$ the ideal of finite, countable, or $<\kappa$-sized subsets of $\delta$, respectively.

In the case of finite support iterations the inclusion of condition (b) in (2) and (3) is redundant of course and similarly for countable support iterations of systems with $\sigma$-complete filters. In general, this is a quite important condition though. In the case of requiring for the generators that $\bar H \restriction \beta \leq \sym(\dot H(\beta))$ (already included in Definition~\ref{def:preiteration}), this is needed for $\bar H$ to represent a group at all. (3)(b) is crucial in showing that $\sF_\alpha$ as defined above results in a normal filter (see more in Lemma~\ref{lem:Isupport}). Seemingly, (2)(b) is less relevant and in principle there is nothing wrong with removing this condition. But it likely does not result in what one wants of an iteration in general. For instance, consider a condition $\bar p = \langle \dot p(n) : n < \omega \rangle$ in an $\omega$-length iteration. After passing to the first symmetric extension, one would like to think of $\bar p = \langle \dot p(n) : 0< n < \omega \rangle$ as a condition in some sort of tail of the iteration. Each $\dot p(n)$ is symmetric as witnessed by a group $\bar H^n \in \sF_n$, but there is no reason to assume that e.g. $\bigcap_{n} H^n(0) \in \sF_0$. So why should the tail of the condition even be an element of this first extension in a meaningful way?

As it stands, the above is just a definition. To really make sense of it in practice though one needs to check that when constructing the limit steps as above one really obtains a symmetric again. First, the following properties are easily checked by induction as in Lemma~\ref{lem:preiteration}.

\begin{lemma}\label{lem:isupport1}
Let $\langle \sS_{\alpha}, \dot \sT_\alpha : \alpha < \delta \rangle$ be a $\mathcal{I}$-support iteration as above and $\alpha < \delta$ arbitrary.
\begin{enumerate}
    \item All $\bar p \in \mathbb{P}_\alpha$, $\bar \pi \in \sG_\alpha$, $\bar H \in \sF_\alpha$ have support in $\mathcal{I}$. In fact, $\sF_\alpha$ is generated exactly by those $\bar H$, with $\supp(\bar H) \in \mathcal{I}$ and $\bar H \restriction \beta \leq \dot H(\beta)$ for each $\beta$. Similarly for $\sG_\alpha$ and $\mathbb{P}_\alpha$.\footnote{This is of course trivial and part of the definition for limit $\alpha$, but in the successor case we must use that $\mathcal{I}$ contains all singletons and since this is not the literal definition, we state this as a separate fact.}
    \item For any of the generators $\bar H, \bar K \in \sF_\alpha$, $\bar E = \langle\dot E(\beta) : \beta \leq \alpha \rangle \in \sF_\alpha$ is a generator, where $\dot E(\beta)$ is the canonical closed $\sS_\beta$ name for $\dot H(\beta) \cap \dot K(\beta)$ for each $\beta < \alpha$, $\supp(\bar E) \subseteq \supp(\bar H) \cup \supp(\bar K)$ and $\bar E \leq \bar H \cap \bar K$.
\end{enumerate}
\end{lemma}

\begin{lemma}\label{lem:Isupport}
    Suppose that $\langle \sS_\beta , \dot \sT_{\beta} : \beta < \alpha \rangle$ is an $\mathcal{I}$-support iteration and let $\sS_\alpha = (\mathbb{P}_\alpha, \sG_\alpha, \sF_\alpha)$ be exactly as described in (1), (2) and (3) of Definition~\ref{def:Isupportiteration}. Then $\sS_\alpha$ is a symmetric system.
\end{lemma}

\begin{proof}
    $\mathbb{P}_\alpha$ is of course just a forcing poset. We will check that each $\bar \pi$ indeed represents an automorphism of $\mathbb{P}_\alpha$, each $\bar H$ represents a group and that $\sF_\alpha$ is a normal filter. 
    
    First, let us note a few things though. Suppose $\bar \pi, \bar H$ are as in (3)(b). Then, using (6) of Lemma~\ref{lem:preiteration} and the fact that $\sym(\sigma^{-1}(\dot x)) = \sigma^{-1} \sym(\dot x) \sigma$ in any symmetric system, conclude that \begin{align*}
        \bar K &:= \langle( \bar \pi \restriction \beta)^{-1}\left(\dot \pi(\beta)^{-1} \dot H(\beta) \dot \pi(\beta)\right) : \beta < \alpha \rangle\\
        &= \bigcup_{\beta < \alpha} (\bar \pi \restriction \beta)^{-1} (\bar H \restriction \beta) (\bar \pi \restriction \beta)
    \end{align*} is also supposed to represent a generator of $\sF_\alpha$.\footnote{In the union above we replace the group $(\bar \pi \restriction \beta)^{-1} (\bar H \restriction \beta) (\bar \pi \restriction \beta)$ of the system $\sS_\beta$ with the corresponding sequence as computed in Lemma~\ref{lem:preiteration}.} Moreover, if we define \[\bar \tau := \langle (\bar \pi \restriction \beta)^{-1}(\dot \pi(\beta)^{-1}) : \beta < \alpha \rangle = \bigcup_{\beta < \alpha} (\bar \pi \restriction \beta)^{-1},\] then $\bar K \restriction \beta \leq \dot \tau(\beta)$, for each $\beta < \alpha$, and $\supp(\bar \tau) = \supp(\bar \pi)$. In particular, $\bar \tau$ is supposed to represent an automorphism in $\sG_\alpha$, and in fact an element of $\bar H$.  Finally, using (5) of Lemma~\ref{lem:preiteration}, if $\bar E \leq \bar H \cap \bar K$ is defined as in Lemma~\ref{lem:isupport1}, then \begin{align*} \bar L &:= \langle \dot \pi(\beta) \bar \pi \restriction \beta(\dot H(\beta)) \dot \pi(\beta)^{-1} : \beta < \alpha \rangle \\ &= \bigcup_{\beta < \alpha} (\bar \pi \restriction \beta)(\bar E \restriction \beta) (\bar \pi \restriction \beta)^{-1}\end{align*}
    represents a generator in $\sF_\alpha$.

    Let us check that $\bar \pi$ indeed represents an automorphism of $\mathbb{P}_\alpha$. Let $\bar p \in \mathbb{P}_\alpha$ be arbitrary and without loss of generality, using Lemma~\ref{lem:isupport1}, assume that also $\bar H \restriction \beta \leq \sym(\dot p(\beta))$ for all $\beta$. Then $\bar \pi(\bar p) = \langle \dot \pi(\beta) \bar \pi \restriction \beta (\dot p(\beta)) : \beta < \alpha \rangle$ and $\supp(\bar \pi (\bar p)) = \supp(\bar p) \in \mathcal{I}$. Also $\bar L \restriction \beta \leq \sym(\dot \pi(\beta) \bar \pi \restriction \beta (\dot p(\beta)))$ for each $\beta$, so $\bar \pi(\bar p) \in \mathbb{P}_\alpha$. The fact that $\bar \pi$ is order preserving follows immediately from how the order on $\mathbb{P}_\alpha$ and the application of $\bar \pi$ are defined in the limit step and the fact that each $\bar \pi \restriction \beta$ is order preserving. Similarly, $\bar \tau$ represents the inverse of $\bar \pi$ as $\bar \tau \restriction \beta$ is the inverse of $\bar \pi \restriction \beta$ for each $\beta$.

    Checking that each $\bar H$ represents a group is just as straightforward now. Given $\bar \pi \in \sG_\alpha$ and $\bar H \in \sF_\alpha$ arbitrary we can always assume without loss of generality that $\bar \pi$ and $\bar H$ are as in (3)(b) and from there we can construct $\bar L \leq \bar \pi H \bar \pi^{-1}$, showing that the filter is normal (or even more directly $\bar K = \bar \pi^{-1} H \bar \pi$ which is sufficient).  \end{proof}

\subsection{Factorization}

One can show that an $\mathcal{I}$-support iteration can indeed be factorized into initial segments and tails that are themselves $\mathcal{I}$-support iterations. This type of factorization is not really a crucial part in most iterated forcing arguments but it can become relevant sometimes. To simplify notation here let us speak of an $\mathcal{I}$-support iteration when in fact we mean the ideal generated by $\mathcal{I}$.\footnote{In some intermediate extension $\mathcal{I}$ might not be closed under taking subsets.}

Suppose that $\langle \sS_{\alpha}, \dot \sT_\alpha : \alpha \leq \delta \rangle = \langle \sS_{\alpha}, \dot \sT_\alpha : \alpha < \delta + 1 \rangle$ is an $\mathcal{I}$-iteration and $\alpha \leq \delta$ is fixed. By recursion on $\delta \geq \alpha$, one defines an $\sS_{\alpha}$-name $\langle \dot \sS_{\alpha, \gamma}, \dot \sT_{\alpha, \gamma} : \gamma \in [\alpha, \delta] \rangle^\bullet$ for an $\mathcal{I}$-support iteration that naturally corresponds to the tail of the iteration. Simultaneously, one defines for each $\sS_\delta$-name $\dot x$, an $\sS_\alpha$-name $[\dot x]_{\alpha, \delta}$ for an $\dot \sS_{\alpha, \delta}$-name, and similarly, for each $\sS_\alpha$-name $\dot y$ for an $\dot \sS_{\alpha, \delta}$-name, an $\sS_\delta$-name $]\dot y[_{\alpha, \delta}$. Further, for any $\bar p \in \mathbb{P}_\delta$, $\bar \pi \in \sG_\delta$ and a generator $\bar H \in \sF_\delta$ as above, one defines $\sS_\alpha$-names $[\bar p \restriction [\alpha, \delta)]$, $[\bar \pi \restriction [\alpha, \delta)]$ and $[\bar H \restriction [\alpha, \delta)]$ for respective objects in the system $\dot \sS_{\alpha, \delta}$. Before giving the exact construction of these let us state the factorization theorem that is analoguous to Theorem~\ref{lem:twostepfactor}. 

\begin{thm}[Factorization for iterations]\label{thm:factorizationiteration}
Whenever $G$ is $\sS_\alpha$-generic over $V$ and $H$ is $\dot \sS_{\alpha, \delta}^G$-generic over $V[G]_{\sS_\alpha}$, then $G * H = \{ \bar p : \bar p \restriction \alpha \in G \wedge [\bar p \restriction [\alpha, \delta)]^G \in H\}$ is $\sS_{\delta}$-generic over $V$. Similarly, whenever $K$ is $\sS_\delta$-generic over $V$, then $K = G * H$, where $G = \{ \bar p \restriction \alpha : \bar p \in K \}$ is $\sS_\alpha$-generic over $V$ and the upwards closure $H$ of $\{ [\bar p \restriction [\alpha, \delta)]^G : \bar p \in K \}$ is $\dot \sS_{\alpha, \delta}^G$-generic over $V[G]_{\sS_\alpha}$.

In either case, $([\dot x]^G)^H = \dot x^{G * H}$ for every $\sS_\delta$-name $\dot x$ and $]\dot y[^{G*H} = (\dot y^G)^H$ for every $\sS_\alpha$-name $\dot y$ for an $\dot \sS_{\alpha, \delta}$-name. In particular, $V[G*H]_{\sS_\delta} = V[G]_{\sS_\alpha}[H]_{\dot \sS_{\alpha, \delta}^G}$.
\end{thm}

The recursive construction proceeds as follows: For $\delta = \alpha$, we let $\dot \sS_{\alpha, \delta}$ be a name for the trivial system $(\{ \mathds 1\}, \{\id \}, \{ \{\id \} \})$. $[\bar p \restriction [\alpha, \alpha)]$ is simply a name for $\mathds 1$. At each step $\delta$, by recursion on the rank of names, we define \[[\dot x]_{\alpha, \delta} = \{ (\bar p \restriction \alpha, ([\bar p \restriction [\alpha, \delta)], [\dot z]_{\alpha, \delta})^\bullet) : (\bar p, \dot z) \in \dot x \},\] and similarly, \[] \dot y[_{\alpha, \delta} = \{ (\bar p, ]\dot z[_{\alpha, \delta}) : \bar p \restriction \alpha \Vdash_{\sS_\alpha} ([\bar p \restriction [\alpha, \delta)], \dot z)^\bullet \in \dot y \}.\]

For any $\delta$, we let $\dot \sT_{\alpha, \delta} = [\dot \sT_\delta]_{\alpha, \delta}$. For $\delta = \gamma +1$, we define \[[\bar p \restriction [\alpha, \delta)] = \cl([\bar p \restriction [\alpha, \gamma)]^\frown [\dot p(\gamma)]_{\alpha, \gamma})^\bullet,\] and for $\delta$ limit, \[[\bar p \restriction [\alpha, \delta)] = \bigcup_{\gamma \in [\alpha, \delta)} [\bar p \restriction [\alpha, \gamma)].\] Similarly for $[\bar \pi \restriction [\alpha, \delta)]$ and $[\bar H \restriction [\alpha, \delta)]$. 

That these objects indeed work as described in the paragraph before the theorem and that Theorem~\ref{thm:factorizationiteration} holds is verified inductively. We leave the details to the reader as this is mostly just a lengthy formal verification (see also \cite{HolySchilhan} for the case of finite support iterations). The successor step is essentially as in the proof of Theorem~\ref{thm:factorizationiteration}. Let us just remark that (2)(b) and (3)(b) are crucial here for $[\bar p \restriction [\alpha, \delta)]$, $[\bar \pi \restriction [\alpha, \delta)]$ and $[\bar H \restriction [\alpha, \delta)]$ to be $\sS_\alpha$-names. For example, if $\bar H$ and $\bar p$ are as in (2)(b) say, one shows that $\bar H\restriction \alpha \leq \sym([\bar p \restriction [\alpha, \delta)])$.

\subsection{Preservation theorems}

Let us give just two very simple cases of preservation results to exemplify the technique. 

\begin{prop}
    Let $\langle \sS_{\alpha}, \dot \sT_\alpha : \alpha \leq \delta \rangle$ be a $<\kappa$-support iteration, where $\kappa$ is a regular cardinal, $\DC_{<\kappa}$ holds and the filter of each iterand is forced to be $<\kappa$-complete. Then the filter of $\sS_\delta$ is $<\kappa$-complete.
\end{prop}

\begin{proof}
 By $\DC_{<\kappa}$ we may without loss of generality consider sequences $\langle \bar H^\xi : \xi < \gamma \rangle$, $\gamma < \kappa$, of generators instead of arbitrary elements of $\sF_\alpha$. Define $\bar H$ inductively by defining $\dot H(\alpha)$ to be the canonical closed $\sS_\alpha$-name for the intersection of $\langle \dot H^\xi(\alpha) : \xi < \gamma \rangle$, which has an $\sS_\alpha$-name with symmetry group $\bar H \restriction \alpha$. Also note that $\supp(\bar H) \subseteq \bigcup_{\xi < \gamma} \supp(\bar H^\xi)$, which by regularity of $\kappa$ has size $< \kappa$.
\end{proof}

The proof implicitly uses that there is always a canonical witness to each instance of the completeness of a filter, namely simply taking the intersection of a given collection of filter elements. On the other hand this is not always the case for other $\forall\exists$-properties of symmetric systems and some sort of uniformity of witnesses is needed.

\begin{definition}
    We say that $F$ witnesses the $<\kappa$-closure of $\mathbb{P}$ if $F \colon \mathbb{P}^{<\kappa} \to \mathbb{P}$ is such that for any decreasing sequence $\langle p_\xi : \xi < \gamma \rangle$, $\gamma < \kappa$, in $\mathbb{P}$, $F(\langle p_\xi : \xi < \gamma \rangle)$ is a common extension of these conditions.
\end{definition}

\begin{prop}
    Let $\langle \sS_{\alpha}, \dot \sT_\alpha : \alpha \leq \delta \rangle$ be a $<\kappa$-support iteration, $\kappa$ regular and assume that $\DC_{<\kappa}$ holds. Assume that the iterand's posets are uniformly $<\kappa$-closed, meaning that there is a sequence $\langle \dot F_\alpha : \alpha \leq \delta \rangle$ and $\bar H \in \sF_\delta$ such that $\dot F_\alpha$ is an $\sS_\alpha$-name for a witness to the $<\kappa$-closure of the poset of $\dot \sT_\alpha$ and $\bar H \restriction \alpha \leq \dot F_\alpha$, for each $\alpha < \delta$. Then the poset $\mathbb{P}_\delta$ of $\sS_\delta$ is $< \kappa$-closed and has a witnessing function.
\end{prop}

We leave the proof to the reader as there is nothing substantially more difficult than in the previous proof. The following question, however, is the real goal.

\begin{quest}
    Let $\langle \sS_{\alpha}, \dot \sT_\alpha : \alpha \leq \delta \rangle$ be a countable support iteration, $\delta$ limit, and suppose that for each $\alpha < \delta$, $\Vdash_{\sS_\alpha} \DC$. Does $\Vdash_{\sS_\delta} \DC$? If the answer is negative, is there a combinatorial condition which is equivalent to the preservation of $\DC$?
\end{quest}
It is not entirely clear how to approach such a question, even in the case where $\delta$ has an uncountable cofinality.

We conclude this section with a much broader question which should be of significant interest.

\begin{quest}
    Let $\varphi$ be a weak choice principle, e.g.\ the Boolean Prime Ideal theorem or $\AC_\omega$, and let $\langle \sS_\alpha,\dot\sT_\alpha : \alpha\leq\delta\rangle$ be an $\mathcal{I}$-support iteration for some ideal $\mathcal{I}$ such that for each $\alpha<\delta$, $\Vdash_{\sS_\alpha}\varphi$. Are the natural and non-trivial conditions on $\sT_\alpha$ and $\mathcal{I}$ which guarantee $\Vdash_{\sS_\delta}\varphi$?
\end{quest}

For example, $\AC_\omega$ is preserved under $\sigma$-distributive forcings. So, can we find a condition on $\mathcal{I}$ which guarantee the preservation of $\AC_\omega$ at limit steps as well? In the case of the Boolean Prime Ideal theorem, we essentially do not have any good combinatorial condition for this preservation, which would be a strong first step towards this solution.

\section{Equivalence of symmetric systems}\label{sec:equiv}

In the context of forcing there are two notions of equivalence between posets $\mathbb{P}$ and $\mathbb{Q}$ that are typically being considered. The first one, which we call \emph{weak equivalence}, simply says that $\mathbb{P}$ and $\mathbb{Q}$ produce exactly the same extensions. The second one, \emph{strong equivalence} says that $\mathbb{P}$ and $\mathbb{Q}$ have isomorphic complete Boolean algebras and is much stronger. For example, all atomic posets are equivalent in the weak sense, while the Boolean algebras they generate might be different, or less trivially, the lottery sum of $|\mathbb{P}|^+$ many copies of a forcing $\mathbb{P}$ will produce the same extensions, but their Boolean completions will be very different indeed. The strong notion is more ``categorical'', in the sense that typical forcing constructions, such as iterations, do not depend on the particular choice of representative (at least assuming $\AC$). This is not the case for weak equivalence. For instance, a finite support product of trivial forcing notions may or may not be trivial. 

While generalizing weak equivalence to symmetric systems is a trivial task this is much less clear for strong equivalence.

\begin{definition}
    Two symmetric systems $\sS$ and $\sT$ are \emph{weakly equivalent} if for any $\sS$-generic $G$, there is a $\sT$-generic $H$ such that $V[G]_\sS = V[H]_\sT$ and vice-versa.
\end{definition}

A priori, this definition is external to $V$. But it turns out that the weak equivalence of $\sS$ and $\sT$ can be expressed inside $V$. This follows from results of Section~\ref{sec:quotients}. More specifically, by Theorem~\ref{thm:quotsummary} a model $M$ is a symmetric extension of $V$ by $\sT$ iff there is a forcing notion in $M$ adding a generic filter $H$ over $V$ such that $M = V[H]_\sT$. To express the equivalence of $\sS$ and $\sT$ in $V$ we may thus say that $\Vdash_\sS$ ``there is a forcing notion adding a $\sT$-generic $H$ over $V$ such that $V[\dot G]_\sS = V[H]_\sT$'' and vice-versa for $\sT$.\footnote{If we want to be very precise with this, we can say that there is a closed and unbounded class of $\alpha$ such that we can force over $V_\alpha$ to satisfy the equation. This can be simplified, of course, by introducing a predicate for the ground model, $V$.}

\subsection{Classes of names}

Often it makes sense to consider names that are not outright hereditarily symmetric. 

\begin{definition}
    We say that $\C$ is a class of names in symmetric systems if $\C$ consists of pairs $(\sS, \dot x)$, where $\sS = (\mathbb{P}, \sG, \sF)$ is a symmetric system, $\dot x$ is a $\mathbb{P}$-name and $\Vdash_{\mathbb{P}} \dot x \in V[\dot G]_\sS$. We write $\C_\sS$ for the class of $\dot x$ such that $(\sS,\dot x) \in \C$ and typically identify $\C$ and $\C_\sS$ in context where $\sS$ is clear. $\SN$ is the maximal class of names, i.e., $\dot x \in \SN_\sS$ exactly when $\Vdash_{\mathbb{P}} \dot x \in V[\dot G]_\sS$. 
\end{definition}

The most important example is of course $\HS$, the class of hereditarily symmetric names.

\begin{definition}
    $\C$ is forcing stable if for any $\sS$, any $\dot x \in \C_\sS$ and any $\dot y \in \SN_\sS$, if $\Vdash \dot x = \dot y$, then also $\dot y \in \C_\sS$. 
\end{definition}

\begin{definition}
Let $\sS = (\mathbb{P}, \sG, \sF)$ be a symmetric system. Then a $\mathbb{P}$-name $\dot x$ is $\sS$-\emph{respected} if \[\res_\sS(\dot x) := \{ \pi \in \sG : ~ \Vdash_\sS \pi(\dot x) = \dot x \} \in \sF.\]
\end{definition}

We may write $\res_\sG(\dot x)$. The class of respected names $\R$ is clearly an example of a forcing stable class. 

\begin{definition}
        For a class $\C$ we define its hereditary class $\HC$ recursively, where $\dot x \in \HC_\sS$ iff $\dot x \in \C_\sS$ and for every $(p,\dot y) \in \dot x$, $\dot y \in \HC_\sS$. We say that $\HC$ is hereditarily forcing stable if it is the hereditary class of a forcing stable $\C$.
\end{definition}

Of course, $\HS$ is the hereditary class of the symmetric names. The class of \emph{hereditarily respected names} $\HR$ is hereditarily forcing stable.

The following is immediate to check.

\begin{lemma}
    $\HS \subseteq \HR \subseteq \SN$.
\end{lemma}

\begin{lemma}\label{lem:hshr}
    There is a uniformly definable map, with parameter $\sS$, $i \colon \HR_\sS \to \HS_\sS$ such that $\Vdash \dot x = i(\dot x)$, for every $\dot x \in \HR_\sS$.
\end{lemma}

\begin{proof}
 The map is defined recursively such that \[i(\dot x) = \bigcup_{\pi \in \res(\dot x)} \pi(\{ (p, i(\dot y)) : (p,\dot y) \in \dot x \}).\qedhere\]
\end{proof}

\subsection{Notions of strong equivalence}

The following definition can formally be thought of as taking place in a conservative extension of $\ZF$ such as $\NBG$ without choice, but can be expressed in a first-order way as we shall see later.

\begin{definition}\label{def:strequi}
 Let $\C$ be a class of names in symmetric systems containing all check names. Then $\sS$ and $\sT$ are called \emph{$\C$-equivalent} and we write $\sS \cong_\C \sT$, if there are maps $i \colon \C_{\sS} \to \C_{\sT}, i^* \colon \C_{\sT} \to \C_{\sS}$, such that 
 \begin{enumerate}
\item for any $x$, $\Vdash_\sT i(\check{x}) = \check x$, and vice-versa $\Vdash_\sS i^*(\check{x}) = \check x$,\footnote{$\check{x}$ always denotes the check name in the system that is clear from the context.}
  \item for any $\dot x \in \C_\sS$, $\Vdash_\sS i^*(i(\dot x)) = \dot x$ and vice-versa, for any $\dot x \in \C_\sT$,\[\Vdash_\sT i(i^*(\dot x)) = \dot x,\]
  \item for any $\Delta_0$-formula $\varphi$ in the language of set theory and any $\dot x_0, \dots, \dot x_n \in \C_\sS$, \[\Vdash_\sS \varphi(\dot x_0, \dots, \dot x_n) \text{ iff } \Vdash_\sT \varphi(i(\dot x_0), \dots, i(\dot x_n)),\] and vice-versa for any $\dot x_0, \dots, \dot x_n \in \C_\sT$, \[\Vdash_\sT \varphi(\dot x_0, \dots, \dot x_n) \text{ iff } \Vdash_\sS \varphi(i^*(\dot x_0), \dots, i^*(\dot x_n)).\] 
 \end{enumerate}
\end{definition}

A rote verfication shows that this is indeed an equivalence relation between symmetric systems.

\begin{lemma}
    Let $\HS \subseteq \C$ and let $i,i^*$ as above. Then (3) already holds for any formula $\varphi$ in the language of set theory. 
\end{lemma}

\begin{proof}
    This is a consequence of the reflection theorem. Let $\mathbb{P}$ and $\mathbb{Q}$ be the forcing notions of $\sS$ and $\sT$ respectively. Note that for any $\mathbb{P}$-generic $G$, $V[G]_\sS = \bigcup_{\alpha\in\Ord}\HS_\alpha^G $ where the union is increasing and continuous. Also, for any $\sT$-name $\dot x$, there is $\alpha$ so that $\Vdash_{\sS} i^*(\dot x) \in \HS_\alpha$ and thus $\Vdash_\sT \dot x \in i(\HS_\alpha)$. Consequently, for any $\mathbb{Q}$-generic $H$ also $\bigcup_{\alpha\in\Ord}i(\HS_\alpha)^H = V[H]_\sT$ and the union is increasing and continuous. Then for any $\varphi(\dot x_0, \dots, \dot x_n)$, we can find clubs of ordinals $S_{0}, S_{1}$ such that \[\forall \alpha \in S_0 \left[\Vdash_\sS \varphi(\dot x_0, \dots,  \dot x_n) \text{ iff } \Vdash_\sS \HS_\alpha \models \varphi(\dot x_0, \dots, \dot x_n)\right]\] and \[\forall \alpha \in S_1 \left[\Vdash_\sT \varphi(i(\dot x_0), \dots,  i(\dot x_n)) \text{ iff } \Vdash_\sT i(\HS_\alpha) \models \varphi(i(\dot x_0), \dots,  i(\dot x_n))\right].\]
    Thus if $\alpha \in S_0 \cap S_1$, \begin{align*}
        \Vdash_\sS \varphi(\dot x_0, \dots,  \dot x_n) &\text{ iff } \Vdash_\sS \HS_\alpha \models \varphi(\dot x_0, \dots, \dot x_n) \\ &\text{ iff } \Vdash_\sT i(\HS_\alpha) \models \varphi(i(\dot x_0), \dots,  i(\dot x_n)) \\ &\text{ iff } \Vdash_\sT \varphi(i(\dot x_0), \dots,  i(\dot x_n)).
    \end{align*}
    
    Similarly in the other direction.\end{proof}

\begin{prop}
    Let $\HS \subseteq \C$. Then $\sS \cong_\C \sT$ implies that $\sS$ and $\sT$ are weakly equivalent.
\end{prop}

\begin{proof}
     By Theorem~\ref{thm:quotsummary}, there is an $\sS$-name for a forcing notion that forces that its ground model $V^\sS$ is of the form $\check{V}[G]_{\check{\sS}}$ for some $\check{\mathbb{P}}$-generic $G$ over $\check V$. This can be transfered to $\sT$. Similarly in the other direction.
\end{proof}

\begin{lemma}\label{lem:hseqhr}
    $\sS \cong_\HR \sT$ is equivalent to $\sS \cong_\HS \sT$.
\end{lemma}

\begin{proof}
    Use Lemma~\ref{lem:hshr}.
\end{proof}

\begin{prop}\label{prop:cBaequi}
For any system $\sS$ there is $\sS'$, uniformly definable from $\sS$, such that $\sS' \cong_\HS \sS$ and $\sS' \cong_\SN \sS$ and the forcing of $\sS'$ is a complete Boolean algebra.
\end{prop}

\begin{proof}
Let $\sS = (\mathbb{P}, \sG, \sF)$. Then let $\mathcal{B}(\mathbb{P})$ be the corresponding complete Boolean algebra of regular open subsets of $\mathbb{P}$. Any automorphism $\pi$ of $\mathbb{P}$ can be identified with the unique automorphism on $\mathcal{B}(\mathbb{P})$ that agrees with $\pi$ on the dense copy of $\mathbb{P}$ within $\mathcal{B}(\mathbb{P})$. Under this identification we can consider the system $\sS' = (\mathcal{B}(\mathbb{P}), \sG, \sF)$.

Viewing $\mathbb{P}$ as a dense subset of $\mathcal{B}(\mathbb{P})$, let $i\colon \HS_\sS \to \HS_{\sS'}$ be the identity. On the other hand, we define \[i^*(\dot x) = \{ (p, i^*(\dot y)) : p \in \mathbb{P} \wedge \exists q \geq p ( (q, \dot y) \in \dot x)  \}\] by recursion on the rank of $\dot x$. (1), (2) and (3) follow from well-known facts about forcing equivalence and the forcing theorem. The same definition works for $\SN$.
\end{proof}

For the sole purpose of the following theorem let us say that a hereditary class $\HC$ is good if $\HS \subseteq \HC$, $\HC$ is hereditarily forcing stable and $\sS \cong_\HC \sS'$, where $\sS'$ is as in Proposition~\ref{prop:cBaequi}. Then we can formulate the following general result. 

\begin{thm}\label{thm:tenequi}
Let $\sS = (\mathbb{P}, \sG, \sF)$ and $\sT = (\mathbb{Q}, \sH, \sE)$ be symmetric systems. Let $\HC$ be either good and $\sS, \sT$ be tenacious, or $\HC = \SN$. Then $\sS \cong_\HC \sT$ iff there is an isomorphism $\pi \colon \B(\mathbb{P}) \to \B(\mathbb{Q})$ such that $\pi``\HC_{\sS'} = \HC_{\sT'}$. In particular, $\mathbb{P}$ and $\mathbb{Q}$ are forcing equivalent. Moreover, whenever $G$ is $\B(\mathbb{P})$-generic over $V$, then $V[G \cap \mathbb{P}]_{\sS} = V[\pi`` G \cap \mathbb{Q}]_{\sT}$ and vice-versa.
\end{thm}
 
\begin{proof}
    First suppose $\pi$ as above exists and we let $i(\dot x) = \pi(\dot x)$, $i^*(\dot x) = \pi^{-1}(\dot x)$, then this clearly witnesses that $\sS' \cong_\HC \sT'$. Since $\sS \cong_\HC \sS'$ and $\sT \cong_\HC \sT'$, also $\sS \cong_\HC \sT$. On the other hand, assume that $\sS \cong_\HC \sT$ and thus ${\sS}' \cong_\HC {\sT}'$, as witnessed by $i$ and $i^*$. Note that there is a dense set $D$ of conditions $p \in \mathbb{P}$ such that $\dot x_p := \{(p, \check{p}) \} \in \HC_{{\sS'}}$. For any $p \in \B(\mathbb{P})$, let \[\pi(p) := \bigvee_{p' \leq p, p' \in D} \llbracket i(\dot x_{p'}) \neq \emptyset \rrbracket \in \B(\mathbb{Q}).\] Similarly define $\pi^* \colon \B(\mathbb{Q}) \to \B(\mathbb{P})$ using $i^*$. Note that for any $p \in \B({\mathbb{P}})$ and any $\varphi$, \begin{align*}
p \Vdash_{{\sS}'} \varphi(\dot x_0, \dots, \dot x_n)  &\text{ iff } \forall p' \leq p, p' \in D ( p' \Vdash_{{\sS}'} \varphi(\dot x_0, \dots, \dot x_n)) \\ &\text{ iff } \forall p' \leq p, p' \in D (\Vdash_{{\sS}'} \dot x_{p'} \neq \emptyset \rightarrow \varphi(\dot x_0, \dots, \dot x_n)) \\ &\text{ iff } \forall p' \leq p, p' \in D (\Vdash_{{\sT}'} i(\dot x_{p'})\neq \emptyset \rightarrow \varphi(i(\dot x_0), \dots, i(\dot x_n))) \\ &\text{ iff } \forall p' \leq p, p' \in D (\llbracket i(\dot x_{p'})\neq \emptyset \rrbracket \Vdash_{{\sT}'}\varphi(i(\dot x_0), \dots, i(\dot x_n))) \\ &\text{ iff } \pi(p) \Vdash_{\sT'} \varphi(i(\dot x_0), \dots, i(\dot x_n)).\end{align*}

Similarly for $\pi^*$ and $i^*$. Then note that for any $p\in \B(\mathbb{P}), p' \in D$, 
\begin{align*}
p \perp p' &\text{ iff } p \Vdash_{{\sS}'} \dot x_{p'} = \emptyset \\ &\text{ iff } \pi(p) \Vdash_{\sT'} i(\dot x_{p'}) = \emptyset \\ &\text{ iff } \pi^*(\pi(p)) \Vdash_{\sS'} i^*(i(\dot x_{p'})) = \dot x_{p'} = \emptyset \\ &\text{ iff } \pi^*(\pi(p)) \perp p'.\end{align*}

Thus $p = \pi^*(\pi(p))$. Similarly $\pi(\pi^*(q)) = q$, for every $q \in \B(\mathbb{Q})$. So $\pi^* = \pi^{-1}$ and $\pi$ is indeed an automorphism. Let us check by induction on the ranks of $\dot x \in \HC_{{\sS}'}$ and $\dot y \in \HC_{\sT'}$ that $\Vdash_{\sT'} \pi(\dot x) = i(\dot x)$ and $\Vdash_{\sS'} \pi^{-1}(\dot y) = i^*(\dot y)$. 

Suppose for instance that $q \Vdash_{\sT'} \dot z \in \pi(\dot x)$, where $\dot z$ is of lower rank than $\dot x$. Then $\pi^{-1}(q) \Vdash_{\sS'} \pi^{-1}(\dot z) \in \dot x$, and by the inductive assumption $\pi^{-1}(q) \Vdash_{\sS'} i^*(\dot z) \in \dot x$. Then, by our previous observation, $q \Vdash_{\sT'} i(i^*( \dot z)) = \dot z \in i(\dot x)$.

On the other hand, if $q \Vdash_{\sT'} \dot z \in i(\dot x)$, then $\pi^{-1}(q) \Vdash_{\sS'} i^*(\dot z) \in \dot x$ and by extending $q$ if necessary, we can assume that $\pi^{-1}(q) \Vdash_{\sS'} i^*(\dot z) = \dot z' \in \dot x$ for some $\dot z'$ of lower rank than $\dot x$. Then again, $q \Vdash_{\sT'} i(i^*(\dot z)) = \dot z = i(\dot z') = \pi(\dot z')$. At the same time, $q\Vdash_{\sT'} \pi(\dot z') \in \pi(\dot x)$ and so $q \Vdash_{\sT'} \dot z \in \pi(\dot x)$.

Thus indeed, $\Vdash_{\sT'} \pi(\dot x) = i(\dot x)$ and the argument for $\Vdash_{\sS'} \pi^{-1}(\dot y) = i^*(\dot y)$ is completely analogous. Since $\HC$ is hereditary forcing stable an easy induction shows that $\pi``\HC_{{\sS'}} \subseteq \HC_{{\sT}'}$ and $\pi^{-1}``\HC_{{\sT}'} \subseteq \HC_{{\sS}'}$. Thus $\HC_{{\sT}'} = \pi\pi^{-1}``\HC_{{\sT}'} = \pi``\HC_{{\sS}'}$.

For the last statement of the theorem, note that $\pi``G \cap \mathbb{Q}$ is $\mathbb{Q}$-generic whenever $G$ is $\B(\mathbb{P})$-generic as $\pi$ is an isomorphism. Then $i(\dot x)^{\pi`` G \cap \mathbb{Q}} = \pi(\dot x)^{\pi``G} = \dot x^G$ and $i(\dot x)^{\pi``G \cap \mathbb{Q}} \in V[\pi`` G \cap \mathbb{Q}]_\sT$ as $i(\dot x) \in \SN_{\sT}$. Thus $V[G \cap \mathbb{P}]_\sS \subseteq V[\pi`` G \cap \mathbb{Q}]_\sT$ and similarly in the other direction.
\end{proof}

In the setting of classical forcing, the generic filter is uniquely determined by how names are evaluated. The same is still true for tenacious systems and moreover when two such systems are equivalent, there is a one-to-one correspondence between (full) generics on one side and the other, just like in the case of forcing. When the system is not tenacious this generally cannot be the case. For example consider any forcing $\mathbb{P}$ and the lottery sum $\mathbb{P} \oplus \mathbb{P} = \{0,1\} \times \mathbb{P}$ with the involution $\pi(i,p) = (1-i,p)$. Consider the recursively defined maps \[i(\dot x) = \{ (p, i(\dot y)) : \exists i \in 2 (((i,p), \dot z) \in \dot x) \}\] and \[i^*(\dot x) = \{ ((0,p), i^*(\dot z)) : (p,\dot z) \in \dot x \} \cup \{((1,p),i^*(\dot z)) : (p,\dot z) \in \dot x \}.\] It is easy to see that these maps witness that $(\mathbb{P} \oplus \mathbb{P}, \sG, \sF ) \cong \mathbb{P}$, where $\sG = \{\id, \pi \}$, $\sF = \{\sG \}$, and $\mathbb{P}$ is identified with the system that amounts to forcing with $\mathbb{P}$, namely $(\mathbb{P}, \{\id \}, \{\{\id \}\})$. But for any $\mathbb{P}$-generic $G$, $\{0\} \times G$ and $\{1\}\times G$ produce exactly the same evaluations. 

\begin{cor}\label{thm:tenequilr}
    Let $\sS = (\mathbb{P}, \sG, \sF)$ and $\sT = (\mathbb{Q}, \sH, \sE)$ be symmetric systems. Then $\sS \cong_\SN \sT$ iff there is an isomorphism $\pi \colon \B(\mathbb{P}) \to \B(\mathbb{Q})$ so that whenever $G$ is $\B(\mathbb{P})$-generic, then $V[G \cap \mathbb{P}]_{\sS} = V[\pi`` G \cap \mathbb{Q}]_{\sT}$ and vice-versa.
\end{cor}

\begin{proof}
The direction from left to right is already included in Theorem~\ref{thm:tenequi}. On the other hand, note that the statement ``for any $\B(\mathbb{P})$-generic $G$, $V[G \cap \mathbb{P}]_{\sS} = V[\pi`` G \cap \mathbb{Q}]_\sT$'' is equivalent to ``for any $\B(\mathbb{P})$-generic $G$, $V[G]_{\sS'} = V[\pi`` G ]_{\sT'}$'' which is equivalent to saying that for any $\mathcal{B}(\mathbb{P})$-name $\dot x$, ``$\Vdash_{\mathcal{B}(\mathbb{P})} \dot x \in V[\dot G]_{\sS'}$ iff $\Vdash_{\mathcal{B}(\mathbb{Q})} \pi(\dot x) \in V[\dot G]_{\sT'}$'', i.e., $\pi \SN_{\sS'} \subseteq \SN_{\sT'}$. The analoguous statement for $\B(\mathbb{Q})$ implies that $\SN_{\sT'} \subseteq \pi``\SN_{\sS'}$. By Theorem~\ref{thm:tenequi}, $\pi``\SN_{\sS'} =\SN_{\sT'}$ is enough to conclude that $\sS' \cong_\SN \sT'$.
\end{proof}

\begin{prop}\label{prop:equitenacious}
For any system $\sS$ there is a tenacious system $\sT$, uniformly definable from $\sS$, such that $\sT \cong_\HS \sS$.
\end{prop}

\begin{proof}
    Let $\sS = (\mathbb{P}, \sG, \sF)$ and assume without loss of generality that $\mathbb{P}$ is a complete Boolean algebra. Then let $\mathbb{Q} \subseteq \mathbb{P}$ be the subalgebra of $\mathbb{P}$ consisting of conditions $p$, such that $\fix(p) \in \sF$. $\mathbb{Q}$ is then closed under automorphisms from $\sG$, by normality of $\sF$. In other words, for any $\pi \in \sG$, $\pi \restriction \mathbb{Q} \in \Aut(\mathbb{Q})$. So we can consider the system $\sT = (\mathbb{Q}, \sG, \sF)$. Note that $\HS_\sT \subseteq \HS_\sS$.

    To show that $\sS \cong \sT$, for any $\dot x \in \HS_\sS$, define \[i(\dot x) = \{ (q,\dot y) : \dot y \in \HS_\sT, q \in \mathbb{Q}, q \Vdash \dot y \in \dot x  \}\] and let $i^* \colon \HS_\sT \to \HS_\sS$ be the identity.
    
    By induction, $i(\dot x) \in \HS_\sT$, and moreover, $\Vdash_\sS \dot x = i(\dot x)$. Clearly $\Vdash_\sS i(\dot x) \subseteq \dot x$. On the other hand, if $p \Vdash_\sS \dot y \in \dot x $ for some $\dot y \in \HS_\sS$, then from the induction we can assume that $\dot y \in \HS_\sT$. Let $H = \sym(\dot x) \cap \sym(\dot y)$. If $q = \bigvee_{\pi \in H} \pi(p)$, then $q \in \mathbb{Q}$, $p \leq q$ and $(q,\dot y) \in i(\dot x)$.
    
\begin{claim}
    Let $G$ be $\sS$-generic over $V$. Then $G \cap \mathbb{Q}$ is $\sT$-generic.
\end{claim}

\begin{proof}
    Let $D \subseteq \mathbb{Q}$ be symmetrically dense and $H \in \sF$ such that $\pi``D = D$ for every $\pi \in H$. It suffices to show that $D$ is predense in $\mathbb{P}$, since then the downwards closure of $D$ is a symmetrically dense set in $\sS$ and so $G \cap D \neq \emptyset$. So suppose $p \in \mathbb{P}$ is arbitrary. Let $q = \bigvee_{\pi \in H} \pi(p) \in \mathbb{Q}$. Then there is $r \in D$ such that $r$ is compatible with $q$. In particular $r$ is compatible with $\pi(p)$, for some $\pi \in H$. But then $\pi^{-1}(r) \in D$ is compatible with $p$.
\end{proof}

It follows from that claim that whenever $\Vdash_\sT \varphi(\dot x_0, \dots, \dot x_n)$, for $\dot x_0, \dots, \dot x_n \in \HS_\sT$, then also $\Vdash_\sS \varphi(\dot x_0, \dots, \dot x_n)$. On the other hand, suppose that $\Vdash_\sS \varphi(\dot x_0, \dots, \dot x_n)$, for $\dot x_0, \dots, \dot x_n \in \HS_\sS$, but there is $q \in \mathbb{Q}$ such that $q \Vdash_\sT \neg \varphi(i(\dot x_0), \dots, i(\dot x_n))$. Then \begin{align*}q \Vdash_\sT \neg \varphi(i(\dot x_0), \dots, i(\dot x_n)) &\text{ iff } \Vdash_\sT \dot x_q \neq \emptyset \rightarrow \neg \varphi(i(\dot x_0), \dots, i(\dot x_n))\\
&\text{ iff } \Vdash_\sS \dot x_q \neq \emptyset \rightarrow \neg \varphi(i(\dot x_0), \dots, i(\dot x_n)) \\
&\text{ iff } q \Vdash_\sS \neg \varphi(i(\dot x_0), \dots, i(\dot x_n)) \\
&\text{ iff } q \Vdash_\sS \neg \varphi(\dot x_0, \dots, \dot x_n),
\end{align*} where $\dot x_q = \{ (q, \check q) \}$.
This is a contradiction and so we have shown (3).
\end{proof}

It will follow from a later result (Theorem~\ref{thm:completion}) that the statement of Proposition~\ref{prop:equitenacious} also holds for $\cong_\SN$.

\begin{cor}
    $\sS \cong_\HS \sT$ as well as $\sS \cong_\SN \sT$ are first-order definable.
\end{cor}

\begin{proof}
    Theorem~\ref{thm:tenequi} gives a first-order characterization for $\sS \cong_\SN \sT$ as well as for $\sS \cong_\HS \sT$, when $\sS,\sT$ are tenacious. For arbitrary $\sS, \sT$, $\sS \cong_\HS \sT$ iff the tenacious systems given by the previous proposition are equivalent.
\end{proof}

Combining Lemma~\ref{lem:hseqhr}, Theorem~\ref{thm:tenequi} and Corollary~\ref{thm:tenequilr} we immediately obtain the following. 

\begin{cor}
    Suppose that $\sS, \sT$ are tenacious. Then $\sS \cong_\HS \sT$ implies $\sS \cong_\SN \sT$. 
\end{cor}

\section{Products and reduced two-step iterations}
Products of symmetric systems have been widely used since the the early days of forcing, as they make it easy to create specific and ``local'' failures of the axiom of choice (see \cite{Monro:DF} for example).

\begin{definition}
    Let $\sS_0 = (\mathbb{P}_0, \sG_0, \sF_0)$ and $\sS_1 = (\mathbb{P}_1, \sG_1, \sF_1)$ be symmetric systems. Then we define the product system $\sS_0 \times \sS_1 = (\mathbb{P}, \sG, \sF)$, where \begin{enumerate}
        \item $\mathbb{P} = \mathbb{P}_0 \times \mathbb{P}_1$,
        \item $\sG = \sG_0 \times \sG_1$ acts on $\mathbb{P}$ coordinatewise, 
        \item and $\sF$ is generated by subgroups of $\sG$ of the form $H_0 \times H_1$, for $H_0 \in \sF_0$, $H_1 \in \sF_1$.
    \end{enumerate} 
\end{definition}

It is straightforward to check that the product is a symmetric system and any $\sS_0 \times \sS_1$-generic is of the form $G_0 \times G_1$ for an $\sS_0$-generic $G_0$ and an $\sS_1$-generic $G_1$.

The goal of this section is to show that $\sS_0 \times \sS_1$ is equivalent, and in fact $\SN$-equivalent, to $\sS_0 * \check \sS_1$ as we would expect. In order to do so we prove a more general result about \emph{reduced two-step iterations} that will also be useful in Section~\ref{sec:quotients}.

\begin{definition}\label{def:namewitness}
    Let $\sS_0 = (\mathbb{P}_0, \sG_0, \sF_0)$ and $\dot \sS_1 = (\dot{\mathbb{P}}_1, \dot \sG_1, \dot \sF_1)$ be an $\sS_0$-name for a symmetric system, $\sym(\dot \sS_1) = \sG_0$. Then a \emph{name-witness} for $\dot \sS_1$ is a triple $\mathbf{S}_1 = (\mathbf{P}_1, \mathbf G_1, \mathbf F_1)$ of sets of $\sS_0$-names such that \begin{enumerate}
        \item $\mathbf{P}_1$, $\mathbf{G}_1$ and $\mathbf{F}_1$ are closed under $\sG_0$,
        \item $\Vdash_{\sS_0} \dot{\mathbb{P}}_1 = \mathbf{P}_1^\bullet \wedge \dot \sG_1 = \mathbf{G}_1^\bullet \wedge \dot \sF_1 = \mathbf{F}_1^\bullet$,\footnote{Note that by the previous clause, $\mathbf{P}_1^\bullet$, $\mathbf{G}_1^\bullet$, $\mathbf{F}_1^\bullet$ are $\sS_0$-names.}
        \item for any $\dot p \in \mathbf{P}_1$ and $\dot \pi \in \mathbf{G}_1$, there is a unique $\dot q \in \mathbf{P}_1$ such that $\Vdash_{\sS_0} \dot q = \dot \pi(\dot p)$,
        \item for any $\Delta_0$-formula $\psi$ and any $\dot x_0, \dots, \dot x_n \in \mathbf{G}_1 \cup \mathbf{F}_1$, if \[\Vdash_{\sS_0} \exists! x \in \dot \sG_1 \cup \dot \sF_1 \psi(x, \dot x_0, \dots, \dot x_n),\] then there is $\dot x \in \mathbf{G}_1 \cup \mathbf{F}_1$ so that $\Vdash_{\sS_0} \psi(\dot x, \dot x_0, \dots, \dot x_n)$.
    \end{enumerate}
\end{definition}

Let us remark that this definition is quite ad-hoc and simply summarizes conditions that are used to prove Proposition~\ref{prop:reduced}. Looking carefully at the proof one can see for instance that (4) can be weakened by restricting to a small finite list of $\Delta_0$-formulas. In our official definition of two-step iterations (Definition~\ref{def:twostep}) the role of $\mathbf{P}_1$ was played by closed names.

\begin{exmp}
    If $\sS_1 = ({\mathbb{P}}_1, \sG_1, \sF_1)$ is a symmetric system in $V$, then $({\mathbf{P}}_1, \mathbf G_1, \mathbf F_1)$ is a name witness for $\check{\sS_1}$, where ${\mathbf{P}}_1 = \{ \check{p} : p \in \mathbb{P}_1 \}$, $\mathbf G_1 = \{ \check{\pi} : \pi \in \sG_1 \}$ and $\mathbf F_1 = \{ \check{H} : H \in \sF_1 \}$.
\end{exmp}

Specifically note that if $p \Vdash_{\sS_0} \psi(\check x, \check x_0, \dots, \check x_n)$, then already $\Vdash_{\sS_0} \psi(\check x, \check x_0, \dots, \check x_n)$ by $\Delta_0$-absoluteness.

\begin{definition}
    Let $\sS_0$, $\dot \sS_1$ and $\mathbf{S}_1$ be as in Definition~\ref{def:namewitness}. Then we define the reduced two-step iteration $\sS_0 \star_{\mathbf{S}_1} \dot \sS_1$ exactly as in Definition~\ref{def:twostep}, but only using names from the name-witness $\mathbf{S}_1$.
\end{definition}

It is straightforward to check that this forms a symmetric system again, exactly as in the proof of Lemma~\ref{lem:twostep}. The only thing that we need to modify is to replace the names $\dot \pi( \dot p)$ with the respective unique names from $\mathbf{S}_1$ and note that $\pi \circ \sigma$, $\sigma^{-1}$, $\pi H \pi^{-1}$ etc. can be defined by $\Delta_0$ formulas. If we let $\mathbf{P}_1$ be the set of closed names for elements of $\dot{\mathbb{P}}_1$, and $\mathbf{G}_1$ and $\mathbf{F}_1$ be large enough sets of names for elements of $\dot \sG_1$ and $\dot \sF_1$, we obtain $\sS_0 \star_{\mathbf{S}_1} \dot \sS_1 = \sS_0 * \dot \sS_1$. 

Let us note that the uniqueness in clause (3) is not strictly necessary, in the same way that the restriction to closed names can be avoided in Definition~\ref{def:twostep}. For instance, we may simply identify two conditions $(p, \dot q_0)$, $(p, \dot q_1)$ whenever $\Vdash \dot q_0 = \dot q_1$ and obtain an equivalent system. 
 
At this point, one might think naively that one should just be able to just reprove the Factorization Theorem~\ref{lem:twostepfactor} in exactly the same way and be done. This is not quite the case as there is a key difference between closed names and the more general name-witnesses that becomes relevant in the proof of Theorem~\ref{lem:twostepfactor}. Namely, while any object definable from any $\sS_0$-names $\dot x_0, \dots, \dot x_n$ has a unique closed name and can be used in the definition of the two-step iteration, (3) only guarantees this for particular $\dot x_0, \dots, \dot x_n$. For instance, if $[\dot x]$ is an $\sS_0$-name for an $\sS_1$-name, there is a group of the form $(H_0, \sym_{\dot \sG_1}([\dot x]))$ in the filter of the two-step iteration, but this might not be the case for the reduced iteration, simply because there might be no name for $\sym_{\dot \sG_1}([\dot x])$ in $\mathbf{F}_1$. It turns out that the actual proof is quite a bit more technical. First, a simple lemma concerning usual two-step iterations:

\begin{lemma}\label{lem:equiautom}
    Let $\sS_0 * \dot{\sS_1} = (\mathbb{P}, \sG, \sF)$ be a two step iteration as in Definition~\ref{def:twostep}, let $\dot x$ be any $\mathbb{P}$-name and $G = G_0 * G_1$ an $\sS_0 * \dot{\sS_1}$-generic over $V$. Moreover let $(\pi, \dot \sigma), (\pi, \dot \tau) \in (\sG_0, \dot \sG_1)$ and suppose that $\dot \sigma^{G_0} = \dot \tau^{G_0}$. Then also $(\pi, \dot \sigma)(\dot x)^G = (\pi, \dot \tau)(\dot x)^G.$
\end{lemma}

\begin{proof}
    First note that \begin{align*}
        (\pi, \dot \sigma)(p,\dot q) \in G &\text{ iff } (\pi(p), \dot \sigma(\pi(\dot q))) \in G\\
        &\text{ iff }  \pi(p) \in G_0 \wedge \dot \sigma(\pi(\dot q)))^{G_0} \in G_1 \\
        &\text{ iff }  \pi(p) \in G_0 \wedge \dot \sigma^{G_0}(\pi(\dot q)^{G_0}) \in G_1 \\
        &\text{ iff }  \pi(p) \in G_0 \wedge \dot \tau^{G_0}(\pi(\dot q)^{G_0}) \in G_1 \\
        &\text{ iff }  \pi(p) \in G_0 \wedge \dot \tau(\pi(\dot q))^{G_0} \in G_1\\
        &\text{ iff }  (\pi, \dot \tau)(p,\dot q) \in G.
    \end{align*}
    The claim now follows by induction on the rank of $\dot x$. If $(\pi, \dot \sigma)(\dot y)^G \in (\pi, \dot \sigma)(\dot x)^G$, it is because $((p, \dot q), \dot y) \in \dot x$ and $(\pi, \dot \sigma)(p,\dot q) \in G$. 
    
    Then $(\pi, \dot \tau)(p,\dot q) \in G$, so $(\pi, \dot \tau)(\dot y)^G \in (\pi, \dot \tau)(\dot x)^G$ and by the inductive hypothesis $(\pi, \dot \tau)(\dot y)^G = (\pi, \dot \sigma)(\dot y)^G$. The same argument works in the other direction.
\end{proof}

In the following, when $G_0$ is $\sS_0$-generic over $V$ and $G_1$ is $\dot \sS_1^{G_0}$-generic over $V[G_0]_{\sS_0}$, we let \[G_0 \star_{\mathbf{S}_1} G_1 = \{ (p, \dot q) \in \mathbb{P}_0 \times \mathbf{P}_1 : p \in G_0, \dot q^{G_0} \in G_1 \}.\]

We will also drop the subscript $\mathbf{S}_1$ from $\star_{\mathbf{S}_1}$ everywhere to enhance readability. 

\begin{prop}\label{prop:reduced}
    Let $\sS_0$, $\dot \sS_1$ and $\mathbf{S}_1$ be as in Definition~\ref{def:namewitness}. Then $\sS_0 * \dot \sS_1 \cong_\SN \sS_0 \star \dot \sS_1$. Moreover, whenever $G_0 * G_1$ is $\sS_0 * \dot \sS_1$-generic over $V$, $G_0 \star G_1$ is $\sS_0 \star \dot \sS_1$-generic over $V$ and vice-versa, if $G$ is $\sS_0 \star \dot \sS_1$-generic, then $G_0 * G_1$ is $\sS_0 * \dot \sS_1$-generic, where $G_0 = \dom G$ and $G_1 = \{ \dot q^{G_0} : \exists p \in G_0 ((p, \dot q) \in G)\}$. In either case, \[V[G_0 * G_1]_{\sS_0 * \dot{\sS_1}} = V[G_0 \star G_1]_{\sS_0 \star\dot \sS_1}.\]
    
\end{prop}

\begin{proof}
    First, let us introduce a bit of notation to enhance readability. We will denote the forcing notion of $\sS_0 * \dot \sS_1$ by $\mathbb{P}$ and the the forcing notion of $\sS_0 \star \dot \sS_1$ by $\mathbb{P}_\star$. Also we will generally substitute subscripts of the form ${\sS_0 \star \dot \sS_1}$ with $\star$ when reasonable. Thus, we write $\sym_\star( \dot x)$ instead of $\sym_{\sS_0 \star \dot \sS_1}(\dot x)$ and reserve $\sym(\dot x)$ for the system $\sS_0 * \dot \sS_1$. And similarly, for $H_0 \in \sF_0$, $\dot H_1 \in \mathbf{F}_1$, we let \[(H_0, \dot H_1)_{\star} = \{ (\pi_0, \dot \pi_1) : \pi_0 \in H_0, \dot \pi_1 \in \mathbf{G}_1, \Vdash \dot \pi_1 \in \dot H_1 \},\] and keep the meaning of $(H_0, \dot H_1)$ as before. 
    
    Next, without any loss of generality, we can assume that $\mathbf{P}_1$ just consists of closed names: Item (4) of Definition~\ref{def:namewitness} implies that there is a name $\dot \id \in \mathbf{G}_1$ for the identity. Namely, if $\dot \sigma \in \mathbf{G}_1$ is arbitrary, $\dot \sigma \circ \dot\sigma^{-1}$ is $\Delta_0$-definable. Next, the uniqueness condition in (3) implies that for any $\dot p, \dot q \in \mathbf{P}_1$, $\Vdash \dot p = \dot \id(\dot p) = \dot q$ if and only if $\dot p = \dot q$. So we might as well replace each $\dot p$ with $\cl(\dot p)$ and obtain a completely isomorphic situation. In particular then, $\mathbb{P}_{\star}$ is a dense subset of $\mathbb{P}$. Also this means for instance that for $\dot \pi \in \mathbf{G}_1$, $\dot p \in \mathbf{P}_1$, the closed name $\dot \pi(\dot p)$ is already the unique $\dot q \in \mathbf{P}_1$ so that $\Vdash \dot \pi(\dot p) = \dot q$ and there is no difference between computing $\bar \pi(\dot x)$, $\bar \pi = (\pi_0, \dot \pi_1)$, relative to $\sS_0 \star \dot \sS_1$ or relative to $\sS_0 * \dot \sS_1$, when $\dot x$ is a $\mathbb{P}_\star$-name.

Now suppose that $G$ is $\sS_0 \star \dot \sS_1$-generic and $G_0$, $G_1$ are as in the statement of the proposition. Let $D$ be an $\sS_0 * \dot \sS_1$-symmetrically dense set, say $(H_0, \dot H_1)``D = D$. Note that $G_0$ is $\sS_0$-generic and thus there is $\dot H_1' \in \mathbf{F}_1$ so that $\dot H_1^{G_0} = \dot H_1'^{G_0}$. Also define $H_0' = H_0 \cap \sym(\dot H_1')$. Let $E$ consists of the conditions in $\mathbb{P}_\star$ that are below an element of $D$. Then \[\tilde D = \{ (\pi_0, \dot \pi_1)(q_0, \dot q_1) : (q_0, \dot q_1) \in E,  (\pi_0, \dot \pi_1) \in (H_0', \dot{H_1'})_\star \}\] is $\sS_0 \star \dot \sS_1$-symmetric and dense. Whenever $(\pi_0, \dot \pi_1)(q_0,\dot q_1) \in G_0 \star G_1 \cap \tilde D$ is of the form above, there is $(p_0, \dot p_1) \in D$ above $(q_0, \dot q_1)$ and $\dot \tau$ so that $\Vdash \dot \tau \in \dot H_1$ and $\dot \pi_1^{G_0} = \dot \tau^{G_0}$ (see Lemma~\ref{lem:thetrick} and the discussion afterwards). Then $(\pi_0, \dot \tau) \in (H_0, \dot H_1)$ and $(\pi_0, \dot \tau)(p_0, \dot p_1) \in G_0* \dot G_1 \cap D$. This shows that $G_0 * G_1$ is $\sS_0 * \dot \sS_1$-generic.  To show that any $\sS_0 * \dot \sS_1$-generic induces an $\sS_0 \star \dot \sS_1$-generic is not much harder.

Finally, let us show that the two models agree. Suppose that $\dot x$ is an $\sS_0 \star  \dot{\sS_1}$-symmetric name and $\bar H_\star = (H_0, \dot{H_1})_\star \leq \sym_{\star}(\dot x)$. Then $\bar H = (H_0, \dot{H_1})$ is a group in the system $\sS_0 * \dot{\sS_1}$ and we can define a symmetrisation of $\dot x$ by \[s(\bar H, \dot x) = \bigcup_{\bar \sigma \in \bar H} \bar \sigma(\dot x).\]
    
    Note by Lemma~\ref{lem:equiautom} that $s(\bar H, \dot x)^{G_0 * G_1} = \dot x^{G_0 * G_1}$ and in fact $\Vdash_{\mathbb{P}} s(\bar H, \dot x) = \dot x.$ Also $\bar H \leq \sym(s(\bar H, \dot x))$, so $s(\bar H, \dot x)$ is $\sS_0 * \dot{\sS_1}$-symmetric. By recursion on the rank of $\dot x$, define \[j(\dot x) = \bigcup_{\bar H_\star \leq \sym_\star(\dot x)} s(\bar H, \{ (r, j(\dot y)) : (r, \dot y) \in \dot x\}).\]

By induction on the rank one can then also see that $j(\dot x)^{G_0 * G_1} = \dot x^{G_0 * G_1}$, $\Vdash_{\mathbb{P}} j(\dot x) = \dot x$ and that $j(\dot x)$ is an $\sS_0 * \dot \sS_1$-name. This shows that $V[G_0 * G_1]_{\sS_0 * \dot{\sS_1}} \supseteq V[G_0 \star G_1]_{\sS_0 \star\dot \sS_1}$. 

Next, suppose that $\dot x$ is an $\sS_0 * \dot \sS_1$-name of minimal rank such that $\neg \Vdash_{\mathbb{P}} \dot x \in \HS_\star^\bullet$. Define the $\mathbb{P}_\star$-name \[\dot y = \{ (\bar p, \dot z) \in \mathbb{P}_\star \times \HS_{\gamma,\star} : \bar p \Vdash_{\sS_0 * \dot \sS_1} j(\dot z) \in \dot x \}\] for large enough $\gamma$, and note by the minimality of $\dot x$ that $\Vdash_{\mathbb{P}} \dot y = \dot x$. Let $(H_0, \dot H_1) \leq \sym(\dot x)$. Then there is a dense set of $p_0 \in \mathbb{P}_0$ with $\dot H_1' \in \mathbf{F}_1$ so that $p_0 \Vdash \dot H_1 = \dot H_1'$. Given such $p_0$ and $\dot H_1'$, let $H_0' = H_0 \cap \sym(\dot H_1')$ and $\dot x'$ be the symmetrization of $\dot y$ by $(H_0', \dot H_1')_\star$, so $\dot x' \in \HS_\star$. Using Lemma~\ref{lem:equiautom} again, for any $\bar \pi \in (H_0', \dot H_1')_\star$, $(p_0, \mathds 1) \Vdash_{\mathbb{P}} \bar \pi(\dot y) = \bar \pi(\dot x) = \dot x$, and further, $(p_0, \mathds 1) \Vdash_{\mathbb{P}} \dot x' = \dot x$. By density, $\Vdash_{\mathbb{P}} \dot x \in \HS_{\star}^\bullet$ and we obtain a contradiction. Finally, if $\Vdash_{\mathbb{P}} \dot x \in \HS_\star^\bullet$, there is a large enough $\alpha$ such that \[\Vdash_{\mathbb{P}} \dot x \in \HS_{\alpha,\star} = j(\HS_{\alpha,\star})\] and in particular, \[\Vdash_{\sS_0 * \dot{\sS_1}} \dot x \in j(\HS_{\alpha, \star}) \text{ and } \dot x^{G_0 * G_1} \in j(\HS_{\alpha, \star})^{G_0 * G_1} = \HS_{\alpha, \star}^{G_0 * G_1} = \HS_{\alpha, \star}^{G_0 \star G_1}.\]
For the last equality simply note that $G_0 \star G_1 = (G_0 * G_1) \cap \mathbb{P}_\star$. Following the characterization in Theorem~\ref{thm:tenequilr} we have that the two systems are $\SN$-equivalent. This finishes the proof.
\end{proof}

The following is now an immediate corollary:

\begin{thm}
    Let $\sS_0$ and $\sS_1$ be symmetric systems. Then $\sS_0 \times \sS_1 \cong_{\SN} \sS_0 * \check \sS_1$. Moreover $G_0 \times G_1$ is $\sS_0 \times \sS_1$ generic iff $G_0 * G_1$ is $\sS_0 * \check \sS_1$-generic and we have that \[V[G_0 \times G_1]_{\sS_0 \times \sS_1} = V[G_0 * G_1]_{\sS_0 * \check \sS_1}.\]
    Consequently, $\sS_0\times\sS_1\cong_{\SN}\sS_1\times\sS_0$.
\end{thm}
\section{Quotients}\label{sec:quotients}

The dual notion of iterations would be quotients, and much like in the case of forcing, it is can be useful to know when one forcing is a quotient of another. In the following we define when $\sS_0$ is a complete subsystem of $\sS_1$, which will imply that an extension by $\sS_1$ is a two-step iteration of $\sS_0$ and some quotient system. We will define the notion of a complete subsystem, then in order to define the quotient symmetric extension we will first define ``respect bases'', which play an important role in the definition of the quotient system, and will play an important role later as well.

\begin{definition}[Complete subsystems]\label{def:completesubsys}
Let $\sS_0 = (\mathbb{P}, \sG, \sF)$, $\sS_1 = (\mathbb{Q}, \sH, \sE)$ be symmetric systems. Then we say that \emph{$\sS_0$ is a subsystem of $\sS_1$} iff 
\begin{enumerate}
    \item $\mathbb{P}$ is a subforcing of $\mathbb{Q}$, i.e., $\mathbb{P} \subseteq \mathbb{Q}$ and the extension and incompatibility relations coincide, 
    \item $\forall H \in \sF \exists K \in \sE (K \restriction \mathbb{P} := \{ \pi \restriction \mathbb{P} : \pi \in K \} \leq H)$.
\end{enumerate}

We say that \emph{$\sS_0$ is a complete subsystem of $\sS_1$} and write $\sS_0 \lessdot \sS_1$ if additionally 

\begin{enumerate}
\setcounter{enumi}{2}
    \item $\mathbb{P} \lessdot_{\sS_0} \mathbb{Q}$, i.e., every $\sS_0$-symmetrically dense subset of $\mathbb{P}$ is predense in $\mathbb{Q}$.
\end{enumerate}
\end{definition}

Every symmetric system $(\mathbb{P}, \sG, \sF)$ is a complete subsystem of $(\mathbb{P}, \{ \id \}, \{ \{\id \} \})$, which corresponds to forcing with $\mathbb{P}$. This may seem confusing at first, but if we think of the standard case with forcing, when $\mathbb{P}$ completely embeds into $\mathbb{Q}$, that means that adding a generic for $\mathbb{Q}$ will add one for $\mathbb{P}$, so in a meaningful sense, $\mathbb{Q}$ adds ``more'' objects than $\mathbb{P}$. In the case of symmetric extensions we can often see $\sG$ as preserving some kind of structure, and $\sF$ preserving subsets of that structure. In that case, shrinking $\sG$ or enlarging $\sF$ will allow us to preserve either more structure or more subsets, and this culminates with the trivial group (or the improper filter) which will add all the generic subsets.

If we identify $\sS$ naturally with a subsystem of a two-step iteration $\sS * \dot \sT$, we have that $\sS$ is a complete subsystem of $\sS * \dot \sT$. In fact in all applications we consider, we will have that $\mathbb{P} \lessdot \mathbb{Q}$, i.e., every dense subset of $\mathbb{P}$ is predense in $\mathbb{Q}$, and this is particularly the case in two-step iterations. But (3) is a slightly more natural choice and it turns out that it does not add too much complexity.

\begin{lemma}\label{lem:completesubgen}
Let $\sS_0$ be a complete subsystem of $\sS_1$ as above and let $G$ be $\sS_1$-generic. Then $G_0 := G \cap \mathbb{P}$ is $\sS_0$-generic and $V[G_0]_{\sS_0} \subseteq V[G]_{\sS_1}$. In fact every $\sS_0$-name is an $\sS_1$-name.
\end{lemma}

\begin{proof}
    To see that $G_0$ is generic, by (3) it suffices to check that if $D\subseteq \mathbb{P}$ is $\sS_0$-symmetric, then $D$ is also $\sS_1$-symmetric. But this immediately follows from (2). Similarly, if $\dot x$ is an $\sS_0$-name, then by $(2)$ there is $K \in \sE$ such that $K \restriction \mathbb{P} \leq \sym(\dot x)$, so $K$ fixes $\dot x$. This argument can be made hereditarily and thus $\dot x$ is also an $\sS_1$-name. Clearly, $\dot x^G = \dot x^{G_0}$, so $V[G_0]_{\sS_0} \subseteq V[G]_{\sS_1}$.
\end{proof}

\begin{definition}\label{def:Hred}
Let $\mathbb{P}$ be a subforcing of $\mathbb{Q}$, let $p\in \mathbb{P}$, $q\in \mathbb{Q}$ and $H$ be a group of automorphisms on $\mathbb{P}$. Then $p$ is an \emph{$H$-reduction of $q$} if \[\forall r \in \mathbb{P}, r \leq p \exists \pi \in H (\pi(r) \parallel q).\]
\end{definition}

The usual notion of reduction (as given, e.g., in \cite[Definition 7.1]{Kunen1980}) corresponds to the case where $H$ is the trivial group. 

\begin{lemma}\label{lem:Hred}
    Let $\sS = (\mathbb{P}, \sG, \sF)$ and $\mathbb{P}$ be a subforcing of $\mathbb{Q}$. Then $\mathbb{P} \lessdot_{\sS} \mathbb{Q}$ iff for every $H \in \sF$ and $q \in \mathbb{Q}$, $q$ has an $H$-reduction in $\mathbb{P}$. Moreover, in this case, whenever $G \subseteq \mathbb{Q}$ is a filter such that $G_0 = G \cap \mathbb{P}$ is $\sS_0$-generic, then for every $H \in \sF$, any $q \in G$ has an $H$-reduction $p \in G_0$.
\end{lemma}

\begin{proof}
    Suppose that every condition in $\mathbb{Q}$ has an $H$-reduction and $D \subseteq \mathbb{P}$ is symmetrically dense open as witnessed by $H$. Fix $q \in \mathbb{Q}$ and let $p$ be an $H$-reduction. Then $p \parallel r$ for some $r \in D$. Let $r' \in \mathbb{P}$, $r' \leq p,r$ and $\pi \in H$ such that $\pi(r') \parallel q$. Then $\pi(r) \in D$ and $\pi(r) \parallel q$. Since $q$ was arbitrary, $D$ is indeed predense in $\mathbb{Q}$. 
        
    On the other hand, assume that $\mathbb{P} \lessdot_\sS \mathbb{Q}$ and let $q \in \mathbb{Q}$, $H \in \sF$ be arbitrary. Consider $E = \{ p \in \mathbb{P} : \forall \pi \in H (\pi(p) \perp q) \}$ and $D = E \cup \{ p \in \mathbb{P} : p \perp E \}$. $D$ is obviously dense in $\mathbb{P}$ and moreover $\pi``D = D$. Thus $D$ is predense in $\mathbb{Q}$ and there is $p \in D$ such that $p \parallel q$. Clearly $p \notin E$ and thus $p \perp E$. But then for any $r \leq p$, there must be some $\pi \in H$ such that $\pi(r) \parallel q$. So $p$ is an $H$-reduction of $q$.

    Now whenever $G \subseteq \mathbb{Q}$ is a filter such that $q \in G$ and $G_0 = G \cap \mathbb{P}$ is $\sS$-generic, there is of course such $p$ in $D \cap G_0$.\end{proof}

\subsection{Respect-bases}
\begin{definition}
    Let $\sS = (\mathbb{P}, \sG, \sF)$ be a symmetric system and $\dot x$ be a $\mathbb{P}$-name. Then the \emph{respect-diagram} of $\dot x$ is \[ R_\sS(\dot x) := \{ (p,\pi) : p \in \mathbb{P}, \pi \in \sG, p \Vdash \pi(\dot x) = \dot x  \}.\]
 A collection $\mathcal{R} \subseteq \HS$ is a \emph{respect-basis} for $\sS$ if 

 \begin{enumerate}
    \item for every $\dot x \in \HS$, the set of $p \in \mathbb{P}$ for which there are $\dot x' \in \HS$ and $\dot y \in \sR$ such that \begin{enumerate}
        \item $p \Vdash \dot x = \dot x'$ and
        \item $R_\sS(\dot y) \subseteq R_\sS(\dot x')$,
     \end{enumerate} is dense. 
     \item $\sR$ is closed under $\sG$. That is to say, $\sym(\sR^\bullet)=\sG$.
 \end{enumerate}
\end{definition}

A simple case of a respect-basis is when for each $\dot x \in \HS$ there is $\dot y \in \sR$ with $R_\sS(\dot y) \subseteq R_\sS(\dot x)$. This is probably the most common case one will encounter, but (1) is also sufficient and will be relevant for Proposition~\ref{prop:respectbasefromHOD}.

The following is well-known. 

\begin{lemma}
    Assume $\DC_\kappa$. Then for any structure $N$ in a language of size $\leq \kappa$ and any $A \subseteq N$ of size $\leq \kappa$, there is an elementary substructure $M \preccurlyeq N$ of size $\leq \kappa$ with $A \subseteq M$.
\end{lemma}

\begin{lemma}\label{lem:computingrespect}
  $(\ZF+\DC_{\kappa})$  Suppose that $\vert \mathbb{P} \cup \sG \vert = \kappa$. Then $\HS_{\kappa^+}$ is a respect-basis for $\sS$.
\end{lemma}

\begin{proof}
    We may assume without loss of generality that $\mathbb{P}$ and $\sG$ are both coded by subsets of $\kappa$. Let $\dot x\in V_\beta$ for $\beta>\kappa^+$ and let $M$ be an elementary submodel of $V_\beta$ of size $\kappa$ such that $\kappa\cup\{\mathbb{P},\sG,\dot x\}\subseteq M$. Letting $\iota\colon M\to N$ be the Mostowski collapse, then $\iota(\mathbb{P})=\mathbb{P}$ and $\iota(\sG)=\sG$. Let $\dot y=\iota(\dot x)$. For every $\pi\in\sG$, we have that     
    $p\Vdash\pi(\dot x)=\dot x$ iff
    $M\models p\Vdash\pi(\dot x)=\dot x$ iff
    $N\models p\Vdash\pi(\dot y)=\dot y$ iff
    $p\Vdash\pi(\dot y)=\dot y$.
    Let $\alpha=N\cap\Ord$, then $\alpha < \kappa^+$ by $\DC_\kappa$. So, $\dot y \in \HS_\alpha$ and $R_\sS(\dot y) = R_\sS(\dot x)$.\end{proof}

The computation above needs some form of choice in order to obtain the elementary submodel $M$. In the abscence of choice, weakly Löwenheim-Skolem cardinals can still be used towards a uniform bound for respect-bases.

\begin{definition}[Usuba \cite{Usuba2021}]
    A cardinal $\kappa$ is a \emph{weakly Löwenheim-Skolem cardinal} if for every $\gamma < \kappa$, $\alpha \geq \kappa$ and $x \in V_\alpha$ there is an elementary $M \preccurlyeq V_\alpha$, $V_\gamma \cup \{ x \} \subseteq M$ and the Mostowski-collapse of $M$ is an element of $V_\kappa$.
\end{definition}

With essentially the same proof as Lemma~\ref{lem:computingrespect}, this time not relying on $\DC_\kappa$, we have that: 

\begin{lemma}\label{lem:compLS}
    Let $\kappa$ be a weakly Löwenheim-Skolem cardinal with $\sS\in V_\kappa$. Then $\HS_\kappa$ is a respect-basis for $\sS$.
\end{lemma}

In $\ZFC$, weakly LS cardinals are exactly the $\beth$-fixed points. Thus the computation of Lemma~\ref{lem:computingrespect} is in general much better than that of Lemma~\ref{lem:compLS}. In concrete examples, respect-bases can often be found directly. 

Note that for any $\mathbb{P}$-name $\dot x$, we always have that \[\mathbb{P} \times \sym_\sS(\dot x) \subseteq \mathbb{P} \times \res_\sS(\dot x) \subseteq R_\sS(\dot x).\]

Thus, if for $H \in \sF$ we find a name $\dot y$ such that $R_\sS(\dot y) = \mathbb{P} \times H$, then $R_\sS(\dot y) \subseteq R_\sS(\dot x)$ for all $\dot x$ with $H \leq \sym(\dot x)$. This makes it easy to find respect-bases in some concrete examples.

\begin{exmp}
    Let $\sS$ and $\langle \dot c_n : n \in \omega \rangle$ be as in Example~\ref{ex:Cohen}. Then the set of all names of the form $\bar c_s = \langle \dot c_{s(i)} : i < \vert s \vert \rangle^\bullet$ for injective $s \in \omega^{<\omega}$ forms a respect-basis. Namely, whenever $\pi \notin \fix(\ran(s)) = \sym(\bar c_s)$, we have that \[p \Vdash \pi(\bar c_s) \neq \bar c_{s}\] for any $p \in \mathbb{P}$. Thus $R_\sS(\bar c_s) = \mathbb{P} \times \fix(\ran(s))$.
\end{exmp}

\begin{quest}
    Is Lemma~\ref{lem:computingrespect} optimal?
\end{quest}

\begin{remark}
    Suppose that $\mathbb{P}$ is a complete Boolean algebra and $\sG = \Aut(\mathbb{P})$. Then \[R_\sS(\dot y) \subseteq R_\sS(\dot x) \leftrightarrow \res_\sS(\dot y) \subseteq \res_\sS(\dot x).\]
\end{remark}

\begin{proof}
    For any $\pi \in \sG$ and $p \in \mathbb{P}$, we can define an involution $\sigma \in \sG$ that acts as the identity below the complement of $p \vee \pi(p)$, acts as $\pi$ below $p$ and as $\pi^{-1}$ below $\pi(p)$. Then note that $p \Vdash \dot x = \pi(\dot x)$ iff $\Vdash \dot x = \sigma(\dot x)$.
\end{proof}

\subsection{The quotient forcing}

For the rest of the entire section we fix an assignment of a respect-basis $\sR(\sS)$ to each system $\sS$. The definitions we give below technically depend on this choice but none of the results are affected. The particular choice will become more important later though. For now we may simply use the assignment of $\sR(\sS) = \HS_{\sS, \alpha}$, where $\alpha$ is least ordinal for which $\HS_{\sS,\alpha}$ is a respect-basis.  

\begin{definition}\label{def:quotforce}
    Let $\sS = (\mathbb{P}, \sG, \sF)$ be a symmetric system and let $\mathbb{P} \lessdot_{\sS} \mathbb{Q}$. Consider the formula \[\psi(p,r, (\dot x_i, \dot y_i)_{i <n}) = \forall p' \leq p \exists \pi \in \sG (p \Vdash \forall i < n(\pi(\dot x_i) = \dot y_i) \wedge \pi^{-1}(p') \parallel r)\] Then we define an $\sS$-name for a forcing notion \[\dot{\mathbb{Q}/\sS} = \{ (p, (\check r ,\{((\dot x_i)\check{}, \dot y_i ) : i < n\})^\bullet) : \psi(p,r, (\dot x_i, \dot y_i)_{i <n}) \},\] where $p, p'$ range over $\mathbb{P}$, $r$ ranges over $\mathbb{Q}$, $\dot x_i, \dot y_i$ range over $\sR(\sS)$ and $n$ ranges over $\omega$. In $V^{\sS}$, $(r_1, a_1) \leq (r_0, a_0)$ iff $r_1 \leq r_0$ and $a_0 \subseteq a_1$.
\end{definition}

The following is an easy observation.
\begin{prop}
     For any $\sigma \in \sG$, $\sigma(\dot{\mathbb{Q}/\sS}) = \dot{\mathbb{Q}/\sS}$.
\end{prop}

Note that the forcing $\mathbb{P}/\sS$ depends on the particular system $\sS$ and not just on its $\cong_\SN$-equivalence class. For instance, for $\sS = (\mathbb{P}, \{\id\}, \{\{ \id\}\})$, $\mathbb{P}/\sS$ is a trivial forcing, as we would expect but for $\sS = (\mathbb{P}, \Aut(\mathbb{P}), \langle \{ \id \} \rangle)$ the quotient $\mathbb{P}/\sS$ is a big lottery sum of trivial forcings, more specifically $\bigoplus_{\pi \in \Aut(\mathbb{P})^V} \pi``G$, where $G$ is the generic filter (so a trivial forcing on its own). 

This is somewhat of a drawback as we cannot generally guarantee that $\sS * \dot{\mathbb{P}/{\sS}}$ is strongly equivalent to $\mathbb{P}$, something that does hold for quotients of forcing notions. This is also not just due to our specific definition of $\mathbb{P}/\sS$. For a more general example, consider any homogeneous minimal forcing $\mathbb{P}$, such as Sacks forcing. Let $\sS = (\mathbb{P}, \Aut(\mathbb{P}),\{\Aut(\mathbb{P})\})$. Then $\sS$ does not add any new sets, so in order for $\sS * \dot{\mathbb{P}/\sS}$ to generate a $\mathbb{P}$-extension, $\dot{\mathbb{P}/\sS}$ must be a name for a non-trivial forcing notion. But then surely $\mathbb{P}$ and $\mathbb{P} * \dot{\mathbb{P}/\sS}$ are not even weakly equivalent.

\begin{lemma}\label{lem:homquotient}
    $\Vdash_\sS \dot{\mathbb{P}/\sS}$ is weakly homogeneous.
\end{lemma}

\begin{proof}
First note that using a density argument we can replace $\psi$ in Definition~\ref{def:quotforce} by \[\exists \pi \in \sG (p \Vdash \forall i < n(\pi(\dot x_i) = \dot y_i) \wedge \pi^{-1}(p) \leq r)\] and obtain an $\sS$-name for the same object. This is the definition of the name $\dot{\mathbb{P}/\sS}$ that we will use for this proof.

    Let $G$ be $\sS$-generic over $V$. In $V[G]_\sS$, let $(r_0,a_0), (r_1,a_1) \in \mathbb{P}/\sS$ be arbitrary. Then there are $p_0, p_1 \in G$, $\pi_0, \pi_1\in \sG$ and $(p_0, (\check r_0 ,\dot a_0)^\bullet), (p_1,(\check r_1, \dot a_1)^\bullet) \in \dot{\mathbb{P}/\sS}$, witnessed by $\pi_0\in \sG$ and $\pi_1 \in \sG$ respectively, so that $\dot a_0^G = a_0$, $\dot a_1^G = a_1$. Let $\tau = \pi_1^{-1} \circ \pi_0$. We define an automorphism $\sigma$ of $\mathbb{P}/\sS$ such that \[\sigma(r,b) = (\tau (r), \{ (\tau(\dot x), y) : (\dot x, y) \in b \} ),\] for any $(r,b) \in \mathbb{P}/\sS$. We need to check that $\sigma(r,b) \in \mathbb{P}/\sS$. Once we have shown this, it is obvious, from the definition of $\leq$ and the fact that $\sigma$ has an inverse, that $\sigma$ is an automorphism of $\mathbb{P}/\sS$. So let $p \in G$, $(p, (\check{r}, \dot b)^\bullet) \in \dot{\mathbb{P}/\sS}$, where $\dot b^G = b$, and let $\pi$ be a witness as before. Let $\pi' = \pi \circ \tau^{-1}$. As $\pi^{-1}(p) \leq r$, $\pi'^{-1}(p) = \tau \circ \pi^{-1}(p)\leq \tau(r)$. Furthermore, for every $((\dot x)\check{}, \dot y)$ appearing in $\dot b$, \[p \Vdash \dot y = \pi(\dot x) = \pi \circ \tau^{-1} \circ \tau (\dot x) = \pi'(\tau(\dot x)).\] Thus \[(p, ({\tau(r)}\check{}, \{ (\tau(\dot x)\check{}, \dot y) : ((\dot x)\check{}, \dot y) \text{ appears in } \dot b \})) \in \dot{\mathbb{P}/\sS}\] and indeed $\sigma(r,b) \in \mathbb{P}/\sS$. Finally, note that $\sigma(r_0, a_0) \parallel (r_1, a_1)$. Namely, if $p \leq p_0, p_1$ and since $\pi_0^{-1}(p) \leq r_0$ and $\pi_1^{-1}(p) \leq r_1$, we have that $\pi_1^{-1}(p) \leq \tau(r_0), r_1$. If we let $r = \pi_1^{-1}(p)$, it is easy, similarly to before, to check that \[(p,(\check{r}, \dot a_1 \cup \{ (\tau(\dot x)\check{}, \dot y) : ((\dot x)\check{}, \dot y) \text{ appears in } \dot a_0 \})^\bullet) \in \dot{\mathbb{P}/\sS},\] as witnessed by $\pi_1$. 
\end{proof}

$\mathbb{P}/\sS$ is homogeneous no matter how inhomogeneous $\mathbb{P}$ is, which might seem odd at first. But consider for instant the case where $\mathbb{P}$ is completely rigid. Then any symmetric extension based on $\mathbb{P}$ is just a forcing extension by $\mathbb{P}$ and an appropriate notion of a quotient forcing, i.e., a forcing notion that lets us pass from this symmetric extension to a full generic extension, is expected to be trivial since these extensions agree. And indeed it is not hard to see that in such a case $\mathbb{P}/\sS$ is directed. In a sense, the symmetric extension knows the generic on the non-homogeneous parts of $\mathbb{P}$ already and the only information we need to pass to the full forcing extension is the generic on homogeneous parts of $\mathbb{P}$. 

\begin{lemma}\label{lem:canondense}
             Let $\dot D \in V$ be an $\sS$-name for a dense subset of $\mathbb{Q}/{\sS}$ and $q \in \mathbb{Q}$ be arbitrary. Then there is $q' \leq q$, $r \geq q'$ and $X \in [\sR(\sS)]^{<\omega}$ so that  \[q' \Vdash (\check r, \{ ((\dot x)\check{}, \dot x) : \dot x \in X\}) \in \dot D.\] 
\end{lemma}
\begin{proof}
     Let $G$ be any $\mathbb{Q}$-generic over $V$ with $q \in G$ and let $G_0 = G \cap \mathbb{P}$. By extending $q$, without loss of generality, we may find $\dot x_0 \in \sR(\sS)$ and $\dot D' \in \HS$ such that $q \Vdash \dot D = \dot D'$ and $\sR_{\sS}(\dot x_0) \subseteq \sR_{\sS}(\dot D')$. For simplicity we may then also simply assume that $\dot D = \dot D'$. By Lemma~\ref{lem:Hred} there is a $\sym(\dot x_0)$-reduction $p'$ of $q$ in $G_0$. Then note that $\psi(p', q, (\dot x_0, \dot x_0))$, so \[(p', (\check q ,\{((\dot x_0)\check{}, \dot x_0 )\})^\bullet) \in \dot{\mathbb{Q}/{\sS}}.\]  Thus, in $V[G_0]_{\sS}$, we can consider the condition $s = (q, \{(\dot x_0, \dot x_0^{G_0})\} ) \in \mathbb{Q}/\sS$. Since $\dot D^{G_0}$ is dense, there is $t \leq s$ in $\dot D^{G_0}$. According to the definition of $\dot{\mathbb{Q}/\sS}$, there is $p \in G_0, r \in \mathbb{Q}$, $\dot x_i, \dot y_i \in \sR(\sS)$, $i < n$, where $\dot y_0 = \dot x_0$, such that $t = (r, \{ (\dot x_i, \dot y_i^{G}) : i < n \})$ and $\psi(p,r,(\dot x_i, \dot y_i)_{i <n}))$. Moreover, we can assume wlog that $p \Vdash \dot t \leq \dot s \wedge \dot t \in \dot D$, where \[\dot t = (\check{r}, \{ ((\dot x_i)\check{}, \dot y_i) : i < n \})^\bullet\] and \[\dot s = (\check{q}, \{((\dot x_0)\check{}, \dot x_0)\} )^\bullet.\] 
    According to $\psi$, there is some $\pi \in \sG$ such that $\pi^{-1}(p) \parallel r$ and $p \Vdash \pi(\dot x_i) = \dot y_i$, for every $i < n$. In particular, $p \Vdash \pi(\dot x_0) = \dot x_0$. Let $q' \leq \pi^{-1}(p), r$ in $\mathbb{Q}$. By the symmetry lemma, $q' \Vdash \dot x_0 = \pi^{-1}(\dot x_0)$ and more generally $q' \Vdash \dot x_i = \pi^{-1}(\dot y_i)$. Thus, $(q', \pi^{-1}) \in R_{\sS}(\dot D) \supseteq R_{\sS}(\dot x_0)$ and $q' \Vdash \dot D = \pi^{-1}(\dot D).$ All in all, \[ q' \Vdash \pi^{-1}(\dot s) \geq \pi^{-1}(\dot t) \in \dot D,\] and since $\pi^{-1}(\dot t) = (\check{r}, \{ ((\dot x_i)\check{}, \pi^{-1}(\dot y_i)) : i < n \})^\bullet$, \[q' \Vdash \pi^{-1}(\dot t) = (\check{r}, \{ ((\dot x_i)\check{}, \dot x_i) : i < n \}) \in \dot D.\]
    Since $q' \leq r \leq q$ this finishes the proof of our claim. 
\end{proof}
\begin{cor}
    Let $G$ be $\mathbb{P}$-generic, then $K=G\times[\{(\dot x,\dot x^G)):x\in\sR(\sS)\}]^{<\omega}$ is $\mathbb{P}/\sS$-generic over $V[G]_\sS$, and $V[G]_\sS[K]=V[G].$
\end{cor}
\begin{proof}
    The genericity is an easy consequence of Lemma~\ref{lem:canondense}, and since $K$ is defined from $G$, it is easy to see that $V[G]_\sS[K]\subseteq V[G]$. In the other direction, note that $G=\dom K$, and so $G\in V[G]_\sS[K]$, and therefore $V[G]\subseteq V[G]_\sS[K]$ as wanted.
\end{proof}
In fact, more is true, as we will see in Proposition~\ref{prop:thirdquotient}.
\subsection{Quotients of systems}

\begin{definition}\label{def:quotsym}
    Let $\sS_0 \lessdot\sS_1$ where $\sS_0 = (\mathbb{P}, \sG, \sF)$ and $\sS_1 = (\mathbb{Q}, \sH, \sE)$. We let $\sH/\sS_0 = \{ \sigma \in \sH : \sigma \restriction \mathbb{P} \in \sG \}$ and $\sE/\sS_0 = \{ H \in \sE : H \leq \sH/\sS_0 \}$. Then we define an $\sS_0$-name $\dot{\sS_1 / \sS_0}$ for the system $(\mathbb{Q}/\sS_0, \sH/\sS_0, \sE/\sS_0)$, where $\sigma \in \sH/\sS_0$ acts on a condition by \[\sigma(r,\{ (\dot x_i, y_i) : i < n \}) = (\sigma(r), \{ (\sigma(\dot x_i), y_i) : i < n \}).\]

Note that this does define an automorphism of $\mathbb{Q} / \sS_0$. Whenever an instance $p' \leq p$ of $\psi(p,r, (\dot x_i, \dot y_i)_{i <n})$ is witnessed by $\pi \in \mathcal{G}$, then the same instance for $\psi(p,\sigma(r), (\sigma(\dot x_i), \dot y_i)_{i <n})$ is witnessed by $\pi \circ (\sigma^{-1} \restriction \mathbb{P}) \in \sG$.
\end{definition}

By item (2) of Definition~\ref{def:completesubsys}, $\sH/\sS_0 \in \sE$. Thus $\sS_1$ is strongly equivalent to $(\mathbb{Q}, \sH/\sS_0, \sE/\sS_0)$ and we may typically assume without loss of generality that already $\sH = \sH/\sS_0$ and so $\sE = \sE/\sS_0$. Passing to this subgroup of $\sH$ is really only a technical requirement for the definition of the quotient to make sense.\footnote{In fact, in an earlier draft of this paper, $\sH = \sH/\sS_0$ was an explicit requirement in Definition~\ref{def:completesubsys}, but it turned out to be inessential. Also, the results of Section~\ref{sec:characterising} are easier to formulate with this more generous definition.}

\begin{prop}\label{prop:thirdquotient}
Let $G$ be $\sS_1$-generic over $V$. Then letting \[K=G \times [\{ (\dot x, \dot x^G) : \dot x \in \sR(\sS_0)  \}]^{<\omega},\] $K$ is $\sS_1/\sS_0$-generic over $V[G_0]_{\sS_0}$, where $G_0 = G \cap \mathbb{P}$.
\end{prop}

\begin{proof}
    Let us first check that $K \subseteq \mathbb{Q}/{\sS_0}$. Let $(r,\{ (\dot x_i, \dot x_i^G) : i < n \}) \in K$ be arbitrary and consider $H = \bigcap_{i < n} \sym(\dot x_i)$. By Lemma~\ref{lem:Hred} there is an $H$-reduction $p$ of $r$ in $G_0$. Then it suffices to note that $(p, (\check r ,\{((\dot x_i)\check{}, \dot x_i ) : i < n\})^\bullet) \in \dot{\mathbb{Q}/{\sS_0}}$. $K$ is also clearly a filter. 
    
    Now let $\dot D \in V$ be an $\sS_0$-name for a symmetrically dense subset of $\mathbb{Q}/{\sS_0}$. Without loss of generality we may assume that there is $H \in \sE$ such that $H \restriction \mathbb{P} \leq \sym(\dot D)$ and $\Vdash \forall \sigma \in \check{H} (\sigma`` \dot D = \dot D)$.\footnote{See (2) of Definition~\ref{def:completesubsys} and Lemma~\ref{lem:thetrick}. More specifically, picking any $H \in \sE$ such that $H \restriction \mathbb{P} \leq \sym(\dot D)$ and $\forall \sigma \in H (\sigma `` \dot D^G = \dot D^G)$, consider $\chi(\dot D, \check{H}) \equiv \forall \sigma \in \check H (\sigma `` \dot D = \dot D)$ and note for instance that $\Vdash \chi(\mathbb{Q} \dot /{\sS_0}, \check H)$.} Then note that if \[q \Vdash (\check r, \{ ((\dot x)\check{}, \dot x) : \dot x \in X\}) \in \dot D,\] then $\sigma(q) \Vdash (\check r, \{ ((\dot x)\check{}, \sigma(\dot x)) : \dot x \in X\}) \in \dot D$, for any $\sigma \in H$. But since $\sigma$ is also forced to fix $\dot D$ under the action defined in Definition~\ref{def:quotsym}, \[\sigma(q) \Vdash (\sigma(r)\check{}, \{ ((\sigma(\dot x))\check{}, \sigma(\dot x)) : \dot x \in X\}) \in \dot D.\] Thus the set \[\tilde D = \{ q \in \mathbb{Q} : \exists r \geq q,  X \in [\sR(\sS_0)]^{<\omega} ( q \Vdash (\check r, \{ ((\dot x)\check{}, \dot x) : \dot x \in X\}) \in \dot D)\}\] is $\sS_1$-symmetric and according to Lemma~\ref{lem:canondense} it is dense in $\mathbb{Q}$. As $G$ is $\sS_1$-generic, $G \cap \tilde D \neq \emptyset$ and followingly, $K \cap \dot D^{G_0} = K \cap \dot D^{G} \neq \emptyset$.
\end{proof}

\begin{thm}\label{prop:fullquotient}
Let $\sS_0\lessdot \sS_1$. Then $\sS_1$ and $\sS_0 * \dot{\sS_1/\sS_0}$ are weakly equivalent. Moreover: 

\begin{enumerate}
\item Whenever $G$ is $\sS_1$-generic over $V$, $K = G \times [\{ (\dot x, \dot x^G) : \dot x \in \sR(\sS_0)  \}]^{<\omega}$ and $G_0=G\cap\mathbb{P}$, then
     \begin{enumerate}
     \item $G_0$ is $\sS_0$-generic over $V$,
    \item $K$ is $\sS_1/\sS_0$-generic over $V[G_0]_{\sS_0}$,
    \item and $V[G]_{\sS_1} = V[G_0]_{\sS_0}[K]_{\sS_1/\sS_0}$,
    \end{enumerate}
 \item Whenever $K_0$ is $\sS_0$-generic over $V$ and $K$ is $\sS_1/\sS_0$-generic over $V[K_0]_{\sS_0}$, $G = \dom K$ and $G_0 = G \cap \mathbb{P}$, then
     \begin{enumerate}
        \item $G$ is $\sS_1$-generic over $V$,
    \item $V[K_0]_{\sS_0} = V[G_0]_{\sS_0}$, 
    \item $K = G \times [\{ (\dot x, \dot x^G) : \dot x \in \sR(\sS_0)  \}]^{<\omega}$,
    \item and $V[K_0]_{\sS_0}[K]_{\sS_1} = V[G]_{\sS_1}$.
    \end{enumerate}
\end{enumerate}
\end{thm}

\begin{proof}
As mentioned earlier we may assume without any loss of generality that $\sH = \sH/\sS_0$ and thus also $\sE = \sE/\sS_0$. This will simplify the notation slightly.

First, let us consider the reduced iteration $\sS_0 \star  \dot{\sS_1/\sS_0} = \sS_0 \star_{(\mathbf{Q}, \mathbf{H}, \mathbf{E})} \dot{\sS_1/\sS_0}$, where $\mathbf{Q}$ is the set of closed names for elements of ${\mathbb{Q}/\sS_0}$, $\mathbf{H}= \{ \check \sigma : \sigma \in \sH \}$ and $\mathbf{E} = \{\check H : H \in \sE \}$. It is straightforward to check that $(\mathbf{Q}, \mathbf{H}, \mathbf{E})$ is a name-witness for $\dot{\sS_1/\sS_0}$. Then the forcing notions of $\sS_0 * \dot{\sS_1/\sS_0}$ and $\sS_0 \star  \dot{\sS_1/\sS_0}$ agree. We will use similar notation as lined out at the beginning of the proof of Proposition~\ref{prop:reduced}. For instance $(H_0, \check{H}_1)_\star$ will denote a group in the system $\sS_0 \star  \dot{\sS_1/\sS_0}$, subscripts of the form $\sS_0 \star  \dot{\sS_1/\sS_0}$ will typically be replaced by $\star$, etc.

We will now define certain maps $i \colon \HS_{\sS_1} \to \HS_\star$ and $i^* : \HS_\star \to \HS_{\sS_1}$. We recursively define \[i(\dot x) = \{ ((p, (\check{r}, \emptyset)^\bullet), i(\dot y)) : \psi(p, r, \emptyset) \wedge (r, \dot y) \in \dot x\},\] for any $\dot x \in \HS_{\sS_1}$. 
Then inductively,  for $(\pi, \check{\sigma}) \in (\sG, \sH)_\star$, $(\pi, \check{\sigma})(i(\dot x)) = i(\sigma(\dot x))$.
Thus if $H_1 \leq \sym_{\sS_1}(\dot x)$, then $(\sG, \check{H_1})_\star \leq \sym_\star(i(\dot x))$ and by induction $i(\dot x) \in \HS_{\star}$. Also, by induction note that in the context of case (1) of the Theorem, \begin{equation}\label{eq:1i}
    i(\dot x)^{G_0 * K} = \dot x^G \tag{$\alpha$}
\end{equation} and also in case (2), 
\begin{equation}\label{eq:2i}
    i(\dot x)^{K_0 * K} = \dot x^G. \tag{$\beta$}
\end{equation}

On the other hand, for $\dot x \in \HS_{\star}$, recursively  define
\[i^*(\dot x) = \{ (r, i^*(\dot y)) :  \exists ((p,\dot q), \dot y) \in \dot x ~\chi(r, p, \dot q) \},\] where \[\chi(r,p,\dot q) \leftrightarrow \exists p' \leq p, r \leq p' \exists X \in [\sR(\sS_0)]^{<\omega}  ( p' \Vdash  (\check{r}, \{((\dot z)\check{}, \dot z) : \dot z \in X \})^\bullet \leq \dot q).\]
Then note by induction that for $\sigma \in \sH$, $\sigma(i^*(\dot x)) = i^*((\sigma \restriction \mathbb{P}, \check{\sigma})(\dot x))$. 

If $(H_0, \check{H_1})_\star \leq \sym_\star(\dot x )$ and $H \in \sE$ is such that $H \restriction \mathbb{P} \leq H_0$, we thus have that $H \cap H_1 \leq \sym_{\sS_1}(i^*(\dot x))$. By induction we see that $i^*(\dot x) \in \HS_{\sS_1}$. As before, we defined $i^*(\dot x)$ in a way that in case (1), \begin{equation}\label{eq:1i*}
    \dot x^{G_0 * K} = i^*(\dot x)^G.\tag{$\gamma$}
\end{equation}

Let $\varphi$ be $\Delta_0$ and suppose that $\dot x \in \HS_{\sS_1}$ and $\Vdash_{\sS_1} \varphi(\dot x)$, but there is $p \in \mathbb{P}$, $r \in \mathbb{Q}$ and $(\dot x_j, \dot y_j)_{j <n}$ so that $\psi(p,r, (\dot x_j, \dot y_j)_{j <n})$ and \[(p,(\check{r}, \{ ((\dot x_j)\check{}, \dot y_j) : j < n \})) \Vdash_{\star} \neg\varphi(i(\dot x)).\] Let $\pi \in \sG$ be so that $\pi^{-1}(p) \parallel r$ and $p \Vdash \pi(\dot x_j) = \dot y_j$, for all $j <n$. Note that $\Vdash_{\star} (\pi^{-1},\id )(i(\dot x)) = i(\dot x)$. So by the symmetry lemma \[(\pi^{-1}(p),(\check{r}, \{ ((\dot x_j)\check{}, \dot x_j) : j < n \})) \Vdash_{\star} \neg\varphi(i(\dot x)).\]
If $G$ is any $\sS_1$-generic over $V$ containing $\pi^{-1}(p)$ and $r$, we have that $V[G]_{\sS_1} \models \varphi(\dot x^{G})$. Then if $G_0 * K$ is as in (1), it is $\sS_0 * \dot{\sS_1/\sS_0}$-generic by Proposition~\ref{prop:thirdquotient} and contains the condition above. By Proposition~\ref{prop:reduced} it is also $\sS_0 \star \dot{\sS_1/\sS_0}$-generic. By (\ref{eq:1i}), $\dot x^G = i(\dot x)^{G_0 * K}$, yielding a contradiction. 

Now suppose that $\dot x \in \HS_\star$ and $\Vdash_{\star} \varphi(\dot x)$, but there is $r \in \mathbb{Q}$ so that $r \Vdash_{\sS_1} \neg \varphi(i^*(\dot x))$. If $G$ is $\sS_1$-generic containing $r$, then again $i^*(\dot x)^G = \dot x^{G_0 * K}$ and we obtain a contradiction.  

(1)(a) and (b) are of course already covered by Proposition~\ref{prop:thirdquotient}. For (1)(c), (\ref{eq:1i}) and (\ref{eq:1i*}) imply that $V[G]_{\sS_1} = V[G_0 * K]_{\sS_0 \star \dot{\sS_1/\sS_0}} = V[G_0 * K]_{\sS_0 * \dot{\sS_1/\sS_0}}$.

For (2)(a), let $D \in V$ be an arbitrary $\sS_1$-symmetrically dense set and consider $\dot x_D = \{ (q, \emptyset) : q \in D \} \in \HS_{\sS_1}$. Then $\Vdash_{\sS_1} \dot x_D \neq \emptyset$, thus $\Vdash_{\star} i(\dot x_D) \neq \emptyset$. 
Applying the symmetric forcing theorem to the system $\sS_0 \star \dot{\sS_1/\sS_0}$, we have that $i(\dot x_D)^{K_0 * K} \neq \emptyset$. 
By (\ref{eq:2i}), $i(\dot x_D)^{K_0 * K} =  \dot x_D^{G}$, so $G \cap D\neq \emptyset$.

For (2)(c), let $(p,(\check{r}, \{ ((\dot x_j)\check{}, \dot y_j) : j < n \}))$ be a condition as earlier and suppose that for some $j' < n$, \[(p,(\check{r}, \{ ((\dot x_j)\check{}, \dot y_j) : j < n \})) \Vdash_{\sS_0 * \dot{\sS_1/\sS_0}} \dot y_{j'} \neq i(\dot x_{j'}).\] In the same way as before, we find an automorphism $\pi$ such that \[(\pi^{-1}(p),(\check{r}, \{ ((\dot x_j)\check{}, \dot x_j) : j < n \})) \Vdash_{\sS_0 * \dot{\sS_1/\sS_0}} \dot x_{j'} \neq i(\dot x_{j'}).\] Here note that since $\dot y_{j'}$ is an $\sS_0$-name, $(\pi^{-1}, \id)(\dot y_{j'}) = \pi^{-1}(\dot y_{j'})$ Again we reach a contradiction from (\ref{eq:1i}).

(2)(b) is similar to (c). Let $\dot x \in \sS_0$ and $(p,(\check{r}, \{ ((\dot x_j)\check{}, \dot y_j) : j < n \}))$ be arbitrary. Let $\pi \in \sG$ be as before. Then \[(\pi^{-1}(p),(\check{r}, \{ ((\dot x_j)\check{}, \dot x_j) : j < n \})) \not\Vdash_{\sS_0 * \dot{\sS_1/\sS_0}} i(\dot x) \neq \dot x,\] according to (\ref{eq:1i}). Going back, also \[(p,(\check{r}, \{ ((\dot x_j)\check{}, \dot y_j) : j < n \})) \not\Vdash_{\sS_0 * \dot{\sS_1/\sS_0}} \pi(\dot x) \neq i(\dot x).\] Thus, generically there is some $\pi \in \sG$ such that $\pi(\dot x)^{K_0} = i(\dot x)^{K_0 * K} = \dot x^{G} = \dot x^{G_0}$. In particular, if $\dot x = \HS_{\alpha, \sS_0}^\bullet$ for some $\alpha \in \Ord$, then \[(\HS_{\alpha, \sS_0}^\bullet)^{K_0} = \pi(\HS_{\alpha, \sS_0}^\bullet)^{K_0} = (\HS_{\alpha, \sS_0}^\bullet)^{G_0}.\] It follows that $V[K_0]_{\sS_0} = V[G_0]_{\sS_0}$.

(2)(d) now follows directly from (1)(c).
\end{proof}

The proof shows that there are maps $i$, $i^*$ mapping $\SN$-names between the systems and preserving the forcing relation for $\Delta_0$-formulas. But these maps are not inverses of each other and so $i$, $i^*$ fail to witness the strong equivalence. As noted earlier, this could also not be possible in general, at least not with our notion of complete subsystems. One can easily strengthen Definition~\ref{def:completesubsys} to essentially say that $\sS_1$ is of the form $\sS_0 * \dot{\sT}$, but that is beyond the point.

\begin{cor}\label{cor:forcingback}
    Let $\sS = (\mathbb{P}, \sG, \sF)$ be a symmetric system and $G_0$ be $\sS$-generic over $V$. If $K$ is $\mathbb{P} / \sS$ generic over $V[G_0]_\sS$, then $G = \dom K$ is $\mathbb{P}$-generic over $V$, $V[G_0]_{\sS} = V[G]_\sS$ and $V[G_0]_{\sS}[K] = V[G] = V(K)$.
\end{cor}

\begin{proof}
    $\sS \lessdot \sS_1 = (\mathbb{P}, \sG, \sE)$, where $\sE$ is generated by $\{ \id\}$. Then $K$ is $\sS_1/\sS_0$-generic over $V[G_0]_{\sS}$, so $G = \dom K$ is $\sS_1$-generic over $V$, which is the same as being $\mathbb{P}$-generic over $V$ and $V[G_0]_{\sS} = V[G]_\sS$. Clearly $V[G_0]_{\sS}[K]= V[G] = V(K)$ as $K$ can be defined from $G$ over $V$ and vice-versa.
\end{proof}

Applying Theorem~\ref{prop:fullquotient} to the special case of $(\mathbb{P}, \sG, \sF) \lessdot (\mathbb{P},\{\id \}, \{\{\id \}\})$ we obtain the following.

\begin{thm}\label{thm:quotsummary}
    Let $\sS = (\mathbb{P}, \sG, \sF)$ be a symmetric system.
    \begin{enumerate}
        \item For any $\sS$-generic $G$ over $V$, there is a $\mathbb{P}$-generic $H$ over $V$ so that $V[G]_\sS = V[H]_\sS$. Moreover, such $H$ is added generically over $V[G]_\sS$ by $\mathbb{P}/ \sS$.
        \item For any $\mathbb{P}$-generic $G$ over $V$, $V[G]$ is a $\mathbb{P}/\sS$-generic forcing extension of $V[G]_\sS$.
    \end{enumerate}
\end{thm}

This shows that the symmetric extensions formed by symmetric generics agree with those that use full generics. 

Let us note that the only place where the notion of a respect basis really is used is in Lemma~\ref{lem:canondense}. And all that this is needed for is to show that the canonical filter $G \times \{ (\dot x, \dot x^G) : \dot x \in \mathcal{R}(\sS) \}$ is a particular generic. Every other argument essentially goes by moving a particular condition into this canonical generic, using that this one has a particular desirable property, and thus showing that the original condition couldn't force the opposite either.

\section{The seed of a symmetric extension}
Recall that for an inner model $W$ and a class $X$ of $W$, $\HOD^W_X$ consists of the sets which are hereditarily ordinal definable in $W$ using additional parameters from $X$. If $V \subseteq W$ is an inner model of $W$ and $x \in W$ is a set, $V(x)$ is the smallest inner model containing $V \cup \{ x\}$ as a subclass. 

\begin{thm}\label{thm:V(P/S)}
   Let $\sS = (\mathbb{P}, \sG, \sF)$ be a symmetric system and $G$ be $\sS$-generic over $V$. Then $V[G]_\sS = V(\mathbb{P}/\sS)$. When $G$ is $\mathbb{P}$-generic over $V$, then also $V[G]_\sS = \HOD^{V[G]}_{V(\mathbb{P}/\sS)}$.
\end{thm}

The second equality follows from the first and the fact that $V[G]$ is a generic extension of $V(\mathbb{P}/\sS)$ via a homogeneous forcing. The equality $V[G]_\sS = V(\mathbb{P}/\sS)$ follows immediately from Corollary~\ref{cor:forcingback} and the following general observation: 

\begin{lemma}
    Let $V$ be an inner model of $M$, $\mathbb{Q} \in M$ and $G$ be $\mathbb{Q}$-generic over $M$ so that $M \subseteq V(G)$. Then $M = V(\mathbb{Q})$.
\end{lemma}

\begin{proof}
    Clearly $\mathbb{Q} \in V(\mathbb{Q}) \subseteq M$. In particular, $G$ is also $\mathbb{Q}$-generic over $V(\mathbb{Q})$ and $V(\mathbb{Q})[G] = V(G) \supseteq M$. Let $x \in M$ be arbitrary and $\dot x$ be a $\mathbb{Q}$-name in $V(\mathbb{Q})$ so that $\dot x^G=x$. Working in $M$, there is some $p \in G$ so that $p \Vdash \dot x = \check x$. Back in $V(\mathbb{Q})$, we find that for arbitrary $y$, necessarily either \[p \Vdash \check{y} \in \dot x,\] or \[p \Vdash \check{y} \notin \dot x,\] depending on whether $y \in x$ or not, since otherwise we contradict the fact that in $M$, $p \Vdash \dot x = \check{x}$. Thus we have shown that $x \cap V(\mathbb{Q}) \in V(\mathbb{Q})$. An easy rank-induction shows that in fact $x \in V(\mathbb{Q})$.
\end{proof}

Let us remark that Theorem~\ref{thm:V(P/S)} offers a different way to prove that the product of two systems corresponds to the two-step iteration, namely by showing that $V[G*H]_{\sS * \check{\sT}} = V(\mathbb{P}/\sS)(\mathbb{Q}/\sT) = V(\mathbb{P} \times \mathbb{Q} / \sS \times \sT) = V[ G \times H]_{\sS \times \sT}$. In most textbooks on forcing the equivalence of a product with a two-step iteration is shown in a similar way, using that $V[G]$ is the smallest transitive model containing $V$ and $G$, in other words $V[G] = V(G)$.

\begin{prop}\label{prop:respectbasefromHOD}
    Let $\sS= (\mathbb{P}, \sG, \sF)$ be a symmetric system and $\dot x \in \HS$ be such that $\sym(\dot x) = \sG$ and $\Vdash V[\dot G]_\sS = \HOD_{V(\dot x)}^{V[\dot G]}$. Then \[\sR = \{ (\dot y_0, \dots, \dot y_n)^\bullet :~ n \in \omega, \dot y_0, \dots \dot y_n \text{ appear in } \dot x \}\] is a respect-basis for $\sS$.
\end{prop}

Note that $\Vdash V[\dot G]_\sS = V(\dot x)$ also implies that $\Vdash V[\dot G]_\sS = \HOD_{V(\dot x)}^{V[\dot G]}$.

\begin{proof}
    Clearly $\sR$ is closed under $\sG$ as $\sym(\dot x) = \sG$. Let $\dot z \in \HS$ and $p \in \mathbb{P}$ be arbitrary. Note that $\HOD_{V(\dot x)}^{V[\dot G]} = \HOD_{V \cup \trcl(\{\dot x\})}^{V[\dot G]}$. Thus there is $q \leq p$, $u_0, \dots, u_k \in V$,  $\dot y_0, \dots, \dot y_n$ appearing in $\dot x$, $\alpha \in \Ord$ and a formula $\varphi$ so that \[q \Vdash_{\mathbb{P}} \dot z = \{ v \in \HS_\alpha^\bullet : \varphi(v, \check{u}_0, \dots, \check{u}_k, \dot y_0, \dots, \dot y_n, \dot x) \}.\]
    Letting $\dot z'$ be a name for the right-hand side, it is immediate to check that \[R_\sS(\dot y_0, \dots, \dot y_n) \subseteq R_\sS(\dot z').\qedhere\]
\end{proof}

\begin{prop}
    Let $\sS= (\mathbb{P}, \sG, \sF)$ be a symmetric system, $\mathcal{R}$ a respect-basis for $\sS$, and $\dot \Gamma = \{ \pi(\dot G) : \pi \in \sG \}^\bullet$, where $\dot G$ is the canonical name for the $\mathbb{P}$-generic. Let $G$ be $\mathbb{P}$-generic over $V$, $\Gamma = \dot \Gamma^G$ and $R = (\mathcal{R}^\bullet)^G$. Then \[V[G]_\sS = \HOD_{V \cup R \cup \{\Gamma\}}^{V[G]} = \HOD_{V \cup R \cup \{ \mathbb{P}/\sS \}}^{V[G]_\sS}.\] Moreover, there is a map $F$ definable in $V[G]_\sS$ from the parameter $\mathbb{P}/\sS$ that surjects $R \times V$ onto $V[G]_\sS$.\footnote{We do not claim that $R \times V$ is the domain of $F$. This would imply that $V$ is definable in $V[G]_\sS$ which is an open problem in general (but known to be true when $V\models \ZFC$).}
\end{prop}

\begin{proof}
    $\HOD_{V \cup R \cup \{\Gamma\}}^{V[G]} \subseteq V[G]_\sS$ is shown by $\in$-induction. If $x \subseteq V[G]_\sS$ and \[x = \{ y : V[G] \models \varphi(y, v_0, \dots, v_n, r_0, \dots, r_m, \Gamma) \},\] then consider \[\dot x = \{ (p, \dot y) : p \Vdash_{\mathbb{P}} \varphi(\dot y, \check{v}_0, \dots, \check{v}_n, \dot r_0, \dot r_m, \dot \Gamma) \},\] where $\dot r_0^G = r_0, \dots, \dot r_m^G = r_m$, $\dot r_0, \dots, \dot r_m \in \mathcal{R}$. Then $\dot x \in \HS$ and $\dot x^G = x$. 

    On the other hand, let $\dot x \in \HS$ be arbitrary. Since $\mathcal{R}$ is a respect basis, there is $\dot x'\in \HS$ and $\dot r \in \mathcal{R}$ with $\dot x^G = \dot x'^G$ and $R_s(\dot r) \subseteq R_s(\dot x')$. Let $r = \dot r^G$. Consider $\varphi(y, \dot x', \dot r, \dot r^G, \Gamma)$ expressing that for some $H \in \Gamma$ such that $\dot r^H = r$, $y \in \dot x'^H$. Whenever $H \in \Gamma$, then $H = \pi``G$ for some $\pi \in \sG$. Then there must be some $p \in G$ so that $p \Vdash \dot r = \pi^{-1}(\dot r)$. Then also $p \Vdash \dot x' = \pi^{-1}(\dot x')$ and $x = \dot x'^G = \dot x^{\pi``G} = \dot x^H$. This shows that $\varphi$ works. 

    For the last part, we define $F(r, v) = x$ in $V[G]_\sS$ so that when $v$ is a pair of $\sS$-names $(\dot x, \dot r)$ with $\dot r \in \mathcal{R}$, $y \in x$ iff it is forced by $\mathbb{P}/\sS$ that $\varphi(\check{y}, (\dot x)\check{}, (\dot r)\check{}, \check{r}, \dot \Gamma')$, where $\dot \Gamma'$ is a $\mathbb{P}/\sS$-name so that when $K$ is generic, $\dot \Gamma'^K = \{ \pi``\dom(K) : \pi \in \sG \}$. We note that the automorphisms of $\mathbb{P}/\sS$ that witness the homogeneity in Lemma~\ref{lem:homquotient} fix the name for $\dot \Gamma'$, so $\varphi$ is indeed decided by the trivial condition. 
    
  This immediately also shows that $V[G]_\sS = \HOD_{V \cup R \cup \{ \mathbb{P}/\sS \}}^{V[G]_\sS}$.
\end{proof}

\section{The completion of a system}

Let $\mathbb{P}$ be a forcing notion. Then we say that a class $\dot X$ is a \emph{class name} if $\dot X$ consists of pairs $(p, \dot x)$ where $p \in \mathbb{P}$ and $\dot x$ is a (usual set-sized) $\mathbb{P}$-name. 

\begin{definition}
    Let $\mathbb{P}$ be a forcing notion and $\dot X$ be a class name. Then we define the system $\sO(\mathbb{P}, \dot X) = (\mathbb{P}, \sG, \sF)$, where $\sG = \{ \pi \in \Aut(\mathbb{P}) :~ \Vdash \pi(\dot X) = \dot X \}$ and $\sF$ is generated by groups of the form $\res_\sG(\dot x)$ for $\mathbb{P}$-names $\dot x$ such that $\Vdash \dot x \in \dot X$.
\end{definition}

To see that $\sF$ is a normal filter, simply note that if $\Vdash \pi(\dot X) = \dot X$ and $\Vdash \dot x \in \dot X$, then also $\Vdash \pi(\dot x ) \in \dot X$ and $\pi\res_\sG(\dot x)\pi^{-1} = \res_\sG(\pi(\dot x))$.

\begin{definition}\label{def:reflect}
    Let $\mathbb{P}$ be a forcing notion and $\dot X$ be a class name. Then we say that \emph{$\dot X$ reflects to a $\mathbb{P}$-name $\dot x$} if \[\{ \pi \in \Aut(\mathbb{P}) :~ \Vdash \pi(\dot X) = \dot X \} = \{ \pi \in \Aut(\mathbb{P}) :~ \Vdash \pi(\dot x) = \dot x \}.\] We say that $\dot X$ \emph{self-reflects} if $\dot X$ reflects to $\dot x$ with $\Vdash \dot x \in \dot X$. 
\end{definition}

 \begin{lemma}[see, e.g., {\cite[Lemma 25.5]{Jech1997}}]
       Let $\mathbb{P}$ be a complete Boolean algebra. For any $\mathbb{P}$-generic filters $G,H$ over $V$ with $V[H] = V[G]$ and any class name $\dot X$ in $V$ with $\dot X^H = \dot X^G$, there is $\pi \in \Aut(\mathbb{P})^{V}$ so that $H = \pi `` G$ and $\Vdash \dot X = \pi^{-1}(\dot X)$. 
    \end{lemma} 

\begin{prop}\label{prop:nameclassHOD}
    Let $\mathbb{P}$ be a complete Boolean algebra, $\dot X$ be a class name and $G$ be $\mathbb{P}$-generic over $V$. Let $\sS = \sO(\mathbb{P}, \dot X)$ and $X = \dot X^G$. Then $V[G]_{\sS} = \HOD^{V[G]}_{V \cup X \cup \{X\}}$. Moreover, when $\dot X$ self-reflects, then $V[G]_{\sS} = \HOD^{V[G]}_{V \cup X}$.
\end{prop}

\begin{proof}
    First let us check that $\HOD_{V \cup X \cup \{ X \}}^{V[G]} \subseteq V[G]_\sS$. Let $a \subseteq V[G]_\sS$ be arbitrary such that $a = \{b : V[G] \models \varphi(b, \bar x, \bar y, \bar \alpha, X) \}$, for $\bar x \in X^{<\omega}$, $\bar y \in V^{<\omega}$ and $\bar \alpha \in \Ord^{<\omega}$. Furthermore, let $\bar x = (\dot x_0^G, \dots, \dot x_n^G)$, where $\Vdash \dot x_i \in \dot X$ for $i\leq n$. Then consider the name \[\dot a = \{ (p, \dot b): \dot b \in \HS_\gamma \wedge p \Vdash \varphi(b,\dot x_0, \dots, \dot x_n , \bar y, \bar \alpha, \dot X) \},\] for large enough $\gamma$ and note that $\dot a \in \HS$ and $\dot a^G = a$. Thus $a \in V[G]_\sS$. An $\in$-induction shows that the inclusion holds.

    In the other direction, note that the set $\Gamma$ of $\mathbb{P}$-generics over $V$ in $V[G]$ is definable using only parameters from $V$. Let $\dot z \in V$ be an arbitrary $\sS$-name and suppose that $\res_\sG(\dot x_0, \dots, \dot x_n) \leq \sym(\dot z)$, where $\Vdash \dot x_0, \dots, \dot x_n \in X$. Then consider \[z := \bigcup \{ \dot z^H : H \in \Gamma \wedge \dot X^H = X\wedge \forall i \leq n( \dot x_i^H = \dot x_i^G)\},\] which clearly is $\OD_{V \cup X \cup \{ X\}}^{V[G]}$. In either case we claim that $z = \dot z^G$. Clearly $\dot z^G \subseteq z$. On the other hand, for any $H$ as above there is some $\pi \in \Aut(\mathbb{P})^{V}$ so that $H = \pi `` G$ and $\Vdash \dot X = \pi^{-1}(\dot X) \wedge \forall i \leq n( \dot x_i = \pi^{-1}(\dot x_i))\}$. In particular, $\pi^{-1} \in \res_\sG(\dot x_0, \dots, \dot x_n)$ and thus $\dot z^G = \pi^{-1}(\dot z)^G = \dot z^{\pi``G} = \dot z^H$. Since $H$ was arbitrary, we have that $z = \dot z^G$ and we have shown that $\dot z^G \in \OD_{V \cup X \cup \{X\}}^{V[G]}$. By induction, $V[G]_\sS \subseteq \HOD_{V \cup X \cup \{X\}}^{V[G]}$.

     When $\dot X$ reflects to a name $\dot x$, we may replace $X$ by $\dot x^G$ and $\dot X$ by $\dot x$ in the definition of $z$ and obtain an $\OD_{V \cup X}^{V[G]}$ set.
\end{proof}

\begin{definition}
    Let $\sS = (\mathbb{P}, \sG, \sF)$ be a symmetric system. Then the \emph{completion} of $\sS$ is defined as $\hat{\sS} = \sO(\B(\mathbb{P}), \SN_\sS^\bullet)$, where $\mathbb{P}$-names are, as usual, also viewed as $\B(\mathbb{P})$-names. 
\end{definition}

\begin{thm}\label{thm:completion}
    Let $\sS = (\mathbb{P}, \sG, \sF)$ be a symmetric system. Then $\sS \cong_\SN \hat{\sS}$ and whenever $G$ is $\B(\mathbb{P})$-generic over $V$, then $V[G]_{\hat{\sS}} =  V[G\cap \mathbb{P}]_\sS$. Moreover, $\hat{\sS}$ is tenacious and $\SN_{\hat{\sS}} = \HR_{\hat{\sS}}$.
\end{thm}

\begin{proof}
    It is immediate to see that $\SN_\sS^\bullet$ self-reflects. Moreover, ${\SN_\sS^\bullet}^G = V[G]_\sS$. Thus $V[G]_{\hat{\sS}} = \HOD^{V[G]}_{V[G]_\sS} = \HOD^{V[G]}_{V(\mathbb{P}/\sS)} = V[G]_\sS$. The natural identification of $\mathbb{P}$- and $\B(\mathbb{P})$-names witnesses that $\sS \cong_\SN \hat{\sS}$, see Theorem~\ref{thm:tenequilr}. Then it is immediate from the definition of $\hat{\sS}$ that $\SN_{\hat{\sS}} = \HR_{\hat{\sS}}$. To see that $\hat{\sS}$ is tenacious, simply note that for any $p \in \mathbb{P}$, $\dot x_p = \{ (p, \emptyset) \} \in \SN_\sS$ and that to respect $\dot x_p$, $p$ must be fixed as a condition in a separative poset. 
\end{proof}

The completion of the system $\sS$ can be seen, in a sense, as the largest most canonical system that produces the same models as $\sS$ with the same generics. All groups that somehow relate to the symmetry of the symmetric extension are in the filter of $\hat{\sS}$.

\begin{prop}
    For any system $\sS$, $\hat{\hat{\sS}}=\hat{\sS}$. That is, $\hat{\sS}$ is its own completion.
\end{prop}

\begin{proof}
We simply have that $\Vdash \SN_{\sS}^\bullet = \SN_{\hat{\sS}}^\bullet$.
\end{proof}

The following is also immediate.

\begin{lemma}
$\sS \cong_\SN \sT$ iff $\hat{\sS}$ and $\hat{\sT}$ are isomorphic, meaning that up to a relabelling of the forcing notions the two systems are equal. 
\end{lemma}

The completion has some additional desirable properties such as the following. 

\begin{lemma}\label{lem:maximalitycompletion}
If $\AC$ holds, then $\hat \sS$ has mixing, i.e., for any $p \in \hat{\mathbb{P}}$, any $\dot x_0, \dots, \dot x_n \in \HS$ and any formula $\varphi(x,\dot x_0, \dots, \dot x_n)$ in the language of set theory, if \[p \Vdash_{\hat\sS} \exists x \varphi(x,\dot x_0, \dots, \dot x_n),\] then there is $\dot x \in \HS$ such that \[p \Vdash_{\hat\sS} \varphi(\dot x,\dot x_0, \dots, \dot x_n).\]
\end{lemma}

\begin{proof}
    This follows immediately from the maximality principle of forcing, that $\SN_{\hat\sS} = \HR_{\hat\sS}$ and the fact that any $\HR$-name is forcing equivalent to an $\HS$-name (see Lemma~\ref{lem:hshr}).
\end{proof}

\section{Characterizing symmetric extensions}
\label{sec:characterising}

\begin{thm}\label{thm:HODsubsystem}
 Let $\sS = (\mathbb{P}, \sG, \sF)$ be a symmetric system and $G$ be $\mathbb{P}$-generic over $V$. Let $V \subseteq M \subseteq V[G]_\sS$. Then the following are equivalent: 
 
 \begin{enumerate}
  \item $M$ is of the form $\HOD_{V(x)}^{V[G]}$, for some $x \in V[G]_\sS$.
  \item $M = V[G]_{\sT}$ for some $\sT \lessdot \hat{\sS}$ where the forcing of $\sT$ is $\B(\mathbb{P})$.
 \end{enumerate}

 \begin{proof}
     The implication from (2) to (1) follows from Theorem~\ref{thm:V(P/S)}. In the other direction, let $\dot x \in \SN_\sS$ be such that $x = \dot x^G$. Consider a name $\dot X$ for $\trcl(\{ \dot x \})$ and note that $\dot X$ reflects to $\dot x$ (see Definition~\ref{def:reflect}). Thus $\dot X$ is self-reflecting. Letting $\sT = \sO(\B(\mathbb{P}), \dot X)$ it is immediate to check that $\sT \lessdot \hat{\sS}$. By Proposition~\ref{prop:nameclassHOD}, $V[G]_\sT = \HOD^{V[G]}_{V \cup \trcl(\{ x \})} = \HOD^{V[G]}_{V(x)}$. 
 \end{proof}

\end{thm}

Theorem~\ref{thm:HODsubsystem} is a vast generalization of the following result of Grigorieff. 

\begin{cor}[{\cite[\S8.1, Theorem 3]{Grigorieff1975}}]\label{lem:hodsaresymmetricext}
    Let $V[G]$ be a $\B(\mathbb{P})$-generic forcing extension of $V$. Then the models of the form $\HOD_{V(x)}^{V[G]}$ are exactly the symmetric extensions of $V$ of the form $V[G]_\sS$.
\end{cor}

\begin{proof}
   Simply use Theorem~\ref{thm:HODsubsystem} and the fact that $V[G]$ is a symmetric extension of $V$ via $(\mathbb{P}, \{ \id \}, \{\{ \id \}\})$. The completion of this system is $(\B(\mathbb{P}), \Aut(\B(\mathbb{P})), \sF)$, where $\sF$ is generated by the trivial group and every system using $\B(\mathbb{P})$ as a forcing notion is a complete subsystems of this one. 
\end{proof}

\begin{cor}
 Let $V[G]$ be a generic extension of $V$. Then there are at most set-many intermediate models of the form $\HOD_{V(x)}^{V[G]}$ between $V$ and $V[G]$.
\end{cor}

\begin{prop}[{\cite[\S6.1, Theorem 1]{Grigorieff1975}}]\label{lem:V(x)toV[G]}
     Let $V[G]$ be a $\mathbb{P}$-generic forcing extension of $V$ and $x \in V[G]$. Then $V[G]$ is a forcing extension of $V(x)$.
\end{prop}

\begin{proof}
    Consider the forcing $\Coll(\omega, \trcl(\{ x \})) \in V(x)$, which consists of finite partial functions from $\omega$ to $\trcl(\{ x \})$ and let $H$ be generic over $V[G]$ for this. $\bigcup H$ is a generic surjection $f$ from $\omega$ to $\trcl(\{ x \})$. 
    
    The set $E = \{ (n,m) \in \omega \times \omega : f(n) \in f(m) \} \in V(x)[H]$ codes $x$ and the generic $H$ over $V$, so we have that $V(x)[H] = V[E]$. Since $E\subseteq V$, $V[E]$ is an intermediate forcing extension between $V$ and $V[G][H]$ (see, e.g., \cite[Corollary~15.42]{Jech2003}), and in particular there is a quotient forcing notion showing that $V[G][H]$ is a generic extension of $V[E]$. All in all, $V[G][H]$ is a forcing extension of $V(x)$. Again, since $G \subseteq V(x)$, $V[G] \subseteq V[G][H]$ is an intermediate forcing extension of $V(x)$.  
\end{proof}
\begin{thm}\label{thm:characterise}
Let $V \subseteq M$. Then the following are equivalent: 

\begin{enumerate}
    \item $M$ is a symmetric extension of $V$,
    \item $M = V(x)$, for some $x \in V[G]$, where $G$ is generic over $V$, 
    \item $M = \HOD_{V(x)}^{V[G]}$, for some $x \in V[G]$, where $G$ is generic over $V$.
\end{enumerate}
\end{thm}

\begin{proof}
    The equivalence of (1) and (3) follows from Corollary~\ref{lem:hodsaresymmetricext}. By Theorem~\ref{thm:V(P/S)} and (1) of Theorem~\ref{thm:quotsummary}, every symmetric extension of $V$ is of the form $V(x)$ in some generic extension of $V$. On the other hand, suppose that $x \in V[G]$, for $G$ $\mathbb{P}$-generic over $V$, for some $\mathbb{P} \in V$. By Proposition~\ref{lem:V(x)toV[G]} $V[G]$ is a forcing extension of $V(x)$, say via a forcing notion $\mathbb{Q} \in V(x)$. Consider the finite support product $\mathbb{Q}^{<\omega}$ of $\omega$-many copies of $\mathbb{Q}$ and let $H$ be $\mathbb{Q}^{<\omega}$-generic over $V[G]$. Then $V[G][H]$ is also a $\mathbb{Q}^{<\omega}$-generic extension of $V(x)$. Using the homogeneity of $\mathbb{Q}^{<\omega}$ it is easy to verify that $\HOD_{V(x)}^{V[G][H]} = V(x)$. By Corollary~\ref{lem:hodsaresymmetricext}, $V(x)$ is a symmetric extension of $V$.
\end{proof}

Not all models $V(x) \subseteq V[G]$ can be obtained as a symmetric extension using the generic $G$. There are clearly only set-many such models but there can be a proper class of models $V(x)$ between $V$ and $V[G]$, as is the case with a Cohen extension of $V$ (see \cite{Karagila2018} and \cite{ShaniHayut}). One interesting corollary of this observation is that $V(x)$ is not usually closed under ordinal definability. Namely, it can be (and in some cases it is often the case) that $V(x)\neq\HOD_{V(x)}^{V[G]}$.

To present $V(x)$ as a symmetric extension we had to pass to a larger forcing notion $\mathbb{P}^+ = \mathbb{P} * \dot{\mathbb{Q}}^{<\omega}$. $\mathbb{P}^+$ itself can also be seen as just a threshold for this to happen, as $V(x)$ can still be written as $\HOD_{V(x)}$ in any further homogeneous extension. This suggests the following definition.

\begin{definition}
    Let $\sS = (\mathbb{P}, \sG, \sF)$ be a symmetric system and $\mathbb{P} \lessdot \mathbb{Q}$. Then we define the system $\sS^{\langle \mathbb{Q}\rangle}$ as the completion of $\sO(\B(\mathbb{Q}), \SN_{\sS}^\bullet)$, where $\mathbb{P}$-names are viewed as $\B(\mathbb{Q})$-names as usual. Moreover, we write $\sS^{\langle \alpha \rangle}$ for $\sS^{\langle \mathbb{Q}\rangle}$, where $\mathbb{Q} = \mathbb{P} \times \Coll(\omega, V_\alpha)$.
\end{definition}

$\sS^{\langle \mathbb{P}\rangle} = \sS^{\langle 0 \rangle}$ is simply the completion of $\sS$. Whenever $\mathbb{Q}$ is a homogeneous extension of $\mathbb{P}$ (meaning that the quotient forcing is forced to be weakly homogeneous), $V[G \cap \mathbb{P}]_\sS = V[G]_{\sS^{\langle \mathbb{Q}\rangle}}$, for any $\mathbb{Q}$-generic $G$. Recall that any complete extension $\mathbb{Q}$ of $\mathbb{P}$ completely embeds into $\B(\mathbb{P} \times \Coll(\omega, V_\alpha))$, for large enough $\alpha$, fixing $\mathbb{P}$. We may think of $\sS^{\langle \alpha\rangle}$ as higher order completions of $\sS$ and we will see that these suffice to characterize all intermediate symmetric extensions. 

\begin{thm}
    Let $V\subseteq M \subseteq N$, be symmetric extensions of $V$, where $N$ is obtained via $\sS$. Then there is $\alpha$, $\sT \lessdot \sS^{\langle \alpha\rangle}$ and $G$ generic over $V$ so that $M = V[G]_{\sT}$, $N = V[G]_{\sS^{\langle \alpha\rangle}}$.
\end{thm}

\begin{proof}
    Let $M = V(x)$, $\sS = (\mathbb{P}, \sG, \sF)$ and $N = V[G_0]_\sS$, where $G_0$ is $\mathbb{P}$-generic over $V$. Then $V[G_0]$ is a $\mathbb{Q}$-generic forcing extension of $M$ for some $\mathbb{Q} \in M$. There is a large enough $\alpha$ so that $\mathbb{P} * \dot{\mathbb{Q}}^{<\omega}$ completely embeds into $\B(\mathbb{P} \times \Coll(\omega, V_\alpha))$, where $\dot{\mathbb{Q}}$ is an $\sS$-name for $\mathbb{Q}$. Letting $G_1$ be $\Coll(\omega, V_\alpha)$-generic over $V[G_0]$ we then have that \[M = \HOD^{V[G_0,G_1]}_{V(x)} \text{ and }V[G_0]_\sS = V[G_0,G_1]_{\sS^{\langle\alpha \rangle}}.\]
    The rest follows from Theorem~\ref{thm:HODsubsystem}.
\end{proof}

The following is interesting to note and mimics the fact forcing notions $\mathbb{P}, \mathbb{Q}$ are weakly equivalent exactly when they have strongly equivalent lottery sums $\bigoplus_{i \in \alpha} \mathbb{P} \cong \bigoplus_{i \in \beta} \mathbb{Q}$. 

\begin{prop}
   Let $\sS$, $\sT$ be symmetric systems. The following are equivalent: 

   \begin{enumerate}
       \item $\sS$ and $\sT$ are weakly equivalent. 
       \item $\sS^{\langle \alpha \rangle} \cong_\SN \sT^{\langle \beta \rangle}$, for some $\alpha, \beta$.
       \item $\sS^{\langle \alpha \rangle} \cong_\SN \sT^{\langle \alpha \rangle}$, for some $\alpha$.
   \end{enumerate}
\end{prop}

\begin{proof}
    (3) implies (2) implies (1) is clear. It suffices to show that (1) implies (3). Suppose that $\sS = (\mathbb{P}, \sG, \sF)$ and $\sT = (\mathbb{Q}, \sH, \sE)$ are weakly equivalent. Then for any $\mathbb{P}$-generic $G$ over $V$, there is a forcing notion $\mathbb{A} \in V[G]_\sS$ adding over $V[G]_\sS$ a $\mathbb{Q}$-generic $H$ over $V$ so that $V[G]_\sS = V[H]_\sT$ and $\mathbb{A} =  (\mathbb{Q} \dot{ / }\sT)^H$ (see Theorem~\ref{thm:quotsummary}). Let $\dot{\mathbb{A}}$ be the closed $\sS$-name for the lottery sum of all forcing notions with this property, of which there can only be set many as the rank of any $\mathbb{Q}/\sT$ is bounded. Then there is a $\mathbb{P} * \dot{\mathbb{A}}$-name $\dot H$ for the $\mathbb{Q}$-generic, inducing a complete subforcing $\mathbb{Q}' = \{ \llbracket \check q \in \dot H \rrbracket : q \in \mathbb{Q} \}$ of $\B(\mathbb{P} * \dot{\mathbb{A}})$. We then see that $\mathbb{Q}'$ is in fact isomorphic to $\mathbb{Q}$, as $\llbracket \check q \in \dot H \rrbracket \neq \mathds{O}$ for each (positive) $q \in \mathbb{Q}$. More precisely, $\dot H$ can be realized as any generic $H$ over $V$, as $V[H]_\sT = V[G]_\sS$ for some $G$, so $(\mathbb{Q} \dot{ / }\sT)^H \in V[G]_\sS$ appears as a forcing $\mathbb{A}$ as above. Thus, identifying $\mathbb{Q}$ and $\mathbb{Q}'$, we may view $\SN_\sT^\bullet$ as a $\B(\mathbb{P} * \dot{\mathbb{A}})$-class-name and we obtain that $\Vdash_{\B(\mathbb{P} * \dot{\mathbb{A}})} \SN_\sS^\bullet = \SN_\sT^\bullet$. Moreover, there is a $\mathbb{Q}$-name $\dot{\mathbb{B}}$ such that $\B(\mathbb{Q} *\dot{\mathbb{B}}) \cong \B(\mathbb{P} * \dot{\mathbb{A}})$, respecting the identification of $\mathbb{Q}$ and $\mathbb{Q}'$. It suffices to find a large enough $\alpha$ so that $\B(\mathbb{P} * \dot{\mathbb{A}})$ and $\B(\mathbb{Q} *\dot{\mathbb{B}})$ completely embed into $\B(\mathbb{P} \times \Coll(\omega, V_\alpha))$ and $\B(\mathbb{Q} \times \Coll(\omega, V_\alpha))$, respecting the copies of $\mathbb{P}$ and $\mathbb{Q}$. The rest follows easily.
\end{proof}

\section{Kinna--Wagner Principles}

We can treat $\AC$ as stating that every set can be mapped injective into an ordinal, or that every set is the surjective map of an ordinal. This lends itself to a hierarchy of principles, called Kinna--Wagner Principles. Kinna and Wagner \cite{KW} defined the first principle as a weak selection principle (that is, we choose a proper subset, rather than an element), and they proved their principle is equivalent to the statement ``Every set injects into the power set of an ordinal''. Their principle was studied extensively, as did weaker selection principles. Monro suggested the following weakening (for finite $\alpha$) in \cite{Monro:KW}.
\begin{definition}\label{def:KW}
    For an ordinal $\alpha$, $\KW_\alpha$ states that every set $x$ injects into $\mathcal{P}^\alpha(\Ord)$. The \emph{Kinna--Wagner Principle} $\KW$ states that there is some $\alpha$ such that $\KW_\alpha$ holds.
\end{definition}

It turns out that forcing is intimately related to a further weakening, which is the Dual Kinna--Wagner Principle.

\begin{definition}[Dual Kinna--Wagner]
    For an ordinal $\alpha$, $\KW^*_\alpha$ states that for every set $x$, there is a surjection from $\mathcal{P}^\alpha(\Ord)$ to $x$.
\end{definition}

What we mean of course is that there is a definable surjection from $\mathcal{P}^\alpha(\Ord)$ to $x$. But this is just the same as to say that there is some large enough ordinal $\eta$, so that there is a surjection from the set $\mathcal{P}^\alpha(\eta)$ to $x$.

These principles were studied extensively by the authors in \cite{KaragilaSchilhan2024}, where the following facts are proved.

\begin{prop}[Balcar--Vop\v{e}nka--Monro, {\cite[Theorem~3.5]{KaragilaSchilhan2024}}]\label{prop:bvm}
  Let $M$ and $N$ be models of $\ZF$ and $M\models\KW^*_\alpha$. If $\mathcal{P}^{\alpha+1}(\Ord)^M=\mathcal{P}^{\alpha+1}(\Ord)^N$, then $M=N$. \qed
\end{prop}

\begin{prop}[{\cite{KaragilaSchilhan2024}}]\label{prop:KWfacts}
    For any ordinal $\alpha$, we have that
    \begin{enumerate}
        \item $\KW_\alpha\rightarrow\KW^*_{\alpha}\rightarrow\KW_{\alpha+1}$,
        \item if $\alpha$ is not a successor then $\KW_\alpha^*$ and $\KW_\alpha$ are equivalent,
        \item $\KW^*_\alpha$ is invariant under forcing,
        \item $\KWP$ and $\neg \KWP$ are invariant under symmetric extensions.
    \end{enumerate}
     For any ordinals $\alpha, \beta$, any $\mathbb{P}$-generic filter $G$ over $V$ and $x \in V[G]$, we have that
        \begin{enumerate}\setcounter{enumi}{4}
                \item if $V \models \KWP_\beta^*$ and $x \subseteq \mathcal{P}^\alpha(\Ord)$, then $V(x) \models \KWP^*_{\max(\alpha, \beta)}$,
                \item if $V[G] \models \KWP^*_\alpha$ and $\mathbb{P} \subseteq \mathcal{P}^\beta(\Ord)$, then $V \models \KWP_{\beta + \alpha}^*$. \qed
    \end{enumerate}
\end{prop}

\begin{lemma}
    There is a uniformly definable bijection between $\mathcal{P}^\alpha(\ON)^{<\omega}$ and $\mathcal{P}^\alpha(\ON)$, for ordinals $\alpha$, that is absolute between transitive models.\qed
\end{lemma}

\begin{thm}[{\cite[\S4]{KaragilaSchilhan2024}}]
    Let $V\models\ZF+\KWP$, $G$ be $\mathbb{P}$-generic over $V$ for some $\mathbb{P}\in V$ and $M\models\ZF$ such that $V\subseteq M\subseteq V[G]$. Then the following are equivalent:
    \begin{enumerate}
        \item $M\models\KW$.
        \item $M=V(x)$ for some $x\in M$.
        \item $M$ is a symmetric extension of $V$.
    \end{enumerate}
    Moreover, if $M\models\KWP_\alpha^*$, then $M=V(\mathcal{P}^{\alpha+1}(\mathbb{P}^{<\omega})^M)$.
\end{thm}

With our notion of a respect-basis, we can prove the following theorem.
\begin{thm}\label{thm:kw}
    Let $\sS$ be a symmetric system with a respect-basis consisting of names for elements of $\mathcal{P}^\alpha(\Ord)$ and $V \models \KW^*_\beta$. Then $\Vdash_\sS \KW^*_{\max(\alpha, \beta)}$.
\end{thm}

\begin{proof}
    By Theorem~\ref{thm:V(P/S)}, a symmetric extension by $\sS$ is of the form $V(\mathbb{P}/ \sS)$. The elements of $\mathbb{P}/ \sS$ have the form $(r, \langle \dot x_i , x_i : i < n\rangle)$, where $r \in \mathbb{P}$, $\dot x_i \in \sR$ and $x_i \in \mathcal{P}^\alpha(\Ord)$, for each $i < n$ and $\sR \in V$ is a respect basis for $\sS$. Thus $\mathbb{P}/ \sS$ can be viewed as a subset of $\mathbb{P} \times \sR^{<\omega} \times \mathcal{P}^\alpha(\Ord)^{<\omega}$. 
    
    Since $X := \trcl( \{\mathbb{P} \times \sR^{<\omega} \}) \in V$, there is a surjection $ f \colon \mathcal{P}^\beta(\Ord) \to X$ in $V$. Then the set $E = \{ (x,y) \in f^{-1}(X) : f(x) \in f(y) \}$, codes $\mathbb{P} \times \sR^{<\omega}$ as its transitive collapse. Then there is a subset $x$ of $f^{-1}(X) \times \mathcal{P}^\alpha(\Ord)^{<\omega}$ coding $\mathbb{P}/\sS$ using $E$. More specifically, we have that $V(\mathbb{P}/ \sS) = V(E, x)$. Altogether $(E,x)$ may itself be viewed as a subset of $\mathcal{P}^{\max(\alpha, \beta)}(\Ord)$ and item (5) of Proposition~\ref{prop:KWfacts} finishes the proof.
\end{proof}

\begin{cor}
    $(\ZFC)$ Let $\sS = (\mathbb{P}, \sG, \sF)$ and $\vert \mathbb{P} \cup \sG \vert = \kappa$. Then $\Vdash_\sS \KW^*_{\kappa^+}$.
\end{cor}

\begin{quest}
    Is the above corollary optimal? Namely, if $\vert\mathbb{P}\cup\sG\vert = \kappa$, is there an ordinal $\alpha<\kappa^+$ such that $\Vdash_\sS\KW^*_\alpha$?
\end{quest}

\begin{quest}
    Consider $\Add(\omega, 1)$, i.e. Cohen forcing, and assume $\AC$. There are in principle $2^{\aleph_0}$ many automorphisms of this partial order, so any system utilising it forces at least $\KWP^*_{(2^{\aleph_0})^+}$. Is this optimal? Do we always get $\KWP^*_{\omega_1}$, or even $\KWP^*_{\alpha}$ for some $\alpha < \omega_1$?
\end{quest}

\begin{thm}
    Assume that $\Vdash_\sS \KW^*_\alpha$. Then the collection of $\sS$-names for elements of $\mathcal{P}^{\alpha+1}(\Ord)$ forms a respect basis for $\sS$. Moreover, there is $\sT$ weakly equivalent to $\sS$ so that the collection of $\sT$-names for elements of $\mathcal{P}^{\alpha}(\Ord)$ forms a respect basis for $\sT$. If $\AC$ holds we can assume that $\sT \cong_\SN \sS$. 
\end{thm}

\begin{proof}
 If $M$ is an $\sS$-extension of $V$, then $V \subseteq V(\mathcal{P}^{\alpha +1}(\Ord)^M) \subseteq M$ and by Proposition~\ref{prop:bvm}, in fact $V(\mathcal{P}^{\alpha +1}(\Ord)^M) = M$. Following essentially the argument of Proposition~\ref{prop:respectbasefromHOD}, we find that the names for elements of $\mathcal{P}^{\alpha +1}(\Ord)$ form a respect basis for $\sS$. According to \cite[Theorem 4.6]{KaragilaSchilhan2024}, $M = V(x)$ for some $x \subseteq \mathcal{P}(\Ord)$. Assuming $\AC$, there is some $\hat{\sS}$-name $\dot x$ for a subset of $\mathcal{P}^\alpha(\Ord)$, so that $\Vdash_{\hat{\sS}} \HS_{\hat{\sS}}^\bullet = V(\dot x)$ (see Lemma~\ref{lem:maximalitycompletion}). Then simply consider the system $\sT \cong_\SN \hat{\sS}$, where we restrict the group of $\hat{\sS}$ to $\sym_{\hat{\sS}}(\dot x)$ and follow Proposition~\ref{prop:respectbasefromHOD}. If $\AC$ does not hold, there is still a dense set of conditions $p$ with names $\dot x$ so that $p$ forces the above. We may obtain a system by restricting $\hat{\sS}$ to $p$ and $\sym_{\hat{\sS}}(\dot x)$. Letting $\sT$ be the lottery sum of all these systems we find that $\sT$ is weakly equivalent to $\sS$.\footnote{The lottery sum $\bigoplus_i \sS_i$ of $\sS_i = (\mathbb{P}_i, \sG_i, \sF_i)$ can be naturally defined on $\mathbb{P} = \bigoplus_i \mathbb{P}_i$ using $\sG = \prod_i \sG_i$ and $\sF$ generated by the $\sF_i$.} 
\end{proof}

\section{Genericity over \texorpdfstring{$\HOD$}{HOD}}

Recall the following well-known theorem of Vop\v{e}nka originating in \cite{HajekVopenka}. 

\begin{prop}[see, e.g., {\cite[Theorem 15.46]{Jech2003}}]\label{prop:genoverHOD}
 Let $x \in V$ be a set of ordinals. Then $x$ is generic over $\HOD^V$.
\end{prop}

Our final result is the natural generalisation of this to symmetric extensions and arbitrary sets $x$.

\begin{thm}\label{thm:symoverHOD}
    Let $x \in V$ be arbitrary and $M = \HOD^V$. Then $M(x)$ is a symmetric extension of $M$. In other words, $x$ is an element of a generic extension of $M$.
\end{thm}

If $V$ satisfies $\AC$, then this is a trivial consequence of Proposition~\ref{prop:genoverHOD} and Theorem~\ref{thm:characterise}, as there is a set of ordinals in $V$ coding $x$. Otherwise, one would like to approach this by first passing to a generic extension $V[G]$ in which $x$ can be coded into a set of ordinals. But then $\HOD^{V[G]}$ might be strictly smaller than $\HOD^V$ and it is not clear how to proceed. In \cite[9.3 Theorem 1]{Grigorieff1975}, Grigorieff does exactly this, proving Theorem~\ref{thm:symoverHOD} in the special case that $V = M(a)$ for some set $a$. We provide a direct argument which is a modification of the usual proof of Proposition~\ref{prop:genoverHOD} above. Interestingly, the notion of symmetric genericity becomes quite relevant.

\begin{proof}
    Without loss of generality we can assume that $x = V_\alpha^{M(x)}$ for some ordinal $\alpha$. Then there is a large enough ordinal $\gamma$ so that every $\OD$ subset of $(V_\alpha)^{<\omega}$ is in fact $\OD$-definable over $(V_\gamma, \in)$. Consider the forcing $\mathbb{P}$ consisting of conditions $p = (a, \varphi)$, where $a \in [\omega]^{<\omega}$, $\varphi$ is a formula with exactly $\vert a \vert$-many free variables $v_n$, $n \in a$, and ordinals $< \gamma$ as the only parameters, and there is $\bar x \in (V_\alpha)^a$ such that $V_\gamma \models \varphi(\bar x)$.\footnote{When we write ``formula'' here, we of course talk about $V$'s formalisation of first order logic and not some meta-theoretic syntactic object.} $(b, \psi)$ extends $(a, \varphi)$ iff $a \subseteq b$ and \[V_\gamma \models \forall \bar v \in (V_\alpha)^b (\psi(\bar v) \implies \varphi(\bar v \restriction a)).\] Then note that $\mathbb{P} \in \HOD$. Let $\sG$ be the group of finitary permutations on $\omega$. A permutation $\pi \in \sG$ then induces the action $\pi(a,\varphi) = (b,\psi),$ where $b = \pi``a$ and $\psi$ is $\varphi$ with each variable $v_n$ substituted by $v_{\pi(n)}$. Finally, let $\sF$ be the filter generated by groups $\fix(e) = \{ \pi \in \sG : \forall n \in e (\pi(n) = n) \}$, for $e \in [\omega]^{<\omega}$.

Then $\sS = (\mathbb{P}, \sG, \sF)\in\HOD$ is a symmetric system. Let $\bar x = \langle x_n : n \in \omega \rangle$ be any enumeration of $V_\alpha$ where each element appears infinitely often, say in a forcing extension of $V$ where $V_\alpha$ is countable. 

We claim that the set $G = \{ (a, \varphi) \in \mathbb{P} : V_\gamma \models \varphi(\bar x \restriction a)\}$ is an $\sS$-generic filter over $M=\HOD$. Namely, let $D \in \HOD$ be an $\sS$-symmetrically dense subset of $\mathbb{P}$. More specifically, let us assume that $e \in [\omega]^{<\omega}$ is such that for each $\pi \in \fix(e)$, $\pi``D = D$. The set \[Y = \{ \bar y \in (V_\alpha)^{e} : \exists a\in[\omega]^{<\omega}, \bar z \in (V_\alpha)^a, (e \cup a, \varphi) \in D (V_\gamma \models \varphi(\bar y, \bar z) ) \}\] is $\OD$ and so is its complement $Y' = (V_\alpha)^{e} \setminus Y$. If $Y'$ were non-empty then by choice of $\gamma$, there is a condition $(e, \chi) \in \mathbb{P}$ so that that $Y' = \{ \bar y \in (V_\alpha)^{e} : V_\gamma \models \chi(\bar y) \}$. But then $D$ is not dense, as $(e, \chi)$ cannot be extended by any member of $D$. Thus $Y' = \emptyset$, $Y = (V_\alpha)^{e}$, and in particular, $\bar x \restriction e \in Y$. So let $a$, $\bar z$, $\varphi$ be so that $(e \cup a, \varphi) \in D$ and $V_\gamma \models \varphi(\bar x \restriction e, \bar z)$. Since $\bar x$ enumerates each element of $V_\alpha$ infinitely often, there is some $\pi \in \fix(e)$, so that $\bar x \restriction e \cup \{ (\pi(i), z_i) : i \in a\} = \bar x \restriction (e \cup \pi``a)$. It follows that $\pi(e\cup a, \varphi) \in G \cap D$.

Finally, let us show that $V_\alpha \in M[G]_\sS$. Towards this end, for each $n \in \omega$ define $\dot x_{n, 0} = \emptyset$ and by recursion on $\beta\leq \alpha$, \[ \dot x_{n, \beta} = \{ \langle (a, \varphi), \dot x_{m,\beta'}\rangle : \beta' < \beta, n,m\in a, V_\gamma \models \forall \bar v \in (V_\alpha)^a ( \varphi(\bar v) \rightarrow v_m \cap V_{\beta'} \in v_n)\}. \]
By induction, we note that $\pi(\dot x_{n,\beta}) = \dot x_{\pi(n), \beta}$ and thus $\fix(\{ n \}) \leq \sym(\dot x_{n,\beta})$ and $\dot x_{n, \beta} \in \HS$. Also, it is straightforward to check by induction that $\dot x_{n, \beta}^G = x_n \cap V_\beta$. To finish the proof, consider $\dot X = \{ \dot x_{n, \alpha} : n \in \omega \} \in \HS$ and obtain that $\dot X^G = V_\alpha$.
\end{proof}

According to our results on quotients, in the previous proof there will be some particular $\bar x$ so that $G$ is $\mathbb{P}$-generic over $M$, but it is not entirely clear how to obtain it directly. This makes for a very interesting application of symmetric genericity.

\bibliographystyle{amsplain}
\bibliography{ref}

\end{document}